\documentclass[a4paper, 11pt]{article}
\usepackage{mypackage}
\usepackage[en-GB]{datetime2}\DTMlangsetup[en-GB]{ord=raise}
\usepackage{ragged2e}
\usepackage{csquotes}
\usepackage{bm}

\usepackage{blindtext}
\usepackage{caption}
\usepackage[scr=boondoxo, scrscaled=.98]{mathalfa}

\newcommand{\mfn}{\mathfrak{N}}
\newcommand{\mcb}{\mathcal{B}}
\newcommand{\mcf}{\mathcal{F}}
\newcommand{\mcg}{\mathcal{G}}
\newcommand{\mcl}{\mathcal{L}}
\newcommand{\mcn}{\mathcal{N}}

\newcommand{\mcx}{\mathcal{X}}

\newcommand{\hmcb}{\hat{\mcb}}
\newcommand{\mvert}{\,|\,}

\newcommand{\tQ}{\tilde{Q}}
\newcommand{\atwo}{\langle2\rangle}
\newcommand{\supp}{\mathrm{supp}}
\newcommand{\Tau}{\mathrm{T}}
\newcommand{\Eta}{\mathrm{H}}

\newcommand{\B}{\mathbb{B}}
\newcommand{\G}{\mathbb{G}}
\newcommand{\X}{\mathbb{X}}
\newcommand{\Y}{\mathbb{Y}}
\newcommand{\hbit}{\hspace*{1.5pt}}
\newcommand{\tsum}{{\textstyle \sum}}
\newcommand{\bigabs}[1]{\bigl|#1\bigr|}
\newcommand{\Bigabs}[1]{\Bigl|#1\Bigr|}
\newcommand{\biggabs}[1]{\biggl|#1\biggr|}

\newcommand{\eps}{\varepsilon}
\newcommand{\gibbs}{\mathrm{Gibbs}}
\newcommand{\pop}{\mathrm{Pop}}
\newcommand{\eqinlaw}{\overset{\text{d}}{=}}
\newcommand{\inlawto}{\overset{\text{d}}{\longrightarrow}}

\newcommand{\bignorm}[1]{\bigl\| #1 \bigr\|}

\newcommand{\runder}{r_*}%{\underline{r}}
\newcommand{\rover}{\hat{r}_*}%{\overline{r}}
\newcommand{\rstarq}{r_*,q}

\xdefinecolor{dgreen}{rgb}{0,0.5,0}
\xdefinecolor{dgreenish}{rgb}{0,0.5,0.2}
\xdefinecolor{dred}{rgb}{0.65,0,0}
\xdefinecolor{dpurple}{rgb}{0.6,0,0.6}

\begin{document}

%\tableofcontents
%\thispagestyle{empty}
%\newpage
%\setcounter{page}{1}

\title{\hspace*{2mm}\\[-16mm] Stein's Method for Spatial Random Graphs}

\author{Dominic Schuhmacher and Leoni Carla Wirth\footnote{Supported by Deutsche Forschungsgemeinschaft GRK 2088.}\\[2mm]
Institute for Mathematical Stochastics\\
University of G\"ottingen
}

\maketitle

\begin{abstract}
  In this article, we derive Stein's method for approximating a spatial random graph by a generalised random geometric graph, which has vertices given by a finite Gibbs point process and edges based on a general connection function. Our main theorems provide explicit upper bounds for integral probability metrics and, at improved rates, a recently introduced Wasserstein metric for random graph distributions. The bounds are in terms of a vertex error term based on the Papangelou kernels of the vertex point processes and two edge error terms based on conditional edge probabilities. In addition to providing new tools for spatial random graphs along the way, such as a graph-based Georgii--Nguyen--Zessin formula, we also give applications of our bounds to the percolation graph of large balls in a Boolean model and to discretising a generalised random geometric graph.
\end{abstract}

\section{Introduction}

Spatial random graphs provide an important framework for the analysis of relations and interactions in networks. Their underlying spatial structure makes them particularly attractive, on the one hand because they are well-suited for modelling real-world systems, such as transportation networks or social networks, on the other hand because they give rise to many interesting theoretical problems, such as percolation and coverage.

Arguably the best known and most widely studied spatial random graph is the \textit{random geometric graph} introduced by \textcite{gilbert1961} in the framework of communication network modelling. A random geometric graph consists of vertices given by a Poisson point process and edges that connect two vertices whenever their spatial distance is smaller than some threshold $r\in\R_+$. The (spatially) dependent edges in spite of an overall simple construction have lead to a wealth of interesting results and applications; 
see \textcite{penrose2003}.

The simple construction also allows to express many properties of the random geometric graph in terms of U-statistics of a Poisson point process. For example, \textcite{reitzner2013} and \textcite{reitzner2017} obtain general Central Limit Theorems including convergence rates, for such properties of random geometric graphs. Their work builds on the Malliavin--Stein approach by \textcite{nourdin2009}.

To better capture the complex nature of real-world networks, more general spatial random graphs have been studied. In this context the asymptotic behaviour of (properties of) these spatial graphs is of particular interest. For example,  \textcite{penrose2016} studies the connectivity of \emph{soft random geometric graphs} as the number of vertices tends to infinity. Soft random geometric graphs generalise random geometric graphs by creating edges according to probabilities supplied by a more or less general connection function that depends on the locations of the two vertices to be connected. 
Another class of spatial random graphs are the \emph{spatial inhomogeneous random graphs} studied in \textcite{vandh2021}. The authors consider graphs of fixed size, assigning a uniformly chosen location and additional random weights to each vertex and connecting two vertices according to a connection function that depends on their spatial distance and their weights. Then local properties and limits of these graphs in the sense of local weak convergence are investigated. 
However, both studies focus on the influence of the vertex locations on the edge structure and properties thereof and do not take into account the (change in) vertex locations themselves. 

In spite of these works, little is known about the rates for the graph convergence itself and corresponding approximation results. An interesting contribution in this direction for non-spatial random graphs was made by \textcite{ReinertRoss2019}, who studied the convergence of exponential random graphs by viewing them as a vector of (dependent) Bernoulli random variables. By an application of Stein's method, the authors obtain explicit bounds on the approximation error between an exponential random graph and an Erd\H{o}s--R\'enyi graph.

In this paper, we obtain approximation results and corresponding convergence rates for spatial random graphs. Our results are based on deriving Stein’s method for approximation by \emph{generalised random geometric graphs}, whose vertices are given by a Gibbs process and whose edges are based on a general connection function depending on the locations of the vertices. This graph model includes the abovementioned soft random geometric graphs. We propose a novel approach by viewing a graph as a pair of point processes. This enables us to focus on the interplay of the geometry and the edge structure of a graph. Furthermore, it allows us to build upon recent research on Stein’s method for Poisson point processes, see \textcite{sx2008} and Gibbs processes, see \textcite{stucki2014}. 

In the special case of two soft random geometric graphs $G_1,G_2$ on some compact space $\mcx\subset\R^d$ with vertex intensity functions $\lambda_1,\lambda_2:\mcx\to\R_+$ and edge connection functions $\kappa_1,\kappa_2:\mcx\times\mcx\to [0,1]$, a direct application of our main Theorem~\ref{thm: general approx result} yields for any real-valued measurable function $f$ on the graph space $\G$ (see Section~\ref{sec:setup})
\begin{align*} 
    \bigl| \E\bigl(f(G_1)\bigr)-\E\bigl(f(G_2)\bigr)\bigr| &\leq 2\norm{f}_{\infty} \int_{\mcx}\bigl|\lambda_1(x)-\lambda_2(x)\bigr| \, \diff x  \notag\\
    &\hspace*{4mm} {} + \norm{f}_{\infty} \int_{\mcx}\int_{\mcx}  \bigabs{\kappa_1(x,y)-\kappa_2(x,y)} \,\lambda_1(x)\,\lambda_2(y) \, \diff x \, \diff y.
\end{align*}
%$B^*=1$ follows from Remark~\ref{re: special cases coupling time}, edge term follows with GNZ.
This result implies a general upper bound on the distance between the distributions of two soft random geometric graphs in terms of any integral probability metric. Considering, more specifically, a suitable Wasserstein metric, we can improve the resulting convergence rates further, %of order $\log^+(\Lambda)/\Lambda$ for $\Lambda = \max\{\int_{\mcx} \lambda_1(x) \text{Leb}^d(\diff x), \int_{\mcx} \lambda_2(x) \text{Leb}^d(\diff x)\}$
see Theorem~\ref{thm: Poisson approx in Wasserstein}.

The plan of the paper is as follows. In Section~\ref{sec:setup}, we introduce the necessary notation and definitions, in particular we briefly discuss Gibbs processes and the GOSPA metric for (deterministic) spatial graphs. Furthermore, we define what we mean by a generalised random geometric graph and study a related graph birth-and-death process. Section~\ref{sec: gnz for srg} discusses an important tool of our proofs, the Georgii-Nguyen-Zessin formula for spatial random graphs, and illustrates its utility in two applications. In Section~\ref{sec: Stein method}, we turn to Stein's method. After a brief overview of the basic ideas, we start by constructing a coupling of two generalised random geometric graphs in Section~\ref{ch: coupling}. This coupling is one of the main ingredients to develop Stein’s method for spatial random graph distributions and is utilised in Section~\ref{ssec: bounds for general IPM} to obtain upper bounds on the distance between two spatial random graph distributions in terms of a general integral probability metric. In Section~\ref{ssec: bounding stein factors}, we choose specifically the Wasserstein metric w.r.t.\ the GOSPA metric in order to improve the Stein factors in the upper bound.
Finally, Section~\ref{sec: applications} provides two applications, giving asymptotic results and convergence rates for the percolation graph of large balls in the Boolean model and for the discretisation of a generalised random geometric graph.

\section{Setup and Notation}  \label{sec:setup}

Let $\mcx\subset \R^d$ be compact and denote by $\mathfrak{N}=\mathfrak{N}(\mcx)$ the space of finite counting measures on $\mcx$ that are \emph{simple} (i.e.\ do not have multi-points). We identify $\xi \in \mathfrak{N}$ with its support $\supp(\xi)$ and make the convention that set notation of the form $\{x_1, \ldots, x_k\}$ shall always imply that $x_1, \ldots, x_k \in \mcx$ are pairwise different except where explicitly stated otherwise. For $A \subset \mcx$ and $x \in \mcx \setminus A$ denote by $A^{\atwo} = \{ \{y,z\} \subset A \}$ the set of potential edges within $A$ and by $\langle A,x \rangle = \{ \{y,x\} \mvert y \in A \}$ the set of potential edges between $A$ and $x$. We make ample use of this notation in particular  for $A=\mcx$ and $A=\xi$.

Denote by $\mfn_2 = \mathfrak{N}(\mcx^{\langle2\rangle})$ the space of finite simple counting measures on the set of two-element subsets of $\mcx$ (equipped with the $\sigma$-algebra obtained by identifying $\mcx^{\langle2\rangle}$ with $\{(x,y) \in \mcx^2 \, | \, x < y\}$ and taking the trace $\sigma$-algebra from $\mcx^2$). On $\mfn$ and $\mfn_2$ we use the usual Borel $\sigma$-algebras with respect to their respective vague topologies (which are also the smallest $\sigma$-algebras that render point counts measurable) and denote them by $\mcn$ and $\mcn_2$, respectively.

We call a probability kernel $Q$ from $\mfn$ to $\mcn_2$ an \emph{edge kernel} if the measure $Q(\xi,\cdot)$ is concentrated on $\mfn_2 \vert_{\xi^{\langle2\rangle}} = \{\sigma \in \mfn_2 \mvert \sigma((\xi^{\langle2\rangle})^c) = 0\}$ for any $\xi \in \mfn$. In what follows we tacitly identify this set with the set $\mfn(\xi^{\langle2\rangle})$ unless stated otherwise.

For a probability distribution $P$ on $\mfn$ and an arbitrary edge kernel $Q$, we consider the probability distribution $P \otimes Q$ on $\mfn \times \mfn_2$ given as the kernel product satisfying
\begin{equation*}
  (P \otimes Q)(f) = \int_{\mfn \times \mfn_2} f(\xi,\sigma) \; (P \otimes Q)(d(\xi,\sigma)) = \int_{\mfn} \int_{\mfn_2} f(\xi, \sigma) \; Q(\xi,d\sigma) \;P(d\xi)
\end{equation*}
for any measurable $f \colon \mfn \times \mfn_2 \to \R_+$. For the generic random element of $\mfn \times \mfn_2$ having this distribution, we write $(\Xi,\Sigma) \sim P \otimes Q$. By definition $\Sigma$ is concentrated on $\mfn_2 \vert_{\xi^{\langle2\rangle}}$ given $\Xi=\xi$.

We furthermore write $\mathbb{G} = \{(\xi,\sigma) \mvert \xi \in \mfn, \sigma \in \mfn_2 \vert_{\xi^{\langle2\rangle}}\}$, which we identify with the set of simple graphs with a finite vertex set $\subset \mcx$.

\subsection{Gibbs processes}\label{ch: gibbs}
Let $\alpha$ be a diffuse measure on the compact set $\mcx\subset \R^d$ and denote by $\pop(\alpha)$ the distribution of the Poisson point process with intensity measure $\alpha$. 

A (finite) Gibbs process on $\mcx$ is a point process that has a hereditary density $u:\mfn\to\R_+$ with respect to $\pop(\alpha)$, where we call a function $u:\mfn\to\R_+$ hereditary, if $u(\xi)=0$ implies $u(\eta)=0$ for any $\xi$, $\eta\in\mfn$ with $\xi\subset \eta$. 

The assumption that $\alpha$ is diffuse implies that the corresponding Poisson and Gibbs processes are simple. We remark, that it is possible to relax this condition and refer to \textcite{stucki2014} for a discussion of this case.

An alternative way to describe a Gibbs process is the conditional intensity $\lambda:\mfn\times\mcx\to\R$ which can be derived from the density~$u$ (using the hereditary property) via
$$\lambda(x\mvert \xi) = \frac{u(\xi + \delta_x)}{u(\xi)},$$
where we set $0/0 = 0$. In particular, we can use the above formula to recursively recover the (unnormalised) density from the conditional intensity~$\lambda$.
We denote the distribution of a Gibbs process with conditional intensity $\lambda$ by $\gibbs(\lambda)$ or $P^{\lambda}$ for short. In the special case where $\lambda$ does not depend on the point configuration, i.e.\ $\lambda(\cdot\mvert \xi) = \lambda(\cdot)$, we recover a Poisson point process whose distribution we denote by $\pop(\lambda)$.

In the following we often impose a somewhat relaxed local stability condition 
\begin{equation} \label{eq: stability cond}
    \sup_{\xi\in\mfn} \int_{\mcx} \lambda(x\mvert\xi) \alpha(\diff x) <\infty, \tag{LS}
\end{equation}
to ensure the finiteness of the expected number of points in the Gibbs process. Note that this condition is weaker than the usual local stability condition, i.e.\ $\sup_{\xi \in \mfn} \lambda(x\mvert \xi)\leq \beta(x)$ for some integrable function $\beta:\mcx\to\R_+$.% Furthermore, local weak stability implies Ruelle stability. 

An important subclass of Gibbs processes are the pairwise interaction processes that have a conditional intensity of the form 
\begin{equation} \label{eq: cif of pip}
    \lambda(x\mvert\xi) = \beta(x)\prod_{y\in\xi} \varphi(x,y)
\end{equation}
for an activity function $\beta:\mcx\to \R_+$ and a symmetric interaction function~$\varphi:\mcx\times\mcx\to\R_+$. We denote the distribution of a pairwise interaction process by $\text{PIP}(\varphi,\beta)$.
If $\varphi\leq 1$, we furthermore call the process inhibitory. Note that an inhibitory pairwise interaction process is locally stable (and thus satisfies \eqref{eq: stability cond}) if $\beta$ is integrable.

\subsection{Generalised random geometric graphs} \label{ch: kappa geom RG}

We use the term \emph{generalised random geometric graph} or, more specifically, $\mathrm{RGG}(\lambda,\kappa)$ to refer to a random graph whose vertex process is a finite Gibbs process with conditional intensity $\lambda:\mcx\times\mfn\to\R_+$ and whose edges are created independently of each other (given the vertex set) with connection probabilities from the measurable function $\kappa:\mcx\times\mcx \to [0,1]$. The latter means that we choose the edge kernel $Q^{\kappa}$ given by
$$
  Q^{\kappa}(\xi,\{\sigma\}) =  \prod_{\{x,y\}\in\xi^{\atwo}}\kappa(x,y)^{\sigma(\{x,y\})}(1-\kappa(x,y))^{1-\sigma(\{x,y\})}
$$
for $\sigma \in \mfn_2\vert_{\xi^{\atwo}}$,
so that the distribution of our $\mathrm{RGG}(\lambda,\kappa)$ is ${P^{\lambda} \otimes Q^{\kappa}}$. To set it apart from more general types of random graphs, we typically denote the random element by $(\Eta,\Tau)$ rather than $(\Xi,\Sigma)$.

Note that the creation of an edge depends only on the location of the two vertices to be connected by it. In particular, the process of edges between $\xi \in \mfn$ and $x\in\mcx\setminus\xi$ is independent of the process of edges within $\xi$, i.e.\
\begin{equation} \label{eq:kappa_grg_product_formula_prelim}
  Q^{\kappa}(\xi+\delta_x, \{\sigma_1+\sigma_2\}) = Q^{\kappa}(\xi, \{\sigma_1\}) \, Q_2^{\kappa}(\xi,x;\{\sigma_2\}) \quad \text{ for all $\sigma_1 \in \mfn_2 \vert_{\xi^{\atwo}}$, $\sigma_2 \in \mfn_2\vert_{\langle \xi,x\rangle}$},
\end{equation}
where $Q_2^{\kappa}(\xi,x;\{\sigma_2\}) := \prod_{y\in\xi}\kappa(x,y)^{\sigma_2(\{x,y\})}(1-\kappa(x,y))^{1-\sigma_2(\{x,y\})}$.\\

We introduce what we call a \emph{graph birth-and-death process (GBDP)}, which is a special pure-jump Markov process (PJMP) with state space $\G$ that has ${P^{\lambda} \otimes Q^{\kappa}}$ as its stationary distribution (see Lemma~\ref{le: stationary distribution}). Denote by $q(\xi) = \int_{\mcx}\lambda(x\vert\xi)\alpha(\diff x) + \abs{\xi}$ the jump rate at $(\xi,\sigma) \in \G$. It is independent of $\sigma$ and we assume from now on that it is positive, since $q(\xi) = 0$ is only possible for $\xi = \emptyset$ and implies that $\emptyset$ is absorbing, which is not a very interesting case for this paper as the stationary distribution would be concentrated on the empty graph.

The dynamics of the GBDP are then as follows. Being in a state $(\xi,\sigma)\in\G$, the process stays there for an Exp($q(\xi)$)-distributed time. After this time, a vertex is added (``born'') with probability $(\int_{\mcx}\lambda(x\vert\xi)\alpha(\diff x))/q(\xi)$ and is deleted (``dies'') with probability $\abs{\xi}/q(\xi)$. In the former case the new vertex is positioned at $x\in\mcx$ according to the density $\lambda(x\vert\xi)/(\int_{\mcx}\lambda(x\vert\xi)\alpha(\diff x))$ and we add edges. In the latter case, the vertex $x\in\xi$ to be removed is chosen uniformly among all vertices $\xi$, i.e.\ with probability $1/\abs{\xi}$. At the same time a vertex $x$ is born, we add edges between $x$ and vertices of $\xi$ according to $Q_2^{\kappa}(\xi,x;\cdot)$, i.e.\ independently of each other and of the existing vertices $\sigma$. If a vertex $x$ dies, we immediately delete all of its incident edges.

In summary, these dynamics belong to a PJMP on $\G \subset \mfn \times \mfn_2$ with jump rate function $q$ and
transition kernel $\overline{P}^{\lambda,\kappa}$ of its jump chain given by
\begin{align*} 
    \overline{P}^{\lambda,\kappa}((\xi,\sigma),F) \; = \ &\frac{1}{q(\xi)} \biggl[\int_{\mcx}\int_{\mfn_2} 1_F(\xi+\delta_x,\sigma + \tau) \; Q_2^{\kappa}(\xi,x;\diff \tau) \; \lambda(x\vert\xi) \alpha(\diff x) \\
    &\hspace*{20mm} + \int_{\mcx} 1_F(\xi-\delta_x,\sigma\vert_{(\xi-\delta_x)^{\atwo}}) \xi(\diff x)\biggr] %\1_{\{q(\xi) >0\}} +  1_F(\xi,\sigma )\1_{\{q(\xi) =0\}}
\end{align*}
for $(\xi,\sigma)\in\G$ and $F\in\mcn\otimes\mcn_2$. By Corollary~G.1 of \textcite{moller2004}, this process is non-explosive if the underlying Gibbs process is locally stable. Following the theory of pure-jump Markov processes, see Section~4.2 and Section~4.11 Problem~5 of \textcite{ethier2009}, the generator $\mcg^{\lambda,\kappa}$ is given by 
\begin{alignat}{1} \label{eq: Generator}
    &  \notag\mcg^{\lambda,\kappa} h(\xi,\sigma)\\
    %&:= \hspace{-0.1cm}\notag\int_{\mcx}\hspace{-0.1cm}\int_{\mfn_2}\hspace{-0.2cm} h\bigl(\xi + \delta_x, \sigma + \tau \bigr) - h(\xi,\sigma)  Q_2^{\kappa}(\xi,x;\diff \tau) \lambda(x\vert\xi) \diff x + \hspace{-0.1cm}\int_{\mcx}\hspace{-0.15cm} h\bigl(\xi - \delta_x, \sigma\vert_{(\xi-\delta_x)^{\atwo}}\bigr) - h(\xi,\sigma)\xi(\diff x)\\
    &\ =\int_{\mcx} \E \bigl[h\bigl(\xi + \delta_x, \sigma + \Tau_{\xi,x}\bigr) - h(\xi,\sigma)\bigr] \lambda(x\vert\xi) \alpha(\diff x) +\int_{\mcx} \bigl[h\bigl(\xi - \delta_x, \sigma\vert_{(\xi-\delta_x)^{\atwo}}\bigr) - h(\xi,\sigma)\bigr]\xi(\diff x),
\end{alignat}
where ${\Tau}_{\xi,x} \sim Q_2^{\kappa}(\xi,x;\cdot)$ and $h\in \mathcal{F}'$ the set of bounded functions $h\colon \G \to \R$ such that the right-hand side is well defined. 

\subsection{Metrics for spatial graphs} \label{ch: metric}

We use either of the graph optimal subpattern assignment (GOSPA) metrics~$d_{\mathbb{G},i} := d_{\mathbb{G},R_i}$, $i=1,2$, introduced in \textcite{sw2023} to assess the difference of two deterministic graphs based on underlying metrics $d_V\leq C_V$ and $d_E\leq C_E$ on the vertex and edge space, respectively. 
To make this more precise, consider two graphs $(\xi,\sigma),(\eta,\tau)\in\G$ with 
$\xi = \sum_{i=1}^n \delta_{x_i}$ and $\sigma = \sum_{\substack{i,j=1\\i<j}}^{n} e_{ij}\,\delta_{\{x_i,x_j\}}$ as well as $\eta = \sum_{i=1}^m \delta_{y_i}$ and $\tau = \sum_{\substack{i,j=1\\i<j}}^{m} f_{ij}\,\delta_{\{y_i,y_j\}}$, where $e_{ij}$, $f_{ij}\in\{0,1\}$ indicate whether the edges $\{x_i,x_j\}$ and $\{y_i,y_j\}$ are present in $(\xi,\sigma)$ and $(\eta,\tau)$, respectively. Furthermore, denote by $\tilde{e}_{ij} = (e_{ij},\{x_i,x_j\}^{e_{ij}})$ and $\tilde{f}_{ij} = (f_{ij},\{y_i,y_{j}\}^{f_{ij}})$ the available information for an edge, where we set $A^1=A$ and $A^0 = \emptyset$ for any set $A$.

In the special case where $(\xi,\sigma)$ and $(\eta,\tau)$ have the same size, i.e.\ if $n = m$, the two versions of the GOSPA metric agree and are given by
\begin{align}  \label{eq: dgi same cardinality}
    d_{\mathbb{G},i}((\xi,\sigma),(\eta,\tau)) = \frac{1}{n} \biggl( \min_{\pi\in S_n} \sum_{i\in [n]} d_V(x_i,y_{\pi(i)}) + \frac{1}{n-1} \sum_{\substack{i,i' \in [n] \\ i < i'}} d_E(\tilde{e}_{ii'}, \tilde{f}_{\pi(i),\pi(i')}) \biggr),
\end{align}
where $S_n$ denotes the set of permutations on $[n] = \{1,\ldots,n\}$. %Note that the edge metric can depend on the spatial positions as well, i.e.\ $d_E((e_{i,i'},\{x_i,x_{i'}\}), (f_{\pi(i),\pi(i')},\{y_{\pi(i)},y_{\pi(i')}\}))$. 

In the general case, where the two compared graphs can be of different size, we keep the above structure and match all vertices of the smaller graph with (some of) the vertices in the larger graph. Additionally, we add a penalty for every vertex (and its corresponding edges) in the larger graph that cannot be matched. For the GOSPA1 metric, this penalty only depends on the number of possible edges, %is constant
while for the GOSPA2 metric the penalty depends on the actual edge structure of the respective unmatched vertices. In particular, if the sizes of the two graphs differ only by one, the penalty for the additional vertex $x$, say, is $C_V + C_E =: \tilde{C}_1$ in the GOSPA1 metric and $C_V + C_E + \frac{1}{n-1} \text{deg}(x) \,C_E \leq C_V + 2C_E =: \tilde{C}_2$ in the GOSPA2 metric (not including the overall normalisation by $1/n$). The additional ${} + C_E$ for the GOSPA2 metric is explained in \textcite[Remark~2.9]{sw2023}. In the general case, where we do not restrict to two graphs with sizes differing only by one, the choice of the penalties and the boundedness of the underlying metrics $d_V$ and $d_E$ yield that the GOSPA metrics are bounded by $C_1 = C_V + \frac12C_E$ and $C_2 = C_V + C_E$, respectively, see \textcite[Proposition~2.14]{sw2023} for details.

To quantify distances between random graph distributions, we introduce the Wasserstein metric $W_{\G}$ (of order 1) with respect to either of the two GOSPA metrics~$d_{\mathbb{G}} = d_{\mathbb{G},i}$, $i=1,2$. Let
$$\mathcal{M}_{\G} = \biggl\{ P \text{ probability measure on } \G \biggm\vert \int d_{\G}((\xi,\sigma), \boldsymbol{0}) \, P(d(\xi,\sigma)) < \infty \biggr\},$$
where $\boldsymbol{0}$ denotes the empty graph.

\begin{definition} \label{def: wasserstein G}
The \emph{Wasserstein metric} $W_{\mathbb{G}}$ on $\mathcal{M}_{1}(\G)$ based on $d_{\mathbb{G}}$ is defined as
$$W_{\G}(P,Q) = \inf_{\substack{(\Xi,\Sigma) \sim P \\ (\Eta,\Tau) \sim Q}} \E \, d_{\mathbb{G}} \bigl( (\Xi,\Sigma), (\Eta,\Tau) \bigr) = \sup_{f \in \mcf_{\G}} \, \biggl| \int_{\mathbb{G}} f \, dP - \int_{\mathbb{G}} f \, dQ \biggr|, \quad P,Q \in \mathcal{M}_{\G},$$
%$$W_{\mathbb{G}}(P,Q) = \sup_{f\in\mathcal{F}_{\G}} \bigg \vert \int f \diff P - \int f \diff Q \bigg\vert, \quad P,Q \in \mathcal{M}_1(\G) $$
%where $\mathcal{F}_{\G} = \bigl\{ f:\mathbb{G} \to \R \bigm| \abs{f(\xi,\sigma)-f(\eta,\tau)} \leq d_{\mathbb{G}}((\xi,\sigma),(\eta,\tau)) \bigr\}$. 
where $\mcf_{\G} = \bigl\{f \colon \G \to \R \bigm| \abs{f(\xi,\sigma) - f(\tau,\eta)} \leq d_{\G}((\xi,\sigma),(\tau,\eta)) \bigr\}$ and the second equality is due to the Kantorovich--Rubinstein theorem.
\end{definition}

Then $W_{\mathbb{G}}(P_n,P) \to 0$ is equivalent to $P_n \to P$ weakly and $\int d_{\G}((\xi,\sigma), \boldsymbol{0})\, P_n(d(\xi,\sigma)) \to \int d_{\G}((\xi,\sigma), \boldsymbol{0}) \, P(d(\xi,\sigma))$ for both underlying metrics; see \textcite[Theorem~6.9]{villani2009}. By the fact that the two GOSPA metrics $d_{\mathbb{G},i}$, $i=1,2$ are bounded, we obtain that the corresponding Wasserstein metrics metrize weak convergence.

We write $\mathcal{M}_{\G,i}$, $\mcf_{\G,i}$ and $W_{\G,i}$ for $i=1,2$ rather than $\mathcal{M}_{\G}$, $\mcf_{\G}$ and $W_{\G}$ for the objects defined above whenever the type of the underlying GOSPA metric is important.

\section{Georgii--Nguyen--Zessin Formula for Spatial Random Graphs} \label{sec: gnz for srg}

We introduce a variant of the GNZ formula for spatial random graphs for a simple (i.e.\ multi-point free) point process $\Xi$ with distribution $P$ on a general localised Borel space $(\mcx, \hmcb, \mcb)$. Here $\mcb$ denotes the $\sigma$-algebra on $\mcx$ and $\hmcb \subset \mcb$ a localising ring. See \textcite[Chapter 1]{Kallenberg2017rm} for the general concept. If $\mcx$ is a complete separable metric space (c.s.m.s.), we may choose for $\mcb$ the Borel $\sigma$-algebra and for $\hmcb$ the system of bounded sets in $\mcb$. 

Assume that $\Xi$ satisfies Condition \eqref{eq:condsigma}, which states that
\begin{equation}  \label{eq:condsigma}
    \Prob \bigl( \Xi(B) = 0 \bigm| \Xi\vert_{B^c} \bigr) > 0 \quad \text{a.s.} \tag{$\Sigma$}
\end{equation}
for any $B \in \hmcb$.

There is then a $P$-a.s.\ unique probability kernel $L$ from $\mfn=\mfn(\mcx)$ to $\mcx$, called the \emph{Papangelou kernel} of $\Xi$, satisfying for every measurable $h \colon \mfn \times \mcx \to \R_+$
\begin{equation}  \label{eq:gnz}
  \E \biggl( \int_{\mcx} h(\Xi-\delta_x,x) \; \Xi(dx) \biggr)
  = \E \biggl( \int_{\mcx} h(\Xi,x) \; L(dx \mvert \Xi) \biggr).
\end{equation}
This equation is often referred to as the \emph{Georgii--Nguyen--Zessin (GNZ) equation}. %, in particular if the kernel $L$ has a density $[(\xi,x) \mapsto \lambda(x \mvert \xi)]$, which is then called the \emph{(Papangelou) conditional intensity} of $\Xi$. 
Since we have assumed $\Xi$ to be simple, we obtain $L(\supp(\xi) \mvert \xi) = 0$ for $P$-a.e.\ $\xi$. For this and further results, see \textcite[Chapter 8]{Kallenberg2017rm}.

To state the GNZ formula for spatial random graphs in Theorem~\ref{thm:graph_gnz} below, some more notation is required. For $\xi \in \mfn$, $x \in \mcx \setminus \xi$ and the measurable map $\varphi_x \colon \mfn_2 \to \mfn_2 \times \mfn_2$, $\sigma \mapsto (\sigma\vert_{(\mcx\setminus\{x\})^{\atwo}},\sigma\vert_{\langle \mcx\setminus\{x\}, x \rangle})$, define the probability measure
\begin{equation*}
    \tQ(\xi,x;\, \cdot) = Q(\xi+\delta_x;\ \cdot) \circ \varphi_x^{-1}
\end{equation*}
on $\mfn_2 \times \mfn_2$, which is concentrated on $\mfn_2 \vert_{\xi^{\atwo}} \times \mfn_2 \vert_{\langle \xi,x \rangle}$. Thus $\tQ(\xi,x;\, \cdot)$ represents the distribution of the edge process on the vertex set $\xi+\delta_x$ as a joint distribution of the edge process on the vertex set $\xi$ and the process of edges between $\xi$ and $x$. We disintegrate the measure $\tQ(\xi,x;\, \cdot)$ with respect to its first marginal $\tilde{Q}_1(\xi,x;\, \cdot) = \tQ\bigl(\xi,x;\, \cdot \times \mfn_2 \bigr)$ as
\begin{equation}  \label{eq:ekernel_disintegration}
    \tQ(\xi,x;\, d(\sigma_1,\sigma_2)) = \tilde{Q}_1(\xi, x;\, d\sigma_1) \otimes \tilde{Q}_{2 \mvert 1}(\xi, x, \sigma_1;\, d\sigma_2);
\end{equation}
see \textcite[Theorem 3.4]{Kallenberg2021fomp}. 

%Reminder for us: In Theorem~\ref{thm:graph_gnz} values of $h$ are only used at $(\xi,\sigma_1,x,\sigma_2)$ with $x \not\in \xi$, $\sigma_1 \in \mfn(\xi^{\atwo})$ and $\sigma_2 \in \mfn(\langle\xi,x\rangle)$) 
\begin{theorem} \label{thm:graph_gnz}
  For any measurable function $h:\mfn \times \mfn_2 \times \mcx \times \mfn_2 \to \R_+$ we have
  \vspace*{-5mm}
  % irgendwas mit den Abständen läuft schief; hier v.a. aber vor und insbesondere nach Theorem-Umgebungen sind die zu klein; ich lös das mal brutal.
  
  \begin{equation*}
  \begin{split}
    &\int_{\mfn \times \mfn_2} \int_{\mcx} h \bigl(\xi-\delta_x, \sigma \vert_{(\xi-\delta_x)^{\atwo}}, x, \sigma \vert_{\langle \xi-\delta_x,x \rangle} \bigr) \; \xi(dx) \; (P \otimes Q)(d(\xi,\sigma)) \\
    &= \int_{\mfn \times \mfn_2} \int_{\mcx} \int_{\mfn_2} h\bigl(\xi, \sigma, x, \sigma_2 \bigr) \; \tilde{Q}_{2 \mvert 1}(\xi,x,\sigma ;\, d\sigma_2) \; L(\diff x \mvert \xi) \; (P \otimes Q)(d(\xi,\sigma))  \\
    &\hspace*{3mm} {} + \int_{\mfn} \int_{\mcx} \int_{\mfn_2} \int_{\mfn_2}\! h\bigl(\xi, \sigma_1, x, \sigma_2 \bigr) \; \tilde{Q}_{2 \mvert 1}(\xi,x,\sigma_1 ;\, d\sigma_2) \; \bigl( \tilde{Q}_1(\xi, x;\, d\sigma_1) - Q(\xi, d\sigma_1) \bigr) \; L(dx \mvert \xi) \; P(d\xi).
  \end{split}
  \end{equation*}
\end{theorem}
\vspace*{0mm}
%Maybe add a remark here to the effect that the idea underlying the proof is more universal than random graphs and would work similarly with kernels from $\mfn$ to any Borel space? And almost exactly this way if one has a ``2-marked'' process, i.e.\ a point process with marks for pairs of points in a Borel mark space.

\begin{proof}%[Proof of Theorem~\ref{thm:graph_gnz}]
  By Tonelli's theorem and the transformation theorem, noting that $\sigma \vert_{(\xi-\delta_x)^{\atwo}} \!= \sigma \vert_{(\mcx \setminus \{x\})^{\atwo}}$ and $\sigma \vert_{\langle \xi-\delta_x,x \rangle} = \sigma \vert_{\langle \mcx \setminus \{x\},x \rangle}$ $Q(\xi,\cdot)$-a.s.\ for any $\xi \in \mfn$ and $x \in \xi$, we have
  \begin{equation*}
  \begin{split}
    &\int_{\mfn \times \mfn_2} \int_{\mcx} h \bigl(\xi-\delta_x, \sigma \vert_{(\xi-\delta_x)^{\atwo}}, x, \sigma \vert_{\langle \xi-\delta_x,x \rangle} \bigr) \; \xi(dx) \; (P \otimes Q)(d(\xi,\sigma)) \\
    &=\int_{\mfn} \int_{\mcx} \int_{\mfn_2} h \bigl(\xi-\delta_x, \sigma \vert_{(\xi-\delta_x)^{\atwo}}, x, \sigma \vert_{\langle \xi-\delta_x,x \rangle} \bigr) \; Q(\xi, d\sigma) \; \xi(dx) \; P(d\xi) \\
    &= \int_{\mfn} \int_{\mcx} \int_{\mfn_2 \times \mfn_2} h \bigl(\xi-\delta_x, \sigma_1, x, \sigma_2 \bigr)  \;\tilde{Q}(\xi-\delta_x,x;\, d(\sigma_1,\sigma_2)) \; \xi(dx) \; P(d\xi).
  \end{split}
  \end{equation*}
  Setting $\tilde{h}(\xi,x) = \int_{\mfn_2 \times \mfn_2} h \bigl(\xi, \sigma_1, x, \sigma_2 \bigr)  \;\tilde{Q}(\xi,x;\, d(\sigma_1,\sigma_2))$, we obtain a measurable function $\tilde{h} \colon \mfn \times \mcx \to \R_{+}$. Applying the usual GNZ formula \eqref{eq:gnz} for point processes to this function yields that the right hand side above is equal to
  \begin{equation*}
  \begin{split}
    &\int_{\mfn} \int_{\mcx} \int_{\mfn_2 \times \mfn_2} h\bigl(\xi, \sigma_1, x, \sigma_2 \bigr) \;\tilde{Q}(\xi,x;\, d(\sigma_1,\sigma_2)) \; L(dx \mvert \xi) \; P(d\xi) \\
    &= \int_{\mfn} \int_{\mcx} \int_{\mfn_2} \int_{\mfn_2} h\bigl(\xi, \sigma, x, \sigma_2 \bigr) \; \tilde{Q}_{2 \mvert 1}(\xi,x,\sigma ;\, d\sigma_2) \; Q(\xi,d\sigma) \; L(\diff x \mvert \xi) \; P(d\xi)  \\
    &\hspace*{3mm} {} + \int_{\mfn} \int_{\mcx} \int_{\mfn_2} \int_{\mfn_2} h\bigl(\xi, \sigma_1, x, \sigma_2 \bigr) \; \tilde{Q}_{2 \mvert 1}(\xi,x,\sigma_1 ;\, d\sigma_2) \; \bigl( \tilde{Q}_1(\xi, x;\, d\sigma_1) - Q(\xi, d\sigma_1) \bigr) \; L(dx \mvert \xi) \; P(d\xi).
  \end{split}
  \end{equation*}
  By Tonelli's theorem and the definition of $P \otimes Q$ we obtain the statement.
\end{proof}

In the context of the original GNZ formula, the second term on the right hand side may be seen as an error term that quantifies the influence of an additional vertex at $x$ on the creation of edges between other vertices. We formulate the special case that there is no such influence in terms of random elements and expectations.
\begin{corollary} \label{cor: GNZwithoutAdditionalInfluence}
  Let $(\Xi, \Sigma) \sim P \otimes Q$ be a random graph with the property that the distribution of the edge process $\Sigma\vert_{\xi^{\atwo}}$ given $\Xi=\xi+\delta_x$ is the same as given $\Xi=\xi$ for $P$-a.e.\ $\xi$.
  Then we have for any measurable function $h:\mfn \times \mfn_2 \times \mcx \times \mfn_2 \to \R_+$
  \begin{equation*}
  \begin{split}
      \E \biggl(\int_{\mcx} &h \bigl(\Xi-\delta_x, \Sigma \vert_{(\Xi-\delta_x)^{\atwo}}, x, \Sigma \vert_{\langle \Xi-\delta_x,x \rangle} \bigr) \; \Xi(dx) \biggr) \\
      &= \E \biggl( \int_{\mcx} \E \Bigl( h \bigl(\Xi, \Sigma, x, \Tilde{\Sigma}_{\Xi,\Sigma,x} \bigr) \Bigm| \Xi, \Sigma \Bigr) \; L(dx \mvert \Xi) \biggr),
  \end{split}
  \end{equation*}
  where $\Tilde{\Sigma}_{\xi,\sigma,x}$ is the random element on $\mfn_2$ with distribution $\mathcal{L} \bigl( \Sigma \vert_{\langle \xi, x \rangle} \bigm| \Xi = \xi+\delta_x, \Sigma \vert_{\xi^{\atwo}} = \sigma \bigr) = \tilde{Q}_{2 \mvert 1}(\xi,x,\sigma ;\, \cdot)$.
  In particular, if $h$ is constant in the last argument,
  \begin{equation*}
  \begin{split}
      \E \biggl( \int_{\mcx} &h \bigl(\Xi-\delta_x, \Sigma \vert_{(\Xi-\delta_x)^{\atwo}}, x\bigr) \; \Xi(dx) \biggr) = \E \biggl( \int_{\mcx} h \bigl(\Xi, \Sigma, x\bigr) \; L(dx \mvert \Xi) \biggr).
  \end{split}
  \end{equation*}
\end{corollary}

We give two examples to illustrate the ``influence assumption'' in Corollary \ref{cor: GNZwithoutAdditionalInfluence} (no influence of an extra vertex on the edge distribution).

\begin{example}
The $k$-nearest neighbour graph on some vertex set $\Xi\subset \mcx$ does not fulfil the influence assumption. To see this, consider for $\xi+\delta_x$ the four corners of a square of which an arbitrary one is labelled $x$. Then the $2$-nearest neighbour graph on $\xi$ given $\Xi = \xi$ is the complete triangle, whereas the $2$-nearest neighbour graph on $\xi$ given $\Xi = \xi+\delta_x$ consists of the edges of the $2$-nearest neighbour graph on $\xi+\delta_x$ (square) that are not adjacent to $x$. Thus it misses the diagonal edge.
\end{example}

\begin{example} \label{ex: kappa-geom rg fulfils influence condition}
The generalised random geometric graph fulfils the influence assumption because edges are created independently with probabilities that depend only on the locations of their own vertices.
%This is because, for $\Tau\sim Q^{\kappa}$ and any $\sigma\in \mfn_2\vert_{\xi^\atwo}$,
%    \begin{align*}
%        &\Prob(\Tau\big\vert_{\xi^{\atwo}} = \sigma\mvert \Eta = \xi + \delta_x)\\
%        &= \sum_{\sigma_2\in \langle\xi,x\rangle} \Prob(\Tau = \sigma + \sigma_2\mvert \Eta = \xi + \delta_x)\\
%        %&= \sum_{\sigma_2\in \langle\xi,x\rangle} \prod_{\{y,z\} \in(\xi+\delta_x)^{\atwo}}\kappa(y,s)^{\sigma(\{y,z\})+ \sigma_2(\{y,z\})}(1-\kappa(y,s))^{1-(\sigma(\{y,z\})+ \sigma_2(\{y,z\}))}\\
%        &= \sum_{\sigma_2\in \langle\xi,x\rangle} \prod_{\{y,z\} \in \xi^{\atwo}}\kappa(y,z)^{\sigma(\{y,z\})}(1-\kappa(y,z))^{1-\sigma(\{y,z\})}\prod_{y\in\xi}\kappa(x,y)^{ \sigma_2(\{x,y\})}(1-\kappa(x,y))^{1-\sigma_2(\{x,y\})} \\
%        %&= \Prob(\Tau\big\vert_{\xi^{\atwo}} = \sigma\mvert \Eta = \xi ) \sum_{\sigma_2\in \langle\xi,x\rangle} \prod_{y\in\xi}\kappa(x,y)^{ \sigma_2(\{x,y\})}(1-\kappa(x,y))^{1-\sigma_2(\{x,y\})} \\
%        &= \Prob(\Tau\big\vert_{\xi^{\atwo}} = \sigma\mvert \Eta = \xi ).
%    \end{align*}
More generally, we may write the decomposition~\eqref{eq:ekernel_disintegration} for the generalised random geometric graph as
$$\tQ(\xi,x;\, d(\sigma_1,\sigma_2)) = \tilde{Q}_1(\xi, x;\, d\sigma_1) \otimes \tilde{Q}_{2 \mvert 1}(\xi, x, \sigma_1;\, d\sigma_2)= Q^{\kappa}(\xi,d\sigma_1) \otimes Q_2^{\kappa}(\xi,x;d\sigma_2)$$
by~\eqref{eq:kappa_grg_product_formula_prelim}.
% More generally, because integrating out $\sigma_2$ reveals $\tilde{Q}_1(\xi, x;\, d\sigma_1) = Q^{\kappa}(\xi,d\sigma_1)$ again.
\end{example}

As an application of the GNZ Formula for spatial random graphs, we show that the Markov process constructed around Equation \eqref{eq: Generator} has the generalised random geometric graph as its stationary distribution. 
\begin{lemma}\label{le: stationary distribution}
The stationary distribution of the Markov process generated by $\mcg^{\lambda,\kappa}$ is $P^{\lambda} \otimes Q^{\kappa}$.
\end{lemma}
\begin{proof}
  Let $(\Eta,\Tau) \sim P^{\lambda} \otimes Q^{\kappa}$ and fix $h\in\mathcal{F}'$. Consider the measurable function $\Tilde{h}:\mfn \times \mfn_2 \times \mcx \times \mfn_2 \to \R$, $\Tilde{h}(\eta,\tau,x,\tau') = h(\eta + \delta_x, \tau + \tau') - h(\eta, \tau)$.
  %$\Tilde{h}\bigl(\Eta-\delta_x, \Tau \vert_{(\Eta-\delta_x)^{\atwo}}, x, \Tau \vert_{\langle \Eta-\delta_x,x \rangle} \bigr) = h \bigl(\Eta, \Tau \bigr) - h\bigl(\Eta - \delta_x, \Tau\vert_{(\Eta-\delta_x)^{\atwo}}\bigr)$.
  Due to Example~\ref{ex: kappa-geom rg fulfils influence condition}, we may apply Corollary \ref{cor: GNZwithoutAdditionalInfluence}
  %and 
  to this function and obtain
\begin{alignat*}{2}
 &\E \biggl(\int_{\mcx} \bigl[ h(\Eta, \Tau) - h\bigl(\Eta - \delta_x, \Tau\vert_{(\Eta-\delta_x)^{\atwo}}\bigr) \bigr] \; \Eta(dx) \biggr) \\
 &= \E \biggl( \int_{\mcx} \E \Bigl( h \bigl(\Eta+\delta_x, \Tau + {\Tau}_{\Eta,x} \bigr) - h \bigl(\Eta, \Tau \bigr)\Bigm| \Eta \Bigr)   \; L(dx \mvert \Eta) \biggr),
\end{alignat*}
where $T_{\Eta,x} \sim Q_2^{\kappa}(\Eta,x; \cdot)$.
As $\Eta\sim \text{Gibbs}(\lambda)$ implies $L(dx \mvert \Eta) = \lambda(x \mvert \Eta) \, dx$, this yields
\begin{alignat*}{1}
    \E \bigl( \mcg^{\lambda,\kappa}h(\Eta,\Tau) \bigr) &=\E\biggl(\int_{\mcx} \E \Bigl(h\bigl(\Eta + \delta_x, \Tau + {\Tau}_{\Eta,x}\bigr) - h(\Eta,\Tau) \Bigm| \Eta \Bigr) \, \lambda(x \mvert \Eta) \alpha(\diff x)\biggr)\\
    & \hspace{1.4cm}- \E\biggl(\int_{\mcx} \bigl[h(\Eta,\Tau) - h\bigl(\Eta - \delta_x, \Tau\vert_{(\Eta-\delta_x)^{\atwo}}\bigr) \bigr]\,\Eta(\diff x)\biggr)=0.
\end{alignat*}
The claim now follows by Proposition~9 of \textcite[p.239]{ethier2009} and we refer to \textcite[Lemma~1]{baddeley2000} for details.
\end{proof}

As a second application, we derive a split-up formula for the edge-structure of a typical vertex in a spatial preferential attachment model as introduced by \textcite{aiello2008} and later generalised by \textcite{jacob2015}. In what follows, we consider the latter model. Let the vertices~$\Xi$ be given by a Poisson point process with intensity measure $\alpha$ on $\mcx:= \T_1\times(0,\infty)$, where $\T_1$ is the one dimensional torus of length~$1$ endowed with the torus metric given by $d(x,y) = \min\{\abs{x-y}, 1-\abs{x-y}\}$ for $x,y \in \mcx$. We interpret the second component as age, i.e.\ we say that $(y,t)$ is \textit{older} than $(x,s)$ if $t<s$, and define $\Xi_t := \Xi\vert_{\T_1\times [0,t]}$ to be the set of vertices already born at time~$t\geq0$. Furthermore, we assign an orientation to the edges by pointing backwards in time from the younger to the older vertex. 
Then we can construct a sequence of graphs $(\Xi_t,\Sigma_t)_{t>0}$ as follows. Given the graph~$(\Xi_{t-},\Sigma_{t-})$ at time $t-$ and a vertex $(y,t)\in\Xi$, we add the vertex $(y,t)$ and independently of other edges (given $(\Xi_{t-},\Sigma_{t-})$), connect the new vertex~$(y,t)$ to any vertex $(x,s)\in \Xi_{t-}$ with probability $\varphi\bigl(\frac{td(x,y)}{f(Z_x(t-))}\bigr)$ dependent on the indegree $Z_{x}(t-)$ of the older vertex~$(x,s)$ at time $t-$ and the torus distance~$d(x,y)$. The functions $\varphi$ and $f$ are referred to as profile function and attachment rule, respectively, and we refer to \textcite{jacob2015} for further assumptions, which assure, among other things, that outdegrees and indegrees are (almost surely) finite (see Lemma~9 and Corollary~13). 

The following corollary provides a formula to assess some statistic $h$ of a typical vertex in the graph %(in the sense of Palm theory, see X)
by treating the vertex as deterministic and splitting up its edges into (conditionally independent) outgoing and incoming edges.
%And  shows that we can split up the edges of a typical vertex into ingoing and outgoing edges and illustrates how Theorem~\ref{thm:graph_gnz} can be applied to directly utilise information about the underlying graph structure.

\begin{corollary}
Let $(\Xi,\Sigma)$ be the graph constructed by the spatial preferential attachment model. 
Consider a function $h:\mfn\times \mcx\times\mfn_2\to\R_+$ that evaluates the edge structure of a (typical) vertex in the spatial preferential attachment model. Then  
\begin{align}\label{eq: spatialPA split up}
    \E\bigg[\int_{\mcx} h(\Xi-\delta_{(x,s)},(x,s),\Sigma\vert_{\langle \Xi-\delta_{(x,s)},(x,s)\rangle}) \, \Xi(\diff (x,s))\bigg] = \int_{\mcx} \E\bigl[h\big(\Xi,(x,s),(\Sigma_2^{<s},\Sigma_2^{>s})\big) \bigr] \, \alpha(\diff (x,s)),
\end{align}
where $\mathcal{L} \bigl( \Sigma_2^{<s},\Sigma_2^{>s} \bigm| \Xi, \Sigma \bigr) = \mathcal{L} \bigl( \Sigma_2^{<s} \bigm| \Xi_{s-}, \Sigma_{s-} \bigr) \otimes \mathcal{L} \bigl( \Sigma_2^{>s} \bigm| \Xi \bigr)$ is the product measure of outgoing and incoming edges for the additional vertex $(x,s)$.
%where $\Sigma_2^{<s}\mvert(\Xi_{s-},\Sigma_{s-}) $ is distributed as the edges of $(x,s)$ when added to the graph~$(\Xi_{s-},\Sigma_{s-})$ and is independent of $\Sigma_2^{>s}\mvert \Xi\setminus\Xi_{s-} $ which follows the edge distribution of the first vertex~$(x,s)$ in the graph~$(\Xi\setminus\Xi_{s-},\Sigma\setminus\Sigma_{s-})$.
\end{corollary}
\begin{remark}
  It is straightforward to generalise the proof below to statistics $h$ that also depend on past edges between other vertices, as well as replacing the underlying Poisson process distribution by a Gibbs distribution $P^{\lambda}$. The formula reads then
  \begin{align*}
    \E\bigg[\int_{\mcx} h(\Xi-\delta_{(x,s)}, \Sigma\vert_{\Xi_{s-}^{\atwo}}, (x,s), &\Sigma\vert_{\langle \Xi-\delta_{(x,s)},(x,s)\rangle}) \, \Xi(\diff (x,s))\bigg] \\
    &= \int_{\mcx} \E\bigl[h\big(\Xi, \Sigma\vert_{\Xi_{s-}^{\atwo}}, (x,s),(\Sigma_2^{<s},\Sigma_2^{>s})\big) \, \lambda(x\mvert \Xi)\bigr] \, \alpha(\diff (x,s)).
\end{align*}
\end{remark}
\begin{proof}
The result is a direct consequence of Theorem~\ref{thm:graph_gnz}. 
Due to the iterative construction of the graph, we may write down the kernel~$\tQ_{2\vert1}$, using $\xi_{s-} = \xi\vert_{\T_1\times [0,s)}$ and $\xi_{s} = \xi\vert_{\T_1\times [0,s]}$, as 
\begin{align*}
    \tQ_{2\vert 1}(\xi,\sigma,(x,s);\{\sigma_2\}) 
    &= \prod_{(y,t) \in \xi_{s-}} \varphi\bigg(\frac{s\hbit d(x,y)}{f(Z_y(s-))}\bigg)^{\sigma_2(\{x,y\})}\bigg(1-\varphi\bigg(\frac{s\hbit d(x,y)}{f(Z_y(s-))}\bigg)\bigg)^{1-\sigma_2(\{x,y\})} \\[-1mm]
    &\hspace*{6mm} \times \prod_{(y,t)\in\xi \setminus \xi_s} \varphi\bigg(\frac{t\hbit d(x,y)}{f(Z_x(t-))}\bigg)^{\sigma_2(\{x,y\})}\biggl(1-\varphi\bigg(\frac{t\hbit d(x,y)}{f(Z_x(t-))}\bigg)\biggr)^{1-\sigma_2(\{x,y\})}\\[1mm]
    &=\Prob \bigl( \Sigma_2^{<s}=\sigma^{<s}_2 \bigm| \Xi_{s-} = \xi_{s-}, \Sigma_{s-} = \sigma_{s-} \bigr) \, \Prob \bigl( \Sigma_2^{>s} = \sigma^{>s}_2 \bigm| \Xi  = \xi \bigr),
\end{align*}
where $\sigma^{<s}_2 = \sigma_2 \vert_{\langle \xi_{s-}, (x,s) \rangle}$ and $\sigma^{>s}_2 = \sigma_2 \vert_{\langle \xi \setminus \xi_{s}, (x,s) \rangle}$ are the outgoing and incoming edges of $(x,s)$, respectively.
%The first part corresponds to the distribution $\mathcal{L}(\Sigma_2^{>s})$ of ingoing edges of $x$ and depends on the indegree of vertex~$(x,s)$ up to time~$t-$, i.e.\ on $\sigma_2$ restricted to edges that connect $(x,s)$ to vertices born in $(s,t-)$.
%The second part corresponds to the distribution~$\mathcal{L}(\Sigma_2^{<s})$ of outgoing edges of $(x,s)$ and depends on the indegree up to time~$s-$ of vertices born in $(0,s-)$, i.e.\ on $\sigma$ restricted to the graph up to time~$s-$. Using this observation, it follows that the first term of the right-hand side in Theorem~\ref{thm:graph_gnz} equals the right-hand side of Equation~\eqref{eq: spatialPA split up}. Thus, it remains to consider the second term vanishes.  
Thus, the first term of the right-hand side in Theorem~\ref{thm:graph_gnz} equals the right-hand side of Equation~\eqref{eq: spatialPA split up}.

To see that the second term vanishes, note that
$$\tQ_1(\xi,(x,s); \hbit\cdot\hbit) = Q(\xi, \hbit\cdot\hbit) \quad \text{ on } \mfn_2\vert_{(\T\times(0,s))^{\atwo}},$$
as vertices do not influence edges constructed in the past, and
$$\tQ_{2\vert1}(\xi,(x,s), \sigma_1; \hbit\cdot\hbit) = \tQ_{2\vert1}(\xi,(x,s), \sigma_1\vert_{(\T\times(0,s))^{\atwo}}; \hbit\cdot\hbit),$$
as edges incident to $(x,s)$ are not influenced by edges created in the future between other vertices. Thus, if $h$ itself depends only on edges of $\sigma_1$ created in the past,  
$$
  h(\xi, \sigma_1, (x,s), \sigma_2) = h(\xi, \sigma_1\vert_{(\T\times(0,s))^{\atwo}}, (x,s), \sigma_2),
$$
in particular if $h$ does not depend on $\sigma_1$ at all, we have
\begin{align*}
  \int_{\mfn_2} \int_{\mfn_2} &h\bigl(\xi, \sigma_1, x, \sigma_2 \bigr) \; \tilde{Q}_{2 \mvert 1}(\xi,x,\sigma_1 ;\, d\sigma_2) \; \bigl( \tilde{Q}_1(\xi, x;\, d\sigma_1) - Q(\xi, d\sigma_1) \bigr) \\
  &= \int_{\mfn_2\vert_{((\T \times (0,s))^{\atwo})^c}} \int_{\mfn_2} h\bigl(\xi, \emptyset, x, \sigma_2 \bigr) \; \tilde{Q}_{2 \mvert 1}(\xi,x,\emptyset ;\, d\sigma_2) \; \bigl( \tilde{Q}_1(\xi, x;\, d\sigma_1) - Q(\xi, d\sigma_1) \bigr)
  = 0,
\end{align*}
since $\tilde{Q}_1$ and $Q$ are probability measures in their last arguments and we integrate over a constant.
Thus the second term in Theorem~\ref{thm:graph_gnz} vanishes. Note that this is not a consequence of Corollary~\ref{cor: GNZwithoutAdditionalInfluence} as the influence assumption is not fulfilled. 
\end{proof}

\section{Stein's Method} \label{sec: Stein method}
In what follows we derive general approximation results on the distance of two random graph distributions using the distribution ${P^{\lambda} \otimes Q^{\kappa}}$ of the generalised random geometric graph $(\Eta,\Tau)$ as target distribution. More formally, given any random graph $(\Xi,\Sigma)\sim P\otimes Q$, where $\Xi\sim P$ satisfies Condition \eqref{eq:condsigma} and $Q$ is an edge kernel,
we aim to bound the expression
\begin{equation} \label{eq: integral metric}
    d_{\mcf}(P\otimes Q ,P^{\lambda} \otimes Q^{\kappa}) = \sup\limits_{f\in\mcf} \big\vert \E(f(\Xi,\Sigma)) - \E(f(\Eta,\Tau))\big\vert.
\end{equation}
For a large part, we work in terms of a general integral probability metric (IPM)~$d_{\mcf}$, i.e.\ $\mcf$ is any suitable class of integrable test functions $\G \to \R$ (see \cite{zolotarev1984, muller1997}). Only from Subsection~\ref{ssec: bounding stein factors} on, we use Wasserstein metrics by setting $\mcf = \mcf_{\G}$.

To derive a bound for \eqref{eq: integral metric}, we apply Stein's method (see \cite{Stein1972}, for normal approximation; \cite{BarbourBrown1992}, for Poisson process approximation) to spatial random graphs. In the following, we describe how Stein's method can be used in our special setting. We refer to \textcite{BarbourChen2005}, \textcite{Ross2011} and \textcite{S2014} for more thorough introductions to Stein's method.

Following the generator approach by \textcite{Barbour1988}, we set up a Stein equation by using the generator of a Markov process that has the target distribution ${P^{\lambda} \otimes Q^{\kappa}}$ as its stationary distribution. In view of Lemma~\ref{le: stationary distribution} such a generator is $\mcg^{\lambda,\kappa}$ and the resulting
Stein equation reads
\begin{equation}\label{eq: Stein equation}
f(\xi,\sigma) - \E(f(\Eta,\Tau)) = \mcg^{\lambda,\kappa} h_f (\xi,\sigma).    
\end{equation}
We have to solve this in $h_f$ for every $f\in\mcf$, where $h_f$ is not restricted to lie in $\mcf'$ but any $h_f\in \mathcal{F}'$ for which the right hand side in~\eqref{eq: Generator} is well-defined is acceptable. The default candidate for a solution takes the form
\begin{equation} \label{eq: Stein solution}
  h_f(\xi,\sigma) = -\int_{0}^{\infty} \E\bigl[f((\Eta_s,\Tau_s)^{(\xi,\sigma)})\bigr] - \E\bigl[f((\Eta,\Tau))\bigr] \diff s,
\end{equation}
where $((\Eta_s,\Tau_s)^{(\xi,\sigma)})_{s\geq0}$ denotes the Markov process generated by $\mcg^{\lambda,\kappa}$ and started in ${(\xi,\sigma)\in\G}$. We show in Lemma~\ref{le: hf solves Stein equation} that this candidate is indeed well-defined and solves~\eqref{eq: Stein equation}.

Taking expectations for random $(\Xi,\Sigma)$ in~\eqref{eq: Stein equation} and combining the result with our initial expression~\eqref{eq: integral metric}, yields
$$d_{\mcf}(P\otimes Q ,P^{\lambda} \otimes Q^{\kappa}) = \sup\limits_{f\in\mcf} \big\vert \E(f(\Xi,\Sigma)) - \E(f(\Eta,\Tau))\big\vert =  \sup\limits_{f\in\mcf} \big\vert \mcg^{\lambda,\kappa} \E(h_f (\Xi,\Sigma))\big\vert. $$
In particular, we can now analyse the distance between two random graph distributions by considering the expression on the right hand side. One of the tools we use for this is a coupling of two Markov processes that we introduce in the following subsection.  

\subsection{Coupling} \label{ch: coupling}
We construct a coupling of two GBDPs generated by the same $\mcg^{\lambda,\kappa}$ and started in graphs $(\xi,\sigma),(\eta,\tau)\in\G$. This coupling combines the coupling approach of \textcite{stucki2014} for the vertex process with a maximal coupling for the edge process.
Set
\begin{align*}
    \lambda_{\text{max}}(x\mvert\xi,\eta) = \max(\lambda(x\mvert\xi),\lambda(x\mvert\eta))\text{\quad and \quad}
    \overline{\lambda}_{\text{max}}(\xi,\eta) = \int_{\mcx} \lambda_{\text{max}}(x\mvert\xi,\eta)\, \alpha(\diff x)
\end{align*}
and analogously
\begin{align*}
    \lambda_{\text{min}}(x\mvert\xi,\eta) = \min(\lambda(x\mvert\xi),\lambda(x\mvert\eta))\text{\quad and \quad}
    \overline{\lambda}_{\text{min}}(\xi,\eta) = \int_{\mcx} \lambda_{\text{min}}(x\mvert\xi,\eta)\, \alpha(\diff x).
\end{align*}
Let $((\Eta_s^{(\xi)},\Tau_s^{(\xi,\sigma)}),(\Eta_s^{(\eta)},\Tau_s^{(\eta,\tau)}))_{s\geq0}$ be a PJMP started at time $S_0 = 0$ and jumping at times $S_i = \sum_{j=1}^i J_j$ for $i\geq 1$, where $J_1,J_2,\ldots$ denote the holding intervals. Being in a state $((\Eta_{S_{i-1}}^{(\xi)},\Tau_{S_{i-1}}^{(\xi,\sigma)}),(\Eta_{S_{i-1}}^{(\eta)},\Tau_{S_{i-1}}^{(\eta,\tau)})) = ((\xi',\sigma'),(\eta',\tau'))$, we stay for a time $J_i\sim \text{Exp}(\overline{\lambda}_{\text{max}}(\xi',\eta') + \abs{\xi'\cup \eta'})$. After this holding time, the process jumps into a new state $((\Eta_{S_i}^{(\xi)},\Tau_{S_i}^{(\xi,\sigma)}),(\Eta_{S_i}^{(\eta)},\Tau_{S_i}^{(\eta,\tau)}))$ that can be described as follows. Define
\begin{align*}
    G_i &\sim \text{Ber}\bigg(\frac{\overline{\lambda}_{\text{max}}(\xi',\eta')}{\overline{\lambda}_{\text{max}}(\xi',\eta') + \abs{\xi'\cup \eta'}}\bigg)\\
    Y_i &\sim \frac{\lambda_{\text{max}}(\cdot\mvert\xi',\eta')}{\overline{\lambda}_{\text{max}}(\xi',\eta')}\\
    U_i &\sim \text{Unif}(\xi'\cup\eta')\\
    B_{\xi,i}\mvert Y_i &\sim \text{Ber}\bigg(\frac{{\lambda}(Y_i\mvert \xi')}{\lambda_{\text{max}}(Y_i \mvert\xi',\eta')}\bigg)\\
    B_{\eta,i}\mvert Y_i &\sim \text{Ber}\bigg(\frac{{\lambda}(Y_i\mvert \eta')}{\lambda_{\text{max}}(Y_i \mvert\xi',\eta')}\bigg)\\
    E_{x,Y_i}\mvert Y_i &\sim \text{Ber}\big(\kappa(x,Y_i)\big)\quad\quad\quad\text{for all }x\in \xi'\cup\eta'\text{ independently}
\end{align*}
and assume that, conditionally on $Y_i$, the edges $(E_{x,Y_i})_{x \in \xi' \cup \eta'}$ are independent of $(B_{\xi,i},B_{\eta,i})$ and $B_{\xi,i}$ and $B_{\eta,i}$ are maximally coupled, i.e.
$$
  \Prob(B_{\xi,i} = B_{\eta,i} = 1\mvert Y_i) = \frac{\lambda_{\text{min}}(Y_i\mvert\xi',\eta')}{\lambda_{\text{max}}(Y_i\mvert\xi',\eta')}.
$$
The variable $G_i$ determines whether there is a birth or a death event at time~$S_i$. In case of a death event the deleted vertex is given by $U_i$. In case of a birth event, the location of the added vertex is given by $Y_i$ while $B_{\xi,i}$ and $B_{\eta,i}$ indicate whether or not the vertex~$Y_i$ is added to $\xi'$ and $\eta'$, respectively. If the vertex~$Y_i$ is added to $\xi'$ and/or $\eta'$, we connect $Y_i$ to any other vertex~$x$ of the point pattern according to $E_{x,Y_i}$.

More formally, from $((\Eta_{S_{i-1}}^{(\xi)},\Tau_{S_{i-1}}^{(\xi,\sigma)}),(\Eta_{S_{i-1}}^{(\eta)},\Tau_{S_{i-1}}^{(\eta,\tau)})) = ((\xi',\sigma'),(\eta',\tau'))$, we construct the next state as follows.
If $G_i = 1$ (birth), set 
\begin{align*}
    \Eta_{S_i}^{(\xi)} &=\xi' + B_{\xi,i} \, \delta_{Y_i}, & \hspace*{-18mm}
    \Tau_{S_i}^{(\xi,\sigma)} &= \sigma' + B_{\xi,i}\, \sum_{y\in \xi'} E_{y,Y_i} \delta_{\{y,Y_i\}} \quad \text{ and} \\[1mm]
    \Eta_{S_i}^{(\eta)} &= \eta' + B_{\eta,i} \, \delta_{Y_i}, & \hspace*{-18mm}
    \Tau_{S_i}^{(\eta,\tau)} &= \tau' + B_{\eta,i}\, \sum_{y\in \eta'} E_{y,Y_i} \delta_{\{y,Y_i\}}.
\end{align*}
On the other hand, if $G_i = 0$ (death), set
\begin{align*}
    \Eta_{S_i}^{(\xi)} &= \xi' - \1{\{U_i\in \xi'\}} \, \delta_{U_i}, & \hspace*{-10.5mm} 
    \Tau_{S_i}^{(\xi,\sigma)} &= \sigma'\vert_{\Eta_{S_i}^{(\xi)}} = \sigma' - \1{\{U_i\in \xi'\}} \, \sigma' \vert_{\langle \xi' ,U_i\rangle} \quad \text{ and} \\[1mm]
    \Eta_{S_i}^{(\eta)} &= \eta' - \1{\{U_i\in \eta'\}} \, \delta_{U_i}, & \hspace*{-10.5mm} 
    \Tau_{S_i}^{(\eta,\tau)} &= \tau'\vert_{\Eta_{S_i}^{(\eta)}} = \tau' - \1{\{U_i\in \eta'\}} \, \tau' \vert_{\langle \eta' ,U_i\rangle}.
\end{align*}

The following proposition shows that this construction of a pure-jump Markov process yields a coupling of $\mcg^{\lambda,\kappa}$-processes started in graphs $(\xi,\sigma),(\eta,\tau)\in\G$, respectively. 
\begin{proposition}
Assume that the conditional intensity~$\lambda$ fulfils \eqref{eq: stability cond}. Then the processes $(\Eta^{(\xi)},\Tau^{(\xi,\sigma)})$ and $(\Eta^{(\eta)},\Tau^{(\eta,\tau)})$ are PJMPs with generator $\mcg^{\lambda,\kappa}$.
\end{proposition}
\begin{proof}
Note that the holding times $J_1,J_2,\ldots$ only depend on the vertex processes $\Eta^{(\xi)}$ and $\Eta^{(\eta)}$. Thus, from the proof of Proposition~35 in \textcite{stucki2014},
\begin{align*}
    \Prob(J_1 >s ) = 1- a(\xi,\eta) s + \mathcal{O}(s^2)\quad \text{and }\quad \Prob(J_1+J_2 \leq s) = \mathcal{O}(s^2) \quad \text{ as $s\to 0$},
\end{align*}
where $a(\xi,\eta) = \overline{\lambda}_{\text{max}}(\xi,\eta) + \abs{\xi\cup\eta}\leq c + \abs{\xi\cup\eta}$ for some $c>0$, which exists by the stability assumption~\eqref{eq: stability cond}.
Let $h :\G \to\R$ be bounded and measurable. Then
%$h\in \mathcal{F}' = \{\Tilde{h}:\G \to\R \text{ measurable}\mvert \norm{\Tilde{h}}_{\infty} <\infty\}$, then
\begin{align*}
    &\E\bigl[h\bigl(\Eta_s^{(\xi)},\Tau_s^{(\xi,\sigma)}\bigr)\bigr]\\
    &= \Prob(J_1\leq s, J_2>s)\, \E\bigl[h\bigl(\Eta_{J_1}^{(\xi)},\Tau_{J_1}^{(\xi,\sigma)}\bigr)\bigr] +  \Prob(J_1>s)\, h(\xi,\sigma) + \mathcal{O}(s^2)\\
    &= a(\xi,\eta) \hbit s \hbit \Bigl[\E h\bigl(\xi + B_{\xi,1} \, \delta_{Y_1}, \, \sigma + B_{\xi,1}\, \tsum_{y\in  \xi} E_{y,Y_1} \delta_{\{y,Y_1\}}\bigr)\,\Prob(G_1 = 1)\\
    &\hspace{5mm}+ \E h\bigl(\xi -  \1{\{U_1\in \xi\}} \, \delta_{U_1}, \, \sigma - \1{\{U_1\in \xi\}} \, \sigma\vert_{\langle \xi, U_1\rangle} \bigr)\,\Prob(G_1 = 0) \!\Bigr]+ (1- a(\xi,\eta))\,s \, h(\xi,\sigma) + \mathcal{O}(s^2).
\end{align*}
The independence of $E_{y,Y_1}\mvert Y_1$ for $y\in\xi$ yields that $\sum_{y\in  \xi} E_{y,Y_1} \delta_{\{y,Y_1\}}\mvert Y_1 \sim Q_2^{\kappa}(\xi,Y_1;\cdot)$ (see Section~\ref{ch: kappa geom RG}). Hence, we obtain
\begin{align*}
    &\lim_{s\to 0 } \frac{\E\big[h\big(\Eta_s^{(\xi)},\Tau_s^{(\xi,\sigma)}\big)\big] - h(\xi,\sigma)}{s} \\
    %&= a(\xi,\eta)\,\Prob(G_1 = 1)\, \bigg[\E h\bigg(\xi + B_{\xi,1} \, \delta_{Y_1},\sigma + B_{\xi,1}\, \sum_{y\in  \xi} E_{y,Y_1} \delta_{\{y,Y_1\}}\bigg) - h(\xi,\sigma)\bigg]\\
    %&\hspace{4mm}+ a(\xi,\eta)\,\Prob(G_1 = 0)\,\bigg[ \E h\bigg(\xi -  \1_{\{U_1\in \xi\}}\delta_{U_1} ,\sigma - \1_{\{U_i\in \xi\}} \sigma\vert_{\langle \xi ,U_1\rangle} \bigg) - h(\xi,\sigma) \bigg]\\
    &= a(\xi,\eta)\,\Prob(G_1 = 1)\, \E \bigl[\bigl(h\bigl(\xi +\delta_{Y_1},\sigma +  \tsum_{y\in  \xi} E_{y,Y_1} \delta_{\{y,Y_1\}}\bigr) - h(\xi,\sigma)\bigr)  B_{\xi,1} \bigr]\\
    &\hspace{4mm}+ a(\xi,\eta)\,\Prob(G_1 = 0)\, \E \bigl[ h\bigl(\xi -  \1{\{U_1\in \xi\}} \, \delta_{U_1} ,\sigma - \1{\{U_1\in \xi\}} \, \sigma\vert_{\langle \xi ,U_1\rangle} \bigr) - h(\xi,\sigma) \bigr]\\[0.5mm]
    &= \overline{\lambda}_{\text{max}}(\xi,\eta) \int_{\mcx} \int_{\mfn_2} h(\xi + \delta_y, \sigma + \tilde{\sigma}) - h(\xi,\sigma) \; Q_2^{\kappa}(\xi,y;\diff \tilde{\sigma}) \,\frac{\lambda(y\mvert\xi)}{\lambda_{\text{max}}(y\mvert \xi,\eta)} \,\frac{\lambda_{\text{max}}(y\mvert \xi,\eta)}{\overline{\lambda}_{\text{max}}(\xi,\eta)}  \, \alpha(\diff y)\\
    &\hspace{4mm}+ \abs{\xi\cup \eta} \, \int_{\mcx} h(\xi -  \1{\{u\in \xi\}} \, \delta_{u} ,\sigma - \1{\{u\in \xi\}} \, \sigma\vert_{\langle \xi ,u\rangle} ) - h(\xi,\sigma)\, \frac{\xi\cup \eta}{\abs{\xi\cup \eta}}(\diff u)\\
    %&= \int_{\mcx} \int_{\mfn_2} h(\xi + \delta_y, \sigma + \tilde{\sigma}) - h(\xi,\sigma) \,Q_2^{\kappa}(\xi,y,\diff \tilde{\sigma}) \,\lambda(y\mvert\xi) \alpha(\diff y)\\
    %&\hspace{4mm}+ \int_{\mcx} h(\xi -  \delta_{u} ,\sigma -  \sigma\vert_{\langle \xi ,u\rangle} ) - h(\xi,\sigma)\, {\xi}(\diff u)\\
    &= \mcg^{\lambda,\kappa}h(\xi,\sigma).
\end{align*}
\vspace*{-12mm}

\hspace*{2mm}
\end{proof}
Define the \emph{coupling time} $\tau_{(\xi,\sigma),(\eta,\tau)} := \inf\bigl\{s\geq0 \bigm| \bigl(\Eta_s^{(\xi)},\Tau_s^{(\xi,\sigma)}\bigr) = \bigl(\Eta_s^{(\eta)},\Tau_s^{(\eta,\tau)} \bigr) \bigr\}$ and refer to $\varrho((\xi,\sigma),(\eta,\tau)) := \bigabs{(\xi\cup \eta)\setminus \bigl\{x\in\xi\cap\eta \bigm| \sigma(\{x,y\}) = \tau(\{x,y\})\;\forall y\in\xi\cap\eta \bigr\}}$ as \emph{graph difference}. The latter is the number of vertices appearing either in only one of the graphs or appearing in both but having different edge structures \emph{within $\xi\cap\eta$}.

\begin{remark} \label{re: coupling in Poi case and edge case}\ 
\begin{enumerate}
    \item[a)] In the special case where the conditional intensity $\lambda$ does not depend on the vertex pattern~$\xi$, i.e.\ $\lambda(\cdot\mvert \xi) = \lambda(\cdot)$, a vertex is (in the case of a birth event) always added to both graphs, i.e.\ $B_{\xi,i} = B_{\eta,i} = 1$ for all $i\in\N$. Writing $(\zeta, \varepsilon) = (\xi \cap \eta, \sigma\cap\tau)$, the coupling then reduces to
\begin{alignat*}{1}
    \Eta_s^{({\xi})} &= \Eta_s^{(\zeta)} + \sum_{x\in\xi\setminus\zeta}\delta_x \1{\{L_x > s\}}\\
    \Tau_s^{(\xi,\sigma)} &= \Tau_s^{(\zeta,\varepsilon)} + \sum_{{\{x,y\}\in\sigma\setminus\varepsilon}} \delta_{\{x,y\}}\1{\{L_x > s\}} \hbit \1{\{L_y > s\}} + \sum_{{x\in\xi\setminus\zeta}} \Tau_{\Eta_s^{(\zeta)}\setminus \zeta,x}\,\1{\{L_x > s\}}
   % &= \Tau_s^{(\zeta,\varepsilon)} + \sum_{{\{x,y\}\in\sigma\setminus\varepsilon}}\sigma(\{x,y\})\,\delta_{\{x,y\}}\1_{\{L_x > s\}}\1_{\{L_y > s\}}+ \sum_{{x\in\xi\setminus\zeta}} \sum_{y\in\Eta_s^{(\zeta)}\setminus \zeta} \Tau_{\Eta_s^{(\zeta)}\setminus \zeta,x}(\{x,y\})\delta_{\{x,y\}}\1_{\{L_x > s\}},
   %\Eta_s^{({\eta})} &:= \Eta_s^{(\zeta)} + \sum_{x\in \eta\setminus \zeta}\delta_x \1_{\{L_x > s\}}\\
    %\Tau_s^{(\eta,\tau)} &:= \Tau_s^{(\zeta,\varepsilon)}+ \sum_{{\{x,y\}\in\tau\setminus\varepsilon}}\tau(\{x,y\})\,\delta_{\{x,y\}}\1_{\{L_x > s\}}\1_{\{L_y > s\}} + \sum_{{x\in\eta\setminus\zeta}} %\Tau_{\Eta_s^{(\zeta)}\setminus \zeta,x}\, \1_{\{L_x > s\}},
    %&= \Tau_s^{(\zeta,\varepsilon)}+ \sum_{{\{x,y\}\in\tau\setminus\varepsilon}}\tau(\{x,y\})\,\delta_{\{x,y\}}\1_{\{L_x > s\}}\1_{\{L_y > s\}} + \sum_{{x\in\eta\setminus\zeta}} \sum_{y\in\Eta_s^{(\zeta)}\setminus \zeta} \Tau_{\Eta_s^{(\zeta)}\setminus \zeta,x}(\{x,y\})\delta_{\{x,y\}}\1_{\{L_x > s\}}. 
\end{alignat*}
and the corresponding expressions for $H_s^{(\eta)}$, $\Tau_s^{(\eta,\tau)}$,
where $L_x$ denote the independent Exp$(1)$-distributed lifetimes of the vertices $x\in {\xi\cup \eta}$ and we use $\Tau_{\chi,x}$ again to denote an independent random element on $\mfn_2$ with distribution
%$\mathcal{L} \bigl( \Tau \vert_{\langle \xi, x \rangle} \bigm| \Eta = \xi+\delta_x \bigr) = $
$Q_2^{\kappa}(\chi,x;\cdot)$.
%$(\Eta_s^{(\zeta)},\Tau_s^{(\zeta,\varepsilon)})_{s\geq0}$ is the Markov process generated by $\mcg^{\lambda,\kappa}$ and started in $(\Eta_0^{(\zeta)},\Tau_0^{(\zeta,\varepsilon)}) = (\zeta, \varepsilon)$.
The coupling time is then bounded from above by the maximum lifetime $L_x$ over all vertices $x$ counted by the quantity $\varrho((\xi,\sigma),(\eta,\tau))$. Therefore
$$\Prob(\tau_{(\xi,\sigma),(\eta,\tau)} \leq s) \geq \bigl(1-e^{-s}\bigr)^{\varrho((\xi,\sigma),(\eta,\tau))}$$
and equality holds if and only if the edge structure within $\xi \cap \eta$ is the same in both graphs.
\item[b)] The (possibly) different edge structure in a graph coupling does not affect the coupling of the vertex processes.
%or the rates of adding/deleting a vertex.
In particular, for $\xi = \eta$, the graph coupling consists of two identical vertex processes and a vertex (in the case of a birth event) is always added to both graphs, i.e.\ $B_{\xi,i} = B_{\eta,i} = 1$ for all $i\in\N$. Writing $\varepsilon = \sigma\cap\tau$, the coupling reduces to
\begin{alignat*}{1}
    \Eta_s^{({\xi})} &=  \Eta_s^{(\eta)} \\%= \Eta_s^{(\eta)}  \\
    \Tau_s^{(\xi,\sigma)} &= \Tau_s^{(\xi,\varepsilon)} + \sum_{{\{x,y\}\in\sigma\setminus\varepsilon}} \delta_{\{x,y\}} \1{\{L_x > s\}} \hbit \1{\{L_y > s\}},
    %\1_{\{x \in \Eta_s^{({\zeta})} \cap \zeta \}}\,\1_{\{y \in \Eta_s^{({\zeta})} \cap \zeta\}},%\\
    %\Tau_s^{(\eta,\tau)} &= \Tau_s^{(\zeta,\varepsilon)}+ \sum_{{\{x,y\}\in\tau\setminus\varepsilon}}\tau(\{x,y\})\,\delta_{\{x,y\}}\1_{\{x \in \Eta_s^{({\zeta})} \cap \zeta \}}\,\1_{\{y \in \Eta_s^{({\zeta})} \cap \zeta\}} ,
\end{alignat*}
and the corresponding expression for $\Tau_s^{(\eta,\tau)}$,
where $L_x$ denote the lifetimes of the vertices $x\in \xi$. By the same reasoning as above, 
$$
%\Prob(\tau_{(\xi,\sigma),(\eta,\tau)} \leq s) =
\Prob(\tau_{(\xi,\sigma),(\xi,\tau)} \leq s) \geq \bigl(1-e^{-s}\bigr)^{\varrho((\xi,\sigma),(\xi,\tau))}.$$
If every vertex $x \in \xi$ has at most one incident edge that is not in both $\sigma$ and $\tau$, the lifetimes of the edge are independent Exp$(2)$-distributed random variables. So
\begin{align} \label{eq: coupling time for pure edge diff}
  \Prob(\tau_{(\xi,\sigma),(\xi,\tau)} \leq s) = \bigl(1-e^{-2s}\bigr)^{\varrho((\xi,\sigma),(\xi,\tau))/2}.
\end{align}
\end{enumerate}
\end{remark}

The observation that the edge structure has no influence on the coupling of the vertex processes and the fact that $\varrho((\xi,\sigma),(\eta,\tau)) \leq \abs{\xi\cup\eta}$ allow us to adapt the coupling time results of \textcite{stucki2014} to our setting. Denote by $\norm{\cdot}$ the total variation norm of signed measures, i.e.\ $\norm{\xi-\eta}$ denotes the number of points in the symmetric difference of the point patterns.

\begin{theorem} \label{thm: coupling time}
Assume that $\lambda$ fulfils \eqref{eq: stability cond}.
For all $(\xi,\sigma),(\eta,\tau)\in \mathbb{G}$ the coupling time $\tau_{(\xi,\sigma),(\eta,\tau)}$ is integrable. In particular, if $\norm{\xi-\eta} = 1$, we have for any $n^*\in\N\cup\{\infty\}$
\begin{align*}
    \E(\tau_{(\xi,\sigma),(\eta,\tau)}) \leq (n^*-1)! \,\bigg(\frac{\varepsilon}{c}\bigg)^{n^*-1} \bigg(\frac{1}{c} \sum_{i=n^*}^{\infty} \frac{c^i}{i!} + \int_0^c \frac{1}{s}\sum_{i=n^*}^{\infty} \frac{s^i}{i!}\diff s\bigg) + \frac{1+\varepsilon}{\varepsilon}\sum_{i=1}^{n^*-1} \frac{\varepsilon^i}{i} =: B^{*},
\end{align*}
where
$$\varepsilon = \sup_{\norm{\tilde{\xi}-\tilde{\eta}} \, = \, 1}\int_{\mcx} \abs{\lambda(x\mvert\tilde{\xi}) - \lambda(x\mvert \tilde{\eta})} \, \alpha(\diff x)$$
and
$$c = c(n^*) = \sup_{\abs{\tilde{\xi} \cup \tilde{\eta}} \, \geq \, n^*}\int_{\mcx}\abs{\lambda(x\mvert\tilde{\xi}) - \lambda(x\mvert\tilde{\eta})} \, \alpha(\diff x).$$
%$\varrho(({\xi},{\sigma}),({\eta},{\tau})) = \abs{\xi \cup \eta} - (\abs{\xi \cap \eta} - e)$ with $e\leq \abs{\xi \cap \eta}$ dependent on $\sigma,\tau$. \\
%Now, if $\varrho \geq n^*$ then this implies $\abs{\xi \cup \eta} \geq n^*$ as $(\abs{\xi \cap \eta} - e)\geq 0$. On the other hand, if $\abs{\xi \cup \eta} \geq n^*$ let $\tau$ contain all possible edges while $\abs{\sigma}=0$. Then $\varrho(({\xi},{\sigma}),({\eta},{\tau})) = \abs{\xi \cup \eta} \geq n^*$.
For $n^* = \infty$, the bound is interpreted as $\E(\tau_{(\xi,\sigma),(\eta,\tau)}) \leq \frac{1+\eps}{\eps} (e^{\eps}-1)$. If any of $\eps$ and $c$ are zero, the bound is to be understood as the corresponding limit.
\end{theorem}

Note that the bound is exactly the same as in \textcite{stucki2014} except that $c$ is defined as the supremum over $\tilde{\xi}, \tilde{\eta} \in \mfn$ with $\abs{\tilde{\xi} \cup \tilde{\eta}} \geq n^*$ rather than $\|\tilde{\xi} - \tilde{\eta}\| \geq n^*$.

\begin{proof}
The proof is similar to that of Theorem~37 in \textcite{stucki2014}, but some careful adjustments are needed.

We call a death event in our coupled GBDP a ``good death'' if it reduces the graph difference~$\varrho$, i.e.\ the number of vertices that either appear in only one of the coupled graphs or appear in both but have different edge structures within the set of common vertices. Due to the edge structure, a good death may reduce the graph difference by more than one. Furthermore, we call a birth event a ``bad birth'' if the graph difference is increased. Such an increase can only happen if the new vertex is added in only one of the coupled processes (if it were added in both, due to the fact that any added edges to common vertices are the same, $\varrho$ would not change). Since, with probability one, a vertex is not born at a position where the coupled processes already have vertices, a bad birth can increase the graph difference by only one as the added edges do not change the edge structure within the common points. 

Assume that $\varrho \bigl((\xi,\sigma)(\eta,\tau)\bigr) = n$ and let $\tau_1 = \inf\bigl\{t\geq 0 \bigm| \varrho \bigl((\Eta_{t}^{(\xi)},\Tau_{t}^{(\xi,\sigma)}),(\Eta_{t}^{(\eta)},\Tau_{t}^{(\eta,\tau)}) \bigr) \neq n \bigr\}$ be the time of the first jump in $\varrho$. Then $A_n= \bigl\{\varrho\bigl((\Eta_{\tau_1}^{(\xi)},\Tau_{\tau_1}^{(\xi,\sigma)}),(\Eta_{\tau_1}^{(\eta)},\Tau_{\tau_1}^{(\eta,\tau)}) \bigr) \leq n-1\bigr\}$ is the event that this first jump is due to a good death.
%Furthermore, define the filtration $\mathcal{F}_t = \sigma(\bigl((\Eta_{s}^{(\xi)},\Tau_{s}^{(\xi,\sigma)}),(\Eta_{s}^{(\eta)},\Tau_{s}^{(\eta,\tau)}) \bigr); s\leq t\}$ and note that  the coupled GBDP fulfils the strong Markov property. In particular, we obtain
%\begin{align*}
%    &\Prob(\text{``good death'' at time }S_j\mvert \mathcal{F}_{S_{j-1}}) = \frac{\varrho\bigl((\Eta_{S_{j-1}}^{(\xi)},\Tau_{S_{j-1}}^{(\xi,\sigma)}),(\Eta_{S_{j-1}}^{(\eta)},\Tau_{S_{j-1}}^{(\eta,\tau)}) \bigr)}{\overline{\lambda}_{\text{max}}(\Eta_{S_{j-1}}^{(\xi)},\Eta_{S_{j-1}}^{(\eta)}) + \abs{\Eta_{S_{j-1}}^{(\xi)}\cup\Eta_{S_{j-1}}^{(\eta)}}}\\
%    &\Prob(\text{``bad birth'' at time }S_j\mvert \mathcal{F}_{S_{j-1}}) = \frac{\overline{\lambda}_{\text{max}}(\Eta_{S_{j-1}}^{(\xi)},\Eta_{S_{j-1}}^{(\eta)})-\overline{\lambda}_{\text{min}}(\Eta_{S_{j-1}}^{(\xi)},\Eta_{S_{j-1}}^{(\eta)})}{\overline{\lambda}_{\text{max}}(\Eta_{S_{j-1}}^{(\xi)},\Eta_{S_{j-1}}^{(\eta)}) + \abs{\Eta_{S_{j-1}}^{(\xi)}\cup\Eta_{S_{j-1}}^{(\eta)}}}.
%\end{align*}
We can adapt Lemma~36 of \textcite{stucki2014} by simply replacing the total variation distance by the $\varrho$-difference to obtain the lower bound
\begin{align} \label{eq: good death bound}
    \Prob(A_n) \geq \bigg(1 + \frac{1}{n} \sup_{\varrho((\tilde{\xi},\tilde{\sigma}),(\tilde{\eta},\tilde{\tau})) = n}\int_{\mcx}\bigabs{\lambda(x\mvert\tilde{\xi}) - \lambda(x\mvert\tilde{\eta})} \, \alpha(\diff x)\bigg)^{-1} >0.
\end{align}

By bounding the right hand side further, we obtain $\Prob(A_n) \geq (1+c(n^*)/n)^{-1}$ 
for $n \geq n^*$ due to the fact that $\varrho((\tilde{\xi},\tilde{\sigma}),(\tilde{\eta},\tilde{\tau})) = n$ implies  $\abs{\tilde{\xi} \cup \tilde{\eta}} \geq n$. We use $\Prob(A_n) \geq (1+\eps)^{-1}$ for $n < n^*$, which follows (for general $n \in \N$) from
\begin{align*} 
  \sup_{\varrho((\tilde{\xi},\tilde{\sigma}),(\tilde{\eta},\tilde{\tau})) = n}\int_{\mcx}\bigabs{\lambda(x\mvert\tilde{\xi}) - \lambda(x\mvert\tilde{\eta})} \, \alpha(\diff x) \leq \sup_{\norm{\tilde{\xi} - \tilde{\eta}} \leq n} \int_{\mcx}\bigabs{\lambda(x\mvert\tilde{\xi}) - \lambda(x\mvert\tilde{\eta})} \, \alpha(\diff x) \leq n \eps.
\end{align*}

In the same way as in the proofs of Theorem~37 and Lemma~38 in \textcite{stucki2014}, we may now couple $\X = \bigl( \varrho((\Eta_{t}^{(\xi)},\Tau_{t}^{(\xi,\sigma)}), (\Eta_{t}^{(\eta)}, \Tau_{t}^{(\eta,\tau)}))\bigr)_{t \geq 0}$ with a birth-and-death process $\Y$ on $\Z_{+}$ that has expected holding time $1/n$ at state $n$ and jumps up with probability given by the lower bound on $\Prob(A_n)$ in such a way that $\Y$ dominates $\X$ at all times. As a consequence, the first hitting time $\tilde{\tau}_n$ of $\Y$ at zero  dominates $\tau_{(\xi,\sigma),(\eta,\tau)}$, so that the bound on $\E \tau_{(\xi,\sigma),(\eta,\tau)}$ can be obtained by applying standard techniques for computing expected hitting times in PJMPs. 

Although the process $\X$ has somewhat different transition probabilities than the process $\bigl( \norm{\Eta_{t}^{(\xi)}-\Eta_{t}^{(\eta)}} \bigr)_{t \geq 0}$ considered in \textcite{stucki2014}, the dominating process $\Y$ is exactly the same (except for the slightly redefined $c^*$). This is because we have the same bounds on $\Prob(A_n)$ and because a bad birth increases the process $\X$ only by one, whereas larger decreases due to good deaths do no compromise domination. Overall the result is the same as in \textcite{stucki2014} but with redefined $c^*$.
\end{proof}

\begin{remark}[Special cases] \label{re: special cases coupling time}
For easy access, we resume here the observations from \textcite[Remark~6]{stucki2014}. 
    \begin{itemize}
        \item[a)] If we use a bound for $c$ does not depend on $n^*$ the choice $n^* = \lceil c/\varepsilon\rceil$ yields an optimal bound.
        \item[b)]  Assume that $\varepsilon<1$ and $\varrho((\xi,\sigma),(\eta,\tau)) = 1$. Then setting $n^* = \infty$ yields
        $$\E(\tau_{(\xi,\sigma),(\eta,\tau)}) \leq \frac{1+\varepsilon}{\varepsilon} \log\bigg(\frac{1}{1-\varepsilon}\bigg)\leq \frac{1+\varepsilon}{1-\varepsilon}.$$\\[-10mm]
        \item[c)] In the special cases of Remark~\ref{re: coupling in Poi case and edge case}, we have $\eps=0$ and therefore $\E(\tau_{(\xi,\sigma),(\eta,\tau)}) \leq 1$ if
        $\varrho((\xi,\sigma),(\eta,\tau)) = 1$.
    \end{itemize}
\end{remark}

\begin{proposition}\label{prop: coupling time for random graphs}
Let $(\Xi,\Sigma)$ and $(\Eta,\Tau)$ be two random graphs (not necessarily with the same distribution) whose vertex processes fulfil the stability condition \eqref{eq: stability cond}. Then the coupling time $\tau_{(\Xi,\Sigma),(\Eta,\Tau)} = \inf\bigl\{s\geq0 \bigm| (\Eta_s^{(\Xi)},\Tau_s^{(\Xi,\Sigma)}) = (\Eta_s^{(\Eta)},\Tau_s^{(\Eta,\Tau)}) \bigr\}$ of the randomly initialised processes % the coupling $((\Eta_s^{(\Xi)},\Tau_s^{(\Xi,\Sigma)}),(\Eta_s^{(\Eta)},\Tau_s^{(\Eta,\Tau)}))_{s\geq0}$
is integrable regardless of the joint distribution of $(\Xi,\Sigma)$ and $(\Eta,\Tau)$.
\end{proposition}
\begin{proof}
The proof follows analogously to the proof of \textcite[Proposition~39]{stucki2014} as $\varrho((\Xi,\Sigma),(\Eta,\Tau))\leq \abs{\Xi\cup \Eta}$.
\end{proof}
We are now able to prove that the candidate~$h_f$ defined in \eqref{eq: Stein solution} is a solution of the Stein equation~\eqref{eq: Stein equation}. 
\begin{lemma} \label{le: hf solves Stein equation}
Let $(\Eta,\Tau)$ be a $\mathrm{RGG}(\lambda,\kappa)$ whose vertex process fulfils the stability condition~\eqref{eq: stability cond} and $f \colon \G \to \R$ a bounded measurable function. Then $h_f$ given in \eqref{eq: Stein solution} is well-defined and solves the Stein equation~\eqref{eq: Stein equation}.
\end{lemma}
\begin{proof}%[Proof of Lemma~\ref{le: hf solves Stein equation}]
%Again, let $((\Eta_s,\Tau_s))_{s\geq0}$ be the pure-jump Markov process generated by $\mcg^{\lambda,\kappa}$ and started in $(\Eta_0,\Tau_0)$. 
Consider $(\xi,\sigma)\in\G$ and let $(\Xi,\Sigma)$ be a random graph. % (whose vertex processes fulfil the stability condition \eqref{eq: stability cond}).
Denote by $((\Eta_s,\Tau_s)^{(\xi,\sigma)})_{s\geq0}$ and $((\Eta_s,\Tau_s)^{(\Xi,\Sigma)})_{s\geq0}$ the GBDPs generated by $\mcg^{\lambda,\kappa}$ and started in $(\xi,\sigma)$ and $(\Xi,\Sigma)$, respectively. Then Proposition~\ref{prop: coupling time for random graphs} yields the integrability of the coupling time and thus
\begin{alignat}{1} \label{eq: hf well-defined}
	&\notag\bigg\vert\int_{0}^{\infty} \E\bigl[f((\Eta_s,\Tau_s)^{(\xi,\sigma)})\bigr] -\E\bigl[f((\Eta_s,\Tau_s)^{(\Xi,\Sigma)})\bigr] \diff s\bigg\vert\\
	&\notag\leq \int_{0}^{\infty} \bigg\vert\E\bigl[f(\Eta_s^{(\xi)},\Tau_s^{(\xi,\sigma)})- f(\Eta_s^{(\Xi)},\Tau_s^{(\Xi,\Sigma)})\bigr] \bigg\vert\diff s\\
    &\leq 2\,\norm{f}_{\infty}  \int_{0}^{\infty}\Prob(\tau_{(\xi,\sigma),(\Xi,\Sigma)} > s) \diff s < \infty.
	%&\leq 2\,\norm{f}_{\infty}\int_{0}^{\infty}\E\bigl[1- \bigl(1-e^{-s}\bigr)^{\varrho((\xi,\sigma),(\Xi,\Sigma))}\bigr] \diff s\\
     %Bernoullis inequality
	%&\leq 2\,\norm{f}_{\infty} \int_{0}^{\infty}\E\bigl[\varrho((\xi,\sigma),(\Xi,\Sigma))\,e^{-s}\bigr] \diff s = 2\,\norm{f}_{\infty}\,\E\bigl[\varrho((\xi,\sigma),(\Xi,\Sigma))\bigr] < \infty
\end{alignat}
In particular, this yields the well-definedness of $h_f$ by taking $(\Xi,\Sigma) = (\Eta,\Tau)$ to be the $\mathrm{RGG}(\lambda,\kappa)$, %and thus distributed according to the stationary distribution of $\mcg^{\lambda,\kappa}$
implying $(\Eta_s,\Tau_s)^{(\Eta,\Tau)} \overset{\text{d}}{=} (\Eta,\Tau)$ for all $s\geq0$.

It remains to show that $h_f$ defined in \eqref{eq: Stein solution} solves the Stein equation~\eqref{eq: Stein equation}. By \textcite[Proposition~1.1]{S2014}, it suffices to show that the map $s\mapsto \E f((\Eta_s,\Tau_s)^{(\xi,\sigma)}) = P_s f(\xi,\sigma)$ is well-defined and continuous at $0$.
%This yields the claim, as the continuity (and well-definedness) of the map $s\mapsto P_s f(\xi,\sigma)$
This follows directly from \textcite[Theorem 13.9]{Kallenberg2021fomp}. 
\end{proof}

\subsection{Bounds for general IPMs} \label{ssec: bounds for general IPM}

We are now able to state a general bound for absolute differences of expectations, which can then be further analysed to obtain bounds for a given IPM based on the properties of functions $f \in \mcf$. 

\begin{theorem} \label{thm: general approx result}
Let $(\Eta,\Tau)\sim P^{\lambda}\otimes Q^{\kappa}$ be the generalised random geometric graph with vertex process fulfilling the stability condition~\eqref{eq: stability cond} and $(\Xi,\Sigma)$ a random graph having law $P\otimes Q$, where $Q$ is an edge kernel and $\Xi\sim P$ fulfils Condition~\eqref{eq:condsigma} and has Papangelou kernel $L$. Then, for all $f \colon \G \to \R$ bounded and measurable,
\begin{align} \label{eq: general approx result}
    \bigl| \E\bigl(f(\Xi,\Sigma)\bigr)-\E\bigl(f(\Eta,\Tau)\bigr)\bigr| &\leq \E \biggl(\int_{\mcx}\triangle_V h_f(\Xi,\Sigma)\,\bigl|\lambda(x\mvert \Xi) \, \alpha(\diff x)-L(\diff x\vert\Xi)\bigr| \biggr) \notag\\
    &\hspace*{4mm} {} + \E\biggl(\int_{\mcx} \triangle_Eh_f(\Xi,\Sigma)\, \E\bigl[\norm{{\hat{v}(\Tilde{\Sigma}_{\Xi,\Sigma,x})}}_1 \bigm| \Xi,\Sigma\bigr] \,\lambda(x\mvert \Xi) \,\alpha(\diff x) \biggr) \notag\\
    &\hspace*{4mm} {} + \E \biggl( \int_{\mcx}\int_{\mfn_2} \triangle_Vh_f(\Xi,\sigma)\,\bigabs{\tilde{Q}_1(\Xi,x;\diff \sigma) - Q(\Xi;\diff \sigma)}\,L(\diff x \,\vert \Xi) \biggr),
\end{align}
\vspace*{-6mm}

where
\begin{align*}
  \Tilde{\Sigma}_{\xi,\sigma,x}\sim \tilde{Q}_{2 \mvert 1}(\xi,x,\sigma ;\, \cdot)
  &= \mathcal{L} \bigl( \Sigma \vert_{\langle \xi, x \rangle} \bigm| \Xi = \xi+\delta_x, \Sigma \vert_{\xi^{\atwo}} = \sigma \bigr) \\[0.5mm]
  \tilde{Q}_1(\xi, x;\, \cdot) &= \mathcal{L} \bigl( \Sigma \vert_{\xi^{\langle 2\rangle}} \bigm| \Xi = \xi+\delta_x\bigr) \\[-0.5mm]
  \norm{\hat{v}(\tilde{\Sigma}_{\xi,\sigma,x})}_1 &= \sum_{i=1}^{\abs{\xi}} \Bigl| \kappa(x_i,x) - \Prob\bigl(\Tilde{\Sigma}_{\xi,\sigma,x}(\{x,x_i\}) = 1 \bigm| (\Tilde{\Sigma}_{\xi,\sigma,x}(\{x,x_j\}))_{j\neq i}\bigr) \Bigr|.
\end{align*}
Furthermore, for all $(\xi,\sigma)\in\mathbb{G}$ and $B^*$ as in Theorem~\ref{thm: coupling time},
the Stein factors are are given as
\begin{align*}
    \triangle_V\,h_f(\xi,\sigma) &= \sup\limits_{x\in\mcx\setminus\xi} \; \Bigl| \hbit \E \hbit h_f(\xi + \delta_x,\sigma + \Tilde{\Sigma}_{\xi,\sigma,x}) -h_f(\xi,\sigma)  \Bigr| \leq 2B^{*}\norm{f}_{\infty}, \\
    \triangle_E\,h_f(\xi,\sigma) &= \sup\limits_{\substack{x\in\mcx\setminus\xi,\,y\in\xi \\
    \sigma_2\in \mfn_2\vert_{\langle\xi\setminus\{y\},x\rangle}
    % \sigma_2 \subset \langle\xi\setminus\{y\},x\rangle
    }} \Bigl| h_f(\xi + \delta_x,\sigma + \sigma_2 +\delta_{\{x,y\}}) - h_f\bigl(\xi + \delta_x,\sigma + \sigma_2 \bigr) \Bigr| \leq \norm{f}_{\infty}.
\end{align*}
\end{theorem}

For the proof of this theorem we need the following auxiliary lemma that gives an approximation of the edge term. 

\begin{lemma} \ \label{le: edge bounds with glauber } 
In the setting of Theorem~\ref{thm: general approx result}, we have
\begin{align*}
    \Bigl| \E \bigl[h_f\bigl(\xi + \delta_x, \sigma + \Tilde{\Sigma}_{\xi,\sigma,x} \bigr) \bigr] - \E \bigl[h_f\bigl(\xi + \delta_x, \sigma +\Tau_{\xi,x}\bigr) \bigr] \Bigr|
    \, \leq \, \triangle_E h_f(\xi,\sigma) \, \E\bigl[\norm{{\hat{v}(\Tilde{\Sigma}_{\xi,\sigma,x})}}_1 \bigr]. 
\end{align*}
\end{lemma}

\begin{proof} Fixing $(\xi,\sigma) \in \G$ with $\xi = \sum_{i=1}^n \delta_{x_i}$ and $x \in \mcx \setminus \xi$, we define the bijection $\Phi: \mfn_2\vert_{\langle\xi,x\rangle}\to \{0,1\}^{n}$, $\sigma' \mapsto (\sigma'(\{x_1,x\}),\ldots,\sigma'(\{x_n,x\}))^\top$ and set $Y_* = \Phi(\Tilde{\Sigma}_{\xi,\sigma,x})$ and $X_* = \Phi(\Tau_{\xi,x})$. 
% \begin{align*}
%     %&\bigg\vert\E\biggl(\int_{\mcx} \E \bigl[h_f\bigl(\Xi + \delta_x, \Sigma +\Tilde{\Sigma}_{\Xi,\Sigma,x} \bigr)\big\vert\, \Xi,\Sigma\bigr]-\E \bigl[h_f\bigl(\Xi + \delta_x, \Sigma +\Tau_{\Xi,x}\bigr)\big\vert\,\Xi,\Sigma\bigr]\,\lambda(x\mvert \Xi) \alpha(\diff x) \biggr)\bigg\vert\\
%     &\leq \int_{\mfn\times\mfn_2}\int_{\mcx} \big\vert\E \bigl[h_f\bigl(\xi + \delta_x, \sigma +\Tilde{\Sigma}_{\xi,\sigma,x} \bigr)\bigr]-\E \bigl[h_f\bigl(\xi + \delta_x, \sigma +\Tau_{\xi,x}\bigr)\bigr]\big\vert\,\lambda(x\mvert \Xi) \alpha(\diff x)\,P\otimes Q (\diff(\xi,\sigma)) \\
%     &= \int_{\mfn\times\mfn_2}\int_{\mcx} \big\vert\E \bigl[\hat{h}_f\bigl(\Tilde{\Sigma}_{\xi,\sigma,x} \bigr)\bigr]-\E \bigl[\hat{h}_f\bigl(\Tau_{\xi,x}\bigr)\bigr]\big\vert\,\lambda(x\mvert \Xi) \alpha(\diff x)\,P\otimes Q (\diff(\xi,\sigma)). 
% \end{align*}
Letting $$h(\cdot) = h_f\bigl(\xi + \delta_x, \sigma + \Phi^{-1}(\cdot) \bigr),$$
it is enough to study $\bigabs{\E h(Y_*) - \E h(X_*)}$ for bounded functions $h \colon \{0,1\}^n \to \R$.

We apply Stein's method using a Glauber dynamics Markov chain, which, at rate 1 in continuous time, picks an index $s \in [n]$ uniformly at random and sets the corresponding component to 1 with probability $\kappa(\{x,x_s\})$. This approach is similar as in the proof of Theorem~2.1 in \textcite{ReinertRoss2019}. However, note that a direct application of Theorem~2.1 is not possible as we might have $\kappa(\{x,x_s\})\in\{0,1\}$ and thus a non-irreducible Markov chain (on $\{0,1\}^n$). %However, using the independence of the components of $X$ we can adapt the proof to obtain a similar bound.

For $z\in\{0,1\}^n$, let $z^{(s,1)}$ and $z^{(s,0)}$ be the vectors having a $1$ and a $0$ in the $s$-th component, respectively, and being equal to $z$ otherwise. %Furthermore, define $\kappa_s:\{0,1\}^n\to [0,1]$ by $z\mapsto \kappa(\{x,x_s\})^{z_s}(1-\kappa(\{x,x_s\}))^{1-z_s}$.
Then the generator $\mathcal{A}$ of the Glauber dynamics Markov chain is given by 
$$\mathcal{A}f(z) = \frac{1}{n}\sum_{s\in [n]} \big[ \kappa(\{x,x_s\}) \Delta_s f(z) + \big(f(z^{(s,0)})-f(z)\big)\big],$$
where $\Delta_s f(z) := f(z^{(s,1)}) - f(z^{(s,0)})$. 
%A straightforward calculation yields $\E(\mathcal{A}h(X)) = 0$ for all bounded $h:\{0,1\}^n\to \R$.

Since the rate, the randomly picked index and the update probabilities of its component do not depend on the current state $z$, we obtain a straightforward coupling $(\hat{X}(t),\hat{Y}(t))_{t\geq 0}$ of two chains with generator $\mathcal{A}$ started in different states: simply construct both chains based on the same random sequences of holding times, indices, and updated components. It then remains to compute the expected coupling time. 

Assume that initially $m$ components are different, i.e.\ $\norm{\hat{X}(0)-\hat{Y}(0)}_1=m$. Let $\tau_m = \inf\{t\geq 0\mvert \norm{\hat{X}(t)-\hat{Y}(t)}_1 = m-1\}$ be the first time we have $m-1$ different components and $\tau^*_m = \inf\{t\geq 0\mvert \norm{\hat{X}(t)-\hat{Y}(t)}_1 = 0\}$ the coupling time under this starting condition. Given a jump occurs at time~$t$, the coupling difference $\norm{\hat{X}(t)-\hat{Y}(t)}_1$ decreases by $1$ with probability $\norm{\hat{X}(t-)-\hat{Y}(t-)}_1 \big/ n$ and remains unchanged otherwise (as we update the same component in the same way). Since the number $N(t)$ of jumps in $(0,t]$ is $\text{Poi}(t)$-distributed, we have
$$\Prob(\tau_m >t) = \sum_{i=0}^{\infty} \Prob(\tau_m >t\mvert N(t) = i) \hbit \Prob(N(t) = i) = \sum_{i=0}^{\infty} \bigg(\frac{n-m}{n}\bigg)^i \frac{1}{i!} e^{-t} t^i = e^{-tm/n}.$$
The strong Markov property yields then
$\E(\tau^*_m) = \E(\sum_{i=1}^m \tau_i) = \sum_{i=1}^m n/i <\infty$.

Now, using the coupling (and the finite expected coupling time) in a similar way as for Lemma~\ref{le: hf solves Stein equation}, it is straightforward to show that, for any bounded $h:\{0,1\}^n\to\R$, the function 
$$g_h(z):= -\int_0^{\infty} \E\bigl[h(\hat{X}(t)) \bigm| \hat{X}(0) = z \bigr] -\E[h(X_*)] \, \diff t$$
is well-defined and satisfies $\mathcal{A}g_h(z) = h(z) - \E[h(X_*)]$.

This allows us to apply \textcite[Lemma~2.4]{ReinertRoss2019}, which together with the above observations yields 
\begin{align*}
    \bigl| \E [h(Y)]-\E[h(X)] \bigr|
    &\leq \frac{1}{n} \sum_{s=1}^n \E\big[v_s(Y)\,\abs{\Delta_s g_h(Y)}\big] \\
    &\leq\frac{1}{n} \norm{\Delta_s h}_{\infty} \sum_{s=1}^n \E[v_s(Y)]\,\E(\tau^*_1) = \norm{\Delta_s h}_{\infty} \sum_{s=1}^n \E[v_s(Y)],
\end{align*}
where %$v(Y) = \big(v_1(Y),\ldots,v_n(Y)\big)$ with components
$v_s(y) = \big\vert \kappa(\{x,x_s\}) - \Prob(Y_s = 1\mvert (Y_j)_{j\neq s} = (y_j)_{j\neq s})\big\vert$. For the second inequality, we use an instance $(\hat{X}(t), \hat{Y}(t))_{t \geq 0}$ of the above coupling started at $(\hat{X}(0), \hat{Y}(0)) = (Y^{(s,1)}, Y^{(s,0)})$. Hence,
\begin{align*}
    \abs{\Delta_s g_h(Y)} &= \biggabs{ \int_0^{\infty} \E\bigl[h(\hat{X}(t))\bigr] - \E\bigl[h(\hat{Y}(t)) \bigr] \diff t } \\
    &\leq  \int_0^{\infty} \E\bigl[\abs{ h(\hat{X}(t)) - h(\hat{Y}(t))} \hbit \1{\{\tau_1^* >t\}} \bigr] \diff t 
    \leq\norm{\Delta_s h}_{\infty}\,\E(\tau^*_1).
\end{align*}
% where in the last step we used that the process~$\hat{X}$ updates components independently of each other which yields that the coupling time of the $s$-th component is independent of all other (already coupled) components of the process~$\hat{X}$. 

Setting $\hat{v}(\cdot) = (\hat{v}_1(\cdot),\ldots,\hat{v}_n(\cdot))$ with $\hat{v}_s = v_s \circ \Phi$ yields in total
\begin{align*}
    \Bigl| \E \bigl[h_f\bigl(\xi + \delta_x, \sigma + \Tilde{\Sigma}_{\xi,\sigma,x} \bigr) \bigr] - \E \bigl[h_f\bigl(\xi + \delta_x, \sigma +\Tau_{\xi,x}\bigr) \bigr] \Bigr| 
    &= \bigl| \E[h(Y)]-\E[h(X)] \bigr|\\
    &\leq \norm{\Delta_s h}_{\infty} \sum_{s=1}^n \E[v_s(Y)] \\
    &\leq \triangle_E h_f (\xi,\sigma) \hbit \E\big[\norm{\hat{v}(\tilde{\Sigma}_{\xi,\sigma,x})}_1\big]
\end{align*}
since
\begin{align*}
    \sup_{s} \, \norm{\Delta_s h}_{\infty} = \sup_{\substack{y\in \xi \\ \sigma_2\in \mfn_2\vert_{\langle \xi\setminus\{y\},x\rangle}}} \big\vert h_f(\xi+\delta_x, \sigma + \sigma_2 + \delta_{\{y,x\}}) -  h_f(\xi+\delta_x, \sigma + \sigma_2 )\big\vert\leq \triangle_E h_f (\xi,\sigma).
\end{align*}
\vspace*{-10mm}

\end{proof}

\begin{proof}[Proof of Theorem~\ref{thm: general approx result}]
Applying Lemma~\ref{le: hf solves Stein equation}, Theorem~\ref{thm:graph_gnz} and the triangle inequality, we obtain 
\begin{align*}
    \big\vert& \E \bigl(f(\Xi,\Sigma)\bigr)-\E\bigl(f(\Eta,\Tau)\bigr)\big\vert = \big\vert \E\bigl(\mcg^{\lambda,\kappa}h_f(\Xi,\Sigma)\bigr)\big\vert\\
    &=\biggl| \int_{\mfn\times\mfn_2} \int_{\mcx} \E \bigl[ h_f(\xi + \delta_x, \sigma +\Tau_{\xi,x}) - h_f(\xi,\sigma) \bigr] \, \lambda(x\mvert \xi) \, \alpha(\diff x) \, P \otimes Q(\diff\xi \diff \sigma) \\
    &\hspace*{4mm} {} + \int_{\mfn\times\mfn_2} \int_{\mcx} \E\bigl[ h_f(\xi,\sigma)-h_f(\xi+\delta_x,\sigma  +\Tilde{\Sigma}_{\xi,\sigma,x}) \bigr] \, L(\diff x \,\vert \xi) \, P \otimes Q(\diff\xi \diff \sigma) \biggr| \\
    &\hspace*{4mm} {} + \biggl| \int_{\mfn} \int_{\mcx} \int_{\mfn_2} \!\E\bigl[h_f\bigl(\xi, \sigma \bigr)-h_f\bigl(\xi+\delta_x,\sigma+\Tilde{\Sigma}_{\xi,\sigma,x}\bigr)\bigr] \, \bigl( \tilde{Q}_1(\xi, x;\, d\sigma) - Q(\xi, d\sigma) \bigr) \, L(dx \mvert \xi) \, P(d\xi) \biggr|.
\end{align*}
The second absolute value is directly seen to be bounded by the last term in Equation~\ref{eq: general approx result}. The first absolute value term can be bounded further by
\begin{align*}
    \biggl| &\int_{\mfn\times\mfn_2} \int_{\mcx} \E\bigl[h_f\bigl(\xi + \delta_x, \sigma +\Tilde{\Sigma}_{\xi,\sigma,x} \bigr) - h_f(\xi,\sigma) \bigr] \bigl( \lambda(x\mvert \xi) \alpha(\diff x) -L(\diff x\,\vert\,\xi) \bigr) \, P \otimes Q(\diff\xi \diff \sigma) \biggr| \\
    &\hspace*{4mm} {} +\biggl| \int_{\mfn\times\mfn_2} \int_{\mcx} \E \bigl[h_f(\xi + \delta_x, \sigma +\Tilde{\Sigma}_{\xi,\sigma,x}) - h_f(\xi + \delta_x, \sigma +\Tau_{\xi,x}) \bigr]\,\lambda(x\mvert \xi) \, \alpha(\diff x) \, P \otimes Q(\diff\xi \diff \sigma) \biggr| \\
    % \leq &  \bigg\vert \E\biggl(\int_{\mcx} \E\bigl[h_f\bigl(\Xi + \delta_x, \Sigma +\Tilde{\Sigma}_{\Xi,\Sigma,x} \bigr) - h_f(\Xi,\Sigma)\big\vert\,\Xi,\Sigma\bigr] \bigl[\lambda(x\mvert \Xi) \alpha(\diff x) -L(\diff x\,\vert\,\Xi)\bigr]\biggr)\bigg\vert\\
    % & +\bigg\vert\E\biggl(\int_{\mcx} \E \bigl[h_f\bigl(\Xi + \delta_x, \Sigma +\Tilde{\Sigma}_{\Xi,\Sigma,x} \bigr)\big\vert \Xi,\Sigma\bigr]-\E \bigl[h_f\bigl(\Xi + \delta_x, \Sigma +\Tau_{\Xi,x}\bigr)\big\vert\,\Xi,\Sigma\bigr]\,\lambda(x\mvert \Xi) \alpha(\diff x) \biggr)\bigg\vert\\
    % &+\bigg\vert \int_{\mfn} \int_{\mcx} \int_{\mfn_2} \E\bigl[h_f\bigl(\xi, \sigma\bigr)-h_f\bigl(\xi+\delta_x,\sigma+\Tilde{\Sigma}_{\xi,\sigma,x}\bigr)\bigr] \; \bigl( \tilde{Q}_1(\xi, x;\, d\sigma) - Q(\xi, d\sigma) \bigr) \; L(dx \mvert \xi) \; P(d\xi)\bigg\vert\\
    &\leq \E\biggl(\int_{\mcx}\triangle_V h_f(\Xi,\Sigma)\, \big\vert\lambda(x\mvert \Xi) \alpha(\diff x)-L(\diff x\vert\Xi)\big\vert \biggr) \\
    &\hspace*{4mm} {} + \E\biggl(\int_{\mcx}  \triangle_E h_f(\Xi,\Sigma)\, \E\bigl[\norm{{\hat{v}(\Tilde{\Sigma}_{\Xi,\Sigma,x})}}_1 \,\vert\,\Xi,\Sigma\bigr] \,\lambda(x\mvert \Xi) \, \alpha(\diff x)
    \biggr)
    % &+ \int_{\mfn}\int_{\mcx}\int_{\mfn_2}\triangle_V h_f(\xi,\sigma)\,\big\vert\tilde{Q}_1(\xi,x;\diff \sigma) - Q(\xi;\diff \sigma)\big\vert\,L(\diff x \,\vert \xi) P(\diff \xi)
\end{align*}
using Lemma~\ref{le: edge bounds with glauber } for the second term.\\
Finally, the Stein factors can be estimated since setting $(\Xi,\Sigma) = (\eta,\tau)$ in Inequality~\eqref{eq: hf well-defined} yields
\begin{align*}
    &\big\vert h_f\bigl(\eta,\tau \bigr) -h_f\bigl(\xi,\sigma \bigr) \big\vert= \bigg\vert \int_0^{\infty} \E\bigl[f\big((\Eta_s,\Tau_s)^{(\xi,\sigma)}\big)] -\E\bigl[f\big((\Eta_s,\Tau_s)^{(\eta,\tau)}\big)]\diff s \bigg\vert
    \leq 2\,\norm{f}_{\infty}  \E \tau_{(\xi,\sigma),(\eta,\tau)}
\end{align*}
for all $(\xi,\sigma),(\eta,\tau)\in\mcg$. 
Thus, we obtain $\triangle_{V} h_f \leq 2B^{*}\norm{f}_{\infty}$ by Theorem~\ref{thm: coupling time} and $\triangle_{E} h_f \leq \norm{f}_{\infty}$ by \eqref{eq: coupling time for pure edge diff} with $\varrho((\xi,\sigma),(\xi,\tau))=2$.
\end{proof}

\subsection{Wasserstein bounds} \label{ssec: bounding stein factors}

As an immediate corollary of Theorem~\ref{thm: general approx result}, we obtain that the total variation distance between $P \otimes Q$ and $P^{\lambda} \otimes Q^{\kappa}$ is bounded by the right hand side of~\eqref{eq: general approx result} with Stein factors $\triangle_V\,h_f \leq B^*$ and  $\triangle_E\,h_f \leq 1/2$ (uniformly in the graph). This is seen by setting $\mcf = \{f \colon \G \to \R \text{ measurable} \mvert \norm{f}_{\infty} = 1/2 \}$. 

However, the total variation metric is very strong and it is known from other situations (e.g.\ \cite[Section 3]{BarbourBrown1992}) that substantially better bounds for the Stein factors can be obtained if we use a Wasserstein metric which takes the metric structure of the space $\mcx$ into account. In the present subsection, we focus on the Wasserstein metrics $W_{\G,i}$ introduced in Section~\ref{ch: metric}, with underlying GOSPA metrics $d_{\mathbb{G},i} \leq C_i$ for $i=1,2$. 
To lighten the notation, write $\mcf_i = \mcf_{\G,i}$, $i=1,2$. By the bounds on the underlying metrics, we can restrict the function classes $\mcf_i$ to $[-C_i/2,C_i/2]$-valued functions without changing the Wasserstein metric $W_{\G,i}$. Thus we obtain an immediate bound from Theorem~\ref{thm: general approx result} again with Stein factors $\triangle_V\,h_f \leq C_i \hbit B^*$ and $\triangle_E\,h_f \leq \frac{C_i}{2}$.

% \begin{corollary}
% Let $(\Eta,\Tau)\sim P^{\lambda}\otimes Q^{\kappa}$ and $(\Xi,\Sigma)$ a random graph having law $P\otimes Q$ where $Q$ is an edge kernel and $\Xi\sim P$ fulfils condition~\eqref{eq:condsigma} and has Papangelou kernel $L$. Then we have
% \begin{align*}
%     W_{\G,i}(P \otimes Q, P^{\lambda} \otimes Q^{\kappa})
%     &\leq C_i \hbit d_{\mathrm{TV}}(P \otimes Q, P^{\lambda} \otimes Q^{\kappa}) \\
%     & \leq \E\biggl(\int_{\mcx} K_V \bigl|\lambda(x\mvert \Xi) \alpha(\diff x)-L(\diff x\vert\Xi)\bigr| \biggr) \\
%     &\hspace*{5mm} + \E\biggl(\int_{\mcx} K_E \hbit \E\bigl[\norm{{\hat{v}(\Tilde{\Sigma}_{\Xi,\Sigma,x})}}_1 \bigm| \Xi,\Sigma\bigr] \,\lambda(x\mvert \Xi) \,\alpha(\diff x) \biggr) \\
%     &\hspace*{5mm} + \int_{\mfn}\int_{\mcx}\int_{\mfn_2} K_V \abs{\tilde{Q}_1(\xi,x;\diff \sigma) - Q(\xi;\diff \sigma)}\,L(\diff x \,\vert \xi) \, P(\diff \xi),
% \end{align*}
% where $K_V = C_i \hbit B^*$, $K_E = \frac{C_i}{2}$ and
% \begin{align*}
%   \Tilde{\Sigma}_{\xi,\sigma,x}\sim \tilde{Q}_{2 \mvert 1}(\xi,x,\sigma ;\, \cdot)
%   &= \mathcal{L} \bigl( \Sigma \vert_{\langle \xi, x \rangle} \bigm| \Xi = \xi+\delta_x, \Sigma \vert_{\xi^{\atwo}} = \sigma \bigr) \\[0.5mm]
%   \tilde{Q}_1(\xi, x;\, \cdot) &= \mathcal{L} \bigl( \Sigma \vert_{\xi^{\langle 2\rangle}} \bigm| \Xi = \xi+\delta_x\bigr) \\[-0.5mm]
%   \norm{\hat{v}(\tilde{\Sigma}_{\xi,\sigma,x})}_1 &= \sum_{i=1}^{\abs{\xi}} \Bigl| \kappa(x_i,x) - \Prob\bigl(\Tilde{\Sigma}_{\xi,\sigma,x}(\{x,x_i\}) = 1 \bigm| (\Tilde{\Sigma}_{\xi,\sigma,x}(\{x,x_j\}))_{j\neq i}\bigr) \Bigr|.
% \end{align*}
% \end{corollary}

The goal of the present subsection is to improve on this bound. For Poisson process approximation without edges, it is known that essentially a factor of order $\Eta(\mcx)^{-1}$ (often times a log-term) can be gained by using a Wasserstein metric; see \textcite{BarbourBrown1992} and \textcite{sx2008}. In the random graph setting, something similar is possible, as we show next.

\begin{theorem} \label{thm: Poisson approx in Wasserstein}
Let $(\Eta,\Tau)\sim P^{\lambda}\otimes Q^{\kappa}$ be a generalised random geometric graph with $\lambda(\cdot \mvert \xi) = \lambda(\cdot)$ integrable and therefore $\Eta \sim \pop(\lambda)$. Let $(\Xi,\Sigma)$ be a random graph having law $P\otimes Q$, where $Q$ is an edge kernel and $\Xi\sim P$ fulfils Condition \eqref{eq:condsigma} and has Papangelou kernel $L$. Then, in the notation of Theorem~\ref{thm: general approx result}, 
\begin{align*}
    W_{\mathbb{G},i}(P\otimes Q, P^{\lambda}\otimes Q^{\kappa})
    &\leq c_V(\lambda)\,\E\biggl(\int_{\mcx}\big\vert\lambda(x) \alpha(\diff x)-L(\diff x\vert\Xi)\big\vert \biggr) \\
    &\hspace*{4mm} {} + c_E(\lambda)\, \int_{\mcx}  \E\bigl[\norm{{\hat{v}(\Tilde{\Sigma}_{\Xi,\Sigma,x})}}_1 \bigr] \,\lambda(x) \, \alpha(\diff x) \\
    &\hspace*{4mm} {} + c_V(\lambda)\, \E \biggl( \int_{\mcx}\int_{\mfn_2} \bigabs{\tilde{Q}_1(\Xi,x;\diff \sigma) - Q(\Xi;\diff \sigma)}\,L(\diff x \,\vert \Xi) \biggr),
\end{align*}
where 
\begin{align}
    c_V(\lambda) &= \min \biggl\{ C_i, % B^{*}
    \frac{1}{\Lambda} \,\bigl(1 +  ({1-e^{-{\Lambda}}}) \log^+{\Lambda} \bigr)\,\tilde{C}_i \biggr\},\label{eq: SteinPoissonV}\\
    c_E(\lambda) &= \min \biggl\{ \frac14, \frac{2-e^{-\Lambda}}{\Lambda} - \frac{1}{\Lambda^2}\Bigl(\frac{3}{2} - e^{-\Lambda}\Bigr) \biggr\} \hbit C_E,\label{eq: SteinPoissonE}
    %&\overline{\triangle}_E\,h_f(\xi,\sigma)\leqc_E(\lambda) = \frac{1}{\Lambda} \,\bigl(1 +  ({1-e^{-\Lambda}}) \log^+{\Lambda} \bigr)\,C_E.
\end{align}
using $\Lambda = \int_{\mcx} \lambda(x) \, \alpha(dx)$. Here $C_i = C_V + \frac{i}{2} C_E$ denotes the upper bound on $d_{\G,i}$, whereas $\tilde{C}_i = C_V +  i \hbit C_E$ is the maximal penalty per additional vertex; see Subsection~\ref{ch: metric}.
\end{theorem}
\begin{remark} \label{re: Stein factor for general process with GOSPA}
  Compared to the general result, the bound above has better and more explicit constants and a rate improvement by a factor of order $\log(\Lambda)/\Lambda$ as $\Lambda \to \infty$.\\
  While it is possible to adapt the proof below for $P^{\lambda'} \otimes Q^{\kappa}$ in a neighbourhood of a Poisson random geometric graph $P^{\lambda} \otimes Q^{\kappa}$, this requires considerably more work and the  improvement factor deteriorates quickly in the $L^1$-distance between $\lambda'$ and $\lambda$. In fact, we are not able to obtain a better rate in this way than by simply bounding 
  \begin{align*}
    W_{\mathbb{G},i}(P\otimes Q, P^{\lambda'}\otimes Q^{\kappa}) \leq W_{\mathbb{G},i}(P\otimes Q, P^{\lambda}\otimes Q^{\kappa}) + W_{\mathbb{G},i}(P^{\lambda'}\otimes Q^{\kappa}, P^{\lambda}\otimes Q^{\kappa})
  \end{align*}
  and applying Theorem~\ref{thm: Poisson approx in Wasserstein} to both terms.\\
  However, we do obtain a somewhat better constant for the edge factor even if $P^{\lambda'}$ is not necessarily close to a Poisson process. Adapting the argument below for deriving the first terms in the minima of \eqref{eq: SteinPoissonV} and \eqref{eq: SteinPoissonE}, we obtain the bounds
  \begin{align*}
   \triangle_V\,h_f \leq C_i \hbit B^* \quad \text{and} \quad \triangle_E\,h_f \leq \frac{C_E}{4} %\leq \frac{C_i}{2}.
   \end{align*}
   in Theorem~\ref{thm: general approx result} for $f \in \mcf_{i}$.
\end{remark}

\begin{proof}[Proof of Theorem~\ref{thm: Poisson approx in Wasserstein}] 
Let $(\xi,\sigma)\in\G$ and $x\in\mcx\setminus\xi$ as well as $\sigma_2\in\mfn\vert_{\langle \xi,x\rangle}$. For a general conditional intensity $\lambda$ the coupling from Section~\ref{ch: coupling} yields
\begin{align*}
    &\big\vert h_f\bigl(\xi + \delta_x,\sigma + \sigma_2\bigr) -h_f\bigl(\xi,\sigma \bigr)  \big\vert\\
    %&= \bigg\vert  \int_0^{\infty} \E\bigl[f\big((\Eta_s,\Tau_s)^{(\xi,\sigma)}\big)] -\E\bigl[f\big((\Eta_s,\Tau_s)^{(\xi+\delta_x,\sigma + {\sigma}_2)}\big)]\diff s \bigg\vert\\
    &= \bigg\vert  \int_0^{\infty} \E\big[f\big(\Eta_s^{(\xi)},\Tau_s^{(\xi,\sigma)}\big)-f\bigl(\Eta_s^{(\xi+\delta_x)},\Tau_s^{(\xi + \delta_x,\sigma+{\sigma_2})}\bigr)\big]\diff s \bigg\vert\\
    %&\leq \int_{\mfn_2} \int_0^{\infty} \E\big[\big\vert f\big(\Eta_s^{(\xi)},\Tau_s^{(\xi,\sigma)}\big)-f\bigl(\Eta_s^{(\xi+\delta_x)},\Tau_s^{(\xi + \delta_x,\sigma+\tilde{\sigma_2})}\bigr)\big\vert \1_{\{\tau_{(\xi,\sigma),(\xi+\delta_x,\sigma+\Tilde{\sigma}_2)} >s\}}\big]\diff s\,
    %\tilde{Q}_{2\mvert 1}(\xi,x,\sigma,\diff \tilde{\sigma}_2) \bigg\vert\\
    &\leq  \int_0^{\infty} \E\big[d_{\mathbb{G},i}\big((\Eta_s^{(\xi)},\Tau_s^{(\xi,\sigma)}),(\Eta_s^{(\xi+\delta_x)},\Tau_s^{(\xi + \delta_x,\sigma+{\sigma_2})})\big)\, \1{\{\tau_{(\xi,\sigma),(\xi+\delta_x,\sigma+{\sigma}_2)} >s\}}\big]\diff s, 
\end{align*}
where the last inequality follows from the Lipschitz continuity of the test functions $f\in\mcf_i$, $i=1,2$ (see Section~\ref{ch: metric}).
In particular, using $d_{\mathbb{G},i}\leq C_i$, this implies $\triangle_V\,h_f \leq C_i \hbit B^*$.  

As we are in the special case of $\lambda(\cdot\mvert \xi) = \lambda(\cdot)$, the coupling takes the simpler form from Remark~\ref{re: coupling in Poi case and edge case}a with $(\zeta,\varepsilon) = (\xi ,\sigma)$. Thus the coupling time $\tau_{(\xi,\sigma),(\xi+\delta_x,\sigma+{\sigma}_2)} \overset{d}{=} L_x$ is independent of $(\Eta_s^{(\xi)},\Tau_s^{(\xi,\sigma)})$ and therefore
\begin{align} \label{eq: proof of stein factors 1} 
    \notag &\int_0^{\infty} \E\big[d_{\mathbb{G},i}\big((\Eta_s^{(\xi)},\Tau_s^{(\xi,\sigma)}),(\Eta_s^{(\xi+\delta_x)},\Tau_s^{(\xi + \delta_x,\sigma+{\sigma_2})})\big)\, \1{\{\tau_{(\xi,\sigma),(\xi+\delta_x,\sigma+{\sigma}_2)} >s\}}\big]\diff s\\
    &= \int_0^{\infty} \E\big[ d_{\G,i}\big(\big(\Eta_s^{(\xi)},\Tau_s^{(\xi,\sigma)}\big),\big(\Eta_s^{(\xi)}+\delta_x,\Tau_s^{(\xi,\sigma)} + {\sigma}_2\vert_{\langle {\xi},x\rangle}+ \Tau_{\Eta_s^{(\xi)}\setminus\xi,x}\big)\big)\big]\;\Prob(L_x>s)\;\diff s.    
\end{align}

Applying now $d_{\mathbb{G},i}\leq C_i$ yields $\triangle_V\,h_f \leq C_i$ and hence the first part of the minimum. For the second part, note that
\begin{align*} 
&\E\big[ d_{\G,i}\big(\big(\Eta_s^{(\xi)},\Tau_s^{(\xi,\sigma)}\big),\big(\Eta_s^{(\xi)}+\delta_x,\Tau_s^{(\xi,\sigma)} + {\sigma}_2\vert_{\langle {\xi},x\rangle}+ \Tau_{\Eta_s^{(\xi)}\setminus\xi,x}\big)\big)\big]
%i=1 case
%&= \E\bigg[ \frac{1}{\abs{\Eta_s^{(\xi)}}+1}\bigg(C_1 + \frac{1}{2} \frac{1}{\abs{\Eta_s^{(\xi)}}} 2\sum_{y\in \Eta_s^{(\xi)}} C_E \1_{\{\abs{\Eta_s^{(\xi)}}\geq 1\}}\bigg)\bigg]\\
%i=2 case
%&= \E\bigg[ \frac{1}{\abs{\Eta_s^{(\xi)}}+1}\bigg(C_2 + \frac{1}{2} \frac{1}{\abs{\Eta_s^{(\xi)}}} 2\sum_{y\in \Eta_s^{(\xi)}} C_E\,d_E(({\sigma}_2\vert_{\langle \Eta_s^{(\xi)},x\rangle}+\Tau_{\Eta_s^{(\xi)}\setminus\xi,x})(\{x,y\}),0)\, \1_{\{\abs{\Eta_s^{(\xi)}}\geq 1\}}\bigg)\bigg]\\
\leq \E\bigg[ \frac{1}{\abs{\Eta_s^{(\xi)}}+1}\bigg]\,\Tilde{C}_i.%\big(C_i +  C_E\big). 
\end{align*}
Applying the coupling of Remark~\ref{re: coupling in Poi case and edge case}a with $\zeta = \emptyset$ yields $H_s^{(\xi)} \overset{\text{d}}{=} H_s^{(\emptyset)} + \sum_{z\in\xi} \delta_z\,\1{\{L_z>s\}}$ for every $s\geq0$. Since $\abs{H_s^{(\emptyset)}} = H_s^{(\emptyset)}(\mcx)\sim \text{Poi}(\lambda_s)$ with $\lambda_s = (1-e^{-s}) \Lambda$ (see Proposition~$3.5$ of \cite{xia2005}), 
\begin{align*}
    \E\bigg[\frac{1}{\abs{H_s^{(\xi)}}+1}\bigg]
    \leq \E\bigg[\frac{1}{\abs{H_s^{(\emptyset)}}+1}\bigg]
    = \sum_{k=0}^{\infty} \frac{1}{k+1} \frac{\lambda_s^k}{k!} e^{-\lambda_s} 
    = \frac{1}{\lambda_s}\sum_{k=1}^{\infty} \frac{\lambda_s^k}{k!} e^{-\lambda_s} 
    = \frac{1-e^{-\lambda_s}}{\lambda_s}.
\end{align*}
Finally, plugging in the above inequalities 
and using that $L_x\sim \text{Exp}(1)$, the right-hand side of Equation~\eqref{eq: proof of stein factors 1} can be bounded by 
\begin{align*}
    %\int_0^{\infty} \E\bigg[\frac{1}{\abs{H_s^{(\xi)}}+1}\bigg]\,\Prob(L_x\geq s)\diff s
    %&\leq  
    \tilde{C}_i\int_0^{\infty} \frac{1-e^{-\lambda_s}}{\lambda_s}\,e^{-s}\diff s
    %&\hspace{-0.45cm}\overset{s = 1-e^{-s}}{= }
    = \tilde{C}_i\int_0^{1}\frac{1-e^{-\Lambda s}}{\Lambda s}\diff s
    &\leq 
 \int_0^{\frac{1}{\Lambda}} \tilde{C}_i \diff s + \int_{\frac{1}{\Lambda}}^1  \tilde{C}_i\frac{1-e^{-\Lambda}}{\Lambda s} \diff s \\
    %&= \bigl(\frac{1}{\Lambda} +  \frac{1-e^{-\Lambda}}{\Lambda} \log{\Lambda} \bigr)\\
    &= \frac{\tilde{C}_i}{\Lambda}\bigl(1 +  ({1-e^{-\Lambda}}) \log{\Lambda} \bigr)
\end{align*}
for $\Lambda\geq 1$, where the first equality follows by substituting $s$ for $1-e^{-s}$. This yields the second part of the minimum in~\eqref{eq: SteinPoissonV}.

It remains to bound the edge term $\triangle_E$.
Let $f\in\mcf_i$ and fix $(\xi,\sigma)\in\G$, $x \in \mcx \setminus \xi$, $y\in\xi$ as well as $\sigma_2\in\mfn_2\vert_{\langle \xi\setminus\{y\},x\rangle}$.
Lemma~\ref{le: hf solves Stein equation} and Remark~\ref{re: coupling in Poi case and edge case}b with $(\zeta,\varepsilon) = (\xi + \delta_x,\sigma +\sigma_2)$ yield
\begin{align} \label{eq: proof of stein factors 2}
&\notag \big\vert h_f\bigl(\xi + \delta_x,\sigma + \sigma_2 +\delta_{\{x,y\}} \bigr)- h_f\bigl(\xi + \delta_x,\sigma + \sigma_2 \bigr)\big\vert\\
%&= \bigg\vert \int_0^{\infty} \E\bigl[f\big((\Eta_s,\Tau_s)^{(\xi + \delta_x,\sigma + \sigma_2)}\big)] -\E\bigl[f\big((\Eta_s,\Tau_s)^{(\xi + \delta_x,\sigma + \sigma_2 + \delta_{\{x,y\}})}\big)]\diff s \bigg\vert\\
&\notag= \bigg\vert \int_0^{\infty} \E\big[f\big(\Eta_s^{(\xi + \delta_x)},\Tau_s^{(\xi + \delta_x,\sigma + \sigma_2)}\big) -f\bigl(\Eta_s^{(\xi + \delta_x)},\Tau_s^{(\xi + \delta_x,\sigma + \sigma_2)} + \delta_{\{x,y\}} \1{\{L_x>s\}} \hbit \1{\{L_y>s\}}
%\,\1_{\{x,y \in \Eta_s^{(\xi + \delta_x)} \cap (\xi + \delta_x) \}}
\bigr)\big]\diff s\bigg\vert\\
%&= \bigg\vert \int_0^{\infty} \E\big[f\big(\Eta_s^{(\xi + \delta_x)},\Tau_s^{(\xi + \delta_x,\sigma + \sigma_2)}\big) -f\bigl(\Eta_s^{(\xi + \delta_x)},\Tau_s^{(\xi + \delta_x,\sigma + \sigma_2)} + \delta_{\{x,y\}}\,\1_{\{L_y>s\}}\bigr)\big]\,\Prob(L_x>s)\diff s\bigg\vert\\
%&\leq  \int_0^{\infty} \E\big[\big\vert f\big(\Eta_s^{(\xi + \delta_x)},\Tau_s^{(\xi + \delta_x,\sigma + \sigma_2)}\big) -f\bigl(\Eta_s^{(\xi + \delta_x)},\Tau_s^{(\xi + \delta_x,\sigma + \sigma_2)} + \delta_{\{x,y\}}\bigr)\big\vert\,\1_{\{x,y \in \Eta_s^{(\xi + \delta_x)} \cap (\xi + \delta_x) \}}\big]\diff s\\
%&\leq \int_0^{\infty}\E\big[d_{\G,i}\big(\big(\Eta_s^{(\xi + \delta_x)},\Tau_s^{(\xi + \delta_x,\sigma + \sigma_2)}\big),\big(\Eta_s^{(\xi + \delta_x)},\Tau_s^{(\xi + \delta_x,\sigma + \sigma_2)} + \delta_{\{x,y\}}\big)\big)\,\1_{\{x,y \in \Eta_s^{(\xi + \delta_x)} \cap (\xi + \delta_x) \}}\big] \diff s .\\
&\notag\leq \int_0^{\infty}\E\big[d_{\G,i}\big(\big(\Eta_s^{(\xi + \delta_x)},\Tau_s^{(\xi + \delta_x,\sigma + \sigma_2)}\big),\big(\Eta_s^{(\xi + \delta_x)},\Tau_s^{(\xi + \delta_x,\sigma + \sigma_2)} + \delta_{\{x,y\}}\big)\big)\,\1{\{L_y >s \}}\big]\Prob(L_x>s) \diff s\\
&=\int_0^{\infty} \E\bigg[\frac{1}{\abs{\Eta_s^{(\xi)}}+1}\frac{1}{\abs{\Eta_s^{(\xi)}}} \,C_E\,\1{\{L_y>s\}}\bigg]\Prob(L_x>s) \diff s,
\end{align}
where the inequality follows from the Lipschitz continuity of $f$ and the independence of $L_x$ and~$L_y$.
We obtain $\triangle_E\,h_f \leq \frac{C_E}{4}$, the first part of the minimum in~\eqref{eq: SteinPoissonE}, since $L_y > s$ implies $\abs{\Eta_s^{(\xi)}}\geq 1$ and $L_x,L_y \sim \text{Exp}(1)$. Note that this result holds for a general intensity function $\lambda$.

For the second part of the minimum under the assumption $\lambda(\cdot\mvert \xi) = \lambda(\cdot)$, we apply again the coupling from Remark~\ref{re: coupling in Poi case and edge case}a with $\zeta = \emptyset$, yielding
$$H_s^{(\xi)} \overset{\text{d}}{=} H_s^{(\emptyset)} + \sum_{z\in\xi} \delta_z\,\1{\{L_z>s\}}= H_s^{(\emptyset)} + \sum_{z\in\xi\setminus \{y\}} \delta_z\,\1{\{L_z>s\}} + \delta_y\,\1{\{L_y>s\}}$$
for every $s\geq0$. By $\abs{H_s^{(\emptyset)}} \sim \text{Poi}(\lambda_s)$, we obtain 
\begin{align*}
    \E\bigg[\frac{1}{\abs{\Eta_s^{(\xi)}}\!+\!1}\frac{1}{\abs{\Eta_s^{(\xi)}}} \1{\{L_y>s\}}\!\bigg]
    \leq \E\bigg[\frac{1}{\abs{\Eta_s^{(\emptyset)}}\!+\!2}\frac{1}{\abs{\Eta_s^{(\emptyset)}}\!+\!1} \bigg]\,\Prob(L_y>s)
    %= \sum_{k=0}^{\infty} \frac{1}{k+2}\frac{1}{k+1} \frac{\lambda_s^k}{k!} e^{-\lambda_s} 
    %= \frac{1}{\lambda_s^2}\sum_{k=2}^{\infty} \frac{\lambda_s^k}{k!} e^{-\lambda_s} 
    = \frac{1-e^{-\lambda_s}(1+\lambda_s)}{\lambda_s^2}\,\Prob(L_y>s).
\end{align*}
Thus for $\Lambda\geq 1$ the right-hand side of Equation~\eqref{eq: proof of stein factors 2} can be bounded by 
    \begin{align*}
    %& \big\vert h_f\bigl(\xi + \delta_x,\sigma + \sigma_2 +\delta_{\{x,y\}} \bigr)- h_f\bigl(\xi + \delta_x,\sigma + \sigma_2 \bigr)\big\vert\\
    %&\leq  \int_0^{\infty} \E\bigg[\frac{1}{\abs{\Eta_s^{(\xi)}}+1}\frac{1}{\abs{\Eta_s^{(\xi)}}} C_E \, \1_{\{L_y>s\}}\bigg]\,\Prob(L_x>s)\diff s\\
    %&\leq  \int_0^{\infty} C_E\, \frac{1-e^{-\lambda_s}(1+\lambda_s)}{\lambda_s^2}\,\Prob(L_y>s)\,\Prob(L_x>s)\diff s\\
    &  C_E\,\int_0^{\infty}\frac{1-e^{-\lambda_s}(1+\lambda_s)}{\lambda_s^2} \,e^{-2s} \diff s \\
    &= C_E\,\int_0^{1}\frac{1-e^{-\Lambda s}(1+\Lambda s)}{\Lambda^2s^2}\, (1-s) \diff s \\
    %&\leq C_E\,\biggl[\int_0^{\frac{1}{\Lambda }}\frac{\Lambda s-\Lambda s\,e^{-\Lambda s}}{\Lambda^2s^2}\, (1-s) \diff s+  \int_{\frac{1}{\Lambda }}^{1}\frac{1-e^{-\Lambda s}}{\Lambda ^2s^2}\, (1-s) \diff s - \int_{\frac{1}{\Lambda }}^{1}\frac{e^{-\Lambda s}}{\Lambda s}\, (1-s) \diff s\biggr] \\
        &\leq C_E\,\biggl[\int_0^{\frac{1}{\Lambda }}\frac{1-e^{-\Lambda s}}{\Lambda s}\, (1-s) \diff s+  \int_{\frac{1}{\Lambda }}^{1}\frac{1-e^{-\Lambda s}}{\Lambda ^2s^2}\,  \diff s - \int_{\frac{1}{\Lambda }}^{1}\frac{e^{-\Lambda s}}{\Lambda s}\, (1-s) \diff s\biggr]  \\
        &\leq C_E\,\biggl[\int_0^{\frac{1}{\Lambda }}(1-s) \diff s +\int_{\frac{1}{\Lambda }}^{1}\frac{1-e^{-\Lambda }}{\Lambda^2s^2} \diff s - 0\biggr] \\
        %&=C_E\, \biggl[\frac{1}{\Lambda } - \frac{1}{2\Lambda^2} + \frac{1-e^{-\Lambda }}{\Lambda^2}[\Lambda -1]\biggr] \\
        &= C_E\,\biggl[\frac{2-e^{-\Lambda }}{\Lambda } - \frac{1}{\Lambda ^2}\biggl(\frac{3}{2} - e^{-\Lambda }\biggr)\biggr],
    \end{align*}
where the first equality follows by substituting $s$ for $1-e^{-s}$ and the first inequality uses $e^{-\Lambda s} \geq 1 -\Lambda s$ to obtain the first term.
This yields the claim.
\end{proof}

\section{Applications} \label{sec: applications}

In the following, we apply our results to the approximation of the percolation graph of large balls in a Boolean model and to discretising a generalised random geometric graph. The former application is relevant, among other things, in telecommunications, for studying the connectivity network of stations (cell towers, participants in a peer-to-peer mobile network, etc.) with high transmitting power. Discretisation of random graphs can be used to treat problems, e.g.\ regarding percolation and coverage, for continuous random graphs, typically by discretising the graph to a regular lattice and then using the lattice distance instead of the Euclidean distance for creating the edges.  A real-world application of discretised random graphs may be the planning of a transportation company. The real (but inscrutable) future transportation network of customer locations with edges for transports that will be actually booked can be approximated by the discrete network of planned pick-up/delivery sites with edges based on the actual transportation costs between them.

\subsection{Percolation graph of large balls in a Boolean model} \label{ssec:perco}

For this application it is more natural to start out with an infinite random graph. Since our construction is based on point processes, the basic notions from the beginning of Section~\ref{sec:setup} can be immediately generalised. All that is required is local finiteness of the vertex set, which directly implies local finiteness of the edge set without ruling out infinite degrees.

Consider a Boolean model, i.e.\ a union of random balls in $\R^d$, with centres $X_i$ given by a (finite or infinite) point process~$\sum_{i=1}^N \delta_{X_i} = \Xi$ on $\R^d$ and i.i.d.\ $\R_{+}$-valued radii $R_i$. %As always in such representations, we assume that $N$ and $X_i$ are defined based on a (fixed) selection of measurable maps $n$, $x_i$, $i \in \N$ on $\mfn$ that satisfy $\xi = \sum_{i=1}^{n(\xi)} \delta_{x_i(\xi)}$, which is possible due to \cite[Lemma~1.6]{Kallenberg2017rm}). 
We assume that $\Xi$ has locally finite expectation measure and fulfils Condition~\eqref{eq:condsigma}, denoting its Papangelou kernel by $L$, and that the distribution of the radii $R_i$ has a density $f_{R}$.
Let $R,S$ be further random variables with the same distribution and assume that $\Xi$, $\bar{R} = (R_i)_{i \in \N}$, $R$ and $S$ are independent.

Based on a Boolean model, we obtain a graph by connecting any pair of centres with overlapping balls, that is, we add an edge between $X_i$ and $X_j$ for any $i \neq j$ if and only if $\norm{X_i-X_j} \leq R_i+R_j$. We refer to such a graph as \emph{percolation graph} based on the fact that we can think of the union of balls as the permeable phase of some physical material. Note that the edges in a percolation graph are locally dependent in the sense that the presence/absence of edges incident to a single vertex %(and connecting to other vertices at similar distances) 
is influenced by the size of the ball at this vertex.

In order to obtain a suitable approximation by a random geometric graph this dependence has to be ``diluted'' by reducing the variance of the $R_i$ (in particular, for deterministic $R_i$ the percolation graph is itself a random geometric graph) 
% with kernel $\kappa(x,y) = \1\{\norm{x-y} \leq r\}$ if R_i=r
and/or by including independent information (in particular, if we replace $R_i+R_j$ by i.i.d.\ random variables $R_{ij}$, $\{i,j\} \in \N^{\atwo}$, we obtain a random geometric graph). 
% with kernel $\kappa(x,y) = S(\norm{x-y})$ if $S$ is the survival function of the $R_{ij}$
We first consider an approach that may incorporate both kinds of dilutions, but then concentrate on the former.

To derive an interesting limit theorem, we remove any balls of radius $< \runder$ for some $\runder \in \R_{+}$ and scale $\R^d$ by a factor of $q^{1/d}$ for some $q > 0$ to compensate for the rarefaction of the Boolean model.
More formally, we consider the thinned and contracted point process of centres
\begin{equation}  \label{eq: percograph transformed vertices}
  \Xi_{\runder,q} = \sum_{i=1}^N I_i \, \delta_{q^{1/d} X_i},
\end{equation}
where $I_i=I_i^{(\runder)}=\1\{R_i \geq \runder\}$.
Note that $\Xi_{\runder,q}$ is a special instance (governed by the $R_i$) of an independent thinning with contraction, i.e.\ $\Xi_{\runder,q} \eqinlaw \Xi_p\tau_q^{-1}$, where $\Xi_p$ is obtained from $\Xi$ by independently keeping each point with probability $p = p(\runder) = \Prob(R \geq \runder)$ (and deleting it otherwise) and $\tau_q: \R^d \to \R^d$, $x\mapsto q^{1/d}x$ is the contraction. It is well known that $\Xi_p\tau_p^{-1} \inlawto \pop(\lambda \mathrm{Leb}^d)$ for some constant $\lambda > 0$ iff $p \, \Xi(p^{-1/d} B) \inlawto \lambda \mathrm{Leb}^d(B)$ for sufficiently many bounded Borel sets $B \subset \R^d$, i.e.\ under suitable ergodicity conditions (see \cite[Section 11.3]{daley2007}). This also means that $q=p$ is the default choice for obtaining a limit theorem.

%Since for the thinning we select all centres with radii $R_i \geq \runder$, the radii after the contraction, have distribution $\mcl(q^{1/d} R_i \mvert R_i > \runder)$ and are still independent of each other and also of $\Xi_{\runder,q}$.

For the edges, we first work in a general setting that allows for a relatively wide choice of parameters and distributions. For this let $(\hat{R}_i)_{i \in \N}$ be another independent sequence of random variables such that $(R_i, \hat{R}_i)$ are i.i.d.\ and independent of $\Xi$. Let $\hat{R},\hat{S}$ be the corresponding generics, i.e.\ $(R,\hat{R})$ and $(S,\hat{S})$ are independent, have the same distribution as the $(R_i, \hat{R}_i)$ and are also independent of $\Xi$ and $(R_i, \hat{R}_i)$.
We then attach to each (surviving) vertex $X_i$ its $\hat{R}_i$-ball, which yields for the transformed edge process
\begin{align} \label{eq: percograph transformed edges}
  \Sigma_{\runder,q} = \sum_{\substack{i,j = 1 \\ i<j}}^{N} \1\bigl\{\norm{X_i-X_j} \leq  \hat{R}_i + \hat{R}_j \bigr\} \, I_i \hbit I_j \, \delta_{\{q^{1/d}X_i,\hbit q^{1/d}X_j\}}.
\end{align}
Each $\hat{R}_i$ may be thought of as an updated $R_i$ that is only relevant if $R_i \geq \runder$. This includes the option not to update $R_i$ by setting $\hat{R}_i = R_i$.
Furthermore, choose $\rover \in \R_+$ such that $\Prob(\hat{R} \geq \rover \mvert R \geq \runder) = 1$.\\

Let $d_V$ and $d_E$ be arbitrary metrics on $\R^d$ and $(\R^d)^{\atwo}$, respectively. As soon as we compute explicit bounds in Wasserstein metric, it will be desirable, in lack of a natural metric between patterns of infinitely many points, to consider point processes that are concentrated on a compact $\mcx \subset \R^d$, which corresponds to the setting of the main part of the paper.
% Which are equivalent (but strictly speaking not equal) to point processes on $\mfn(\mcx)$, but let's not create a fuss here.

We first derive a formula for the Papangelou kernel of the thinned and re-scaled process.
\begin{lemma} \label{le: kernel of thinned and rescaled process}
In the above setting, the Papangelou kernel $L_{\runder,q}$ of the thinned and re-scaled process~$\Xi_{\runder,q}$ satisfies $L_{\runder,q} (\, \cdot \mvert \Xi_{\runder,q}) = p(\runder) \hbit \E \bigl( L(\tau_{q}^{-1}(\cdot) \mvert \Xi\tau_q^{-1}) \bigm| \Xi_{\runder,q} \bigr)$.
\end{lemma}
\begin{proof}
Let $h:\mfn(\R^d)\times \R^d \to\R_+$ be measurable. Conditioning on $\Xi$, and keeping in mind that we tacitly use $\xi = \sum_{i=1}^{n} \delta_{x_i}$%$\xi = \sum_{i=1}^{n(\xi)} \delta_{x_i(\xi)}$ with measurable functions $n$, $x_i$, $i \in \N$ to define $N$ and $X_i$
, we have
\begin{align} \label{eq: papa wpiq}
    \E\bigg[\int_{\R^d} h(\Xi_{\runder,q}-\delta_{x},x) \,\Xi_{\runder,q}(\diff x)\bigg] 
    % &= \E \bigg[\int_{\R^d} h(\Xi^{*}_{\pi}-\delta_{\tau_p(x)}, \tau_p(x)) \,\Xi_{\pi} (\diff x)\bigg]\notag\\
    &= \E \bigg[\sum_{i=1}^N h \bigl(\tsum_{j=1, j\neq i}^{N} I_j \delta_{q^{1/d} X_j}, \hbit q^{1/d} X_i \bigr) \hbit I_i \bigg]\notag\\
    &= \E \bigg[\sum_{i=1}^N \E \bigl( h \bigl(\tsum_{j=1, j\neq i}^{N} I_j \delta_{q^{1/d} X_j}, \hbit q^{1/d} X_i \bigr) \!\bigm| \Xi) \, \E \bigl( I_i \!\bigm| \Xi \bigr) \bigg]\notag\\
    &= \E_{\Xi} \bigg[\int_{\R^d} \E_{\bar{R}} \bigl( h \bigl(g(\Xi-\delta_x, \bar{R})\tau_{q}^{-1}, \tau_q(x) \bigr) \bigr) \, p(\runder) \, \Xi (\diff x)\bigg]\notag\\
    &= \E_{\Xi} \bigg[\int_{\R^d} \E_{\bar{R}} \bigl( h \bigl(g(\Xi, \bar{R})\tau_{q}^{-1}, \tau_q(x) \bigr) \bigr) \, p(\runder) \, L (\diff x \mvert \Xi) \bigg]\notag\\
    &= \E \bigg[\int_{\R^d} h \bigl(\Xi_{\runder,q}, x \bigr) \, p(\runder) \, L (\tau_q^{-1}(\diff x) \mvert \Xi\tau_q^{-1}) \bigg] \notag\\
    &= \E \bigg[\int_{\R^d} h \bigl(\Xi_{\runder,q}, x \bigr) \, p(\runder) \, \E \bigl( L(\tau_{q}^{-1}(dx) \mvert \Xi\tau_q^{-1}) \bigm| \Xi_{\runder,q} \bigr) \bigg],
\end{align}
using the $\mcn \otimes \mcb^{\N}$-measurable map $g$ defined by $g(\xi,\bar{r}) := \sum_{i=1}^{n} \1\{r_i \geq \runder\} \delta_{x_i}$
%$g(\xi,\bar{r}) := \sum_{i=1}^{n(\xi)} \1\{r_i \geq \runder\} \delta_{x_i(\xi)}$
and the GNZ equation~\eqref{eq:gnz} for $\Xi$ in the third last equation. $\E_{\bar{R}}$ denotes the integral w.r.t.\ $P\bar{R}^{-1}$ keeping $\Xi$ fixed (note that $\Xi$, $\bar{R}$ are independent). $\E_{\Xi}$ then integrates w.r.t\ $P\bar{\Xi}^{-1}$. 
% second last line: the second component of $L$ is transformed as well as the point process information needs to be in the same space (domain space of transformation).

By the definition of the Papangelou kernel around Equation~\eqref{eq:gnz}, using that $(\xi,x) \mapsto p(\runder) L\bigl(\tau_q^{-1}(\cdot) \mvert \cdot \bigr)$ inherits the probability kernel property from $L(\cdot \mvert \cdot)$ and that the additional conditioning in $\Xi_{\runder,q}$ provides the required $\sigma(\Xi_{\runder,q})$-measurability, we obtain the statement. 
\end{proof}

\begin{theorem}[Approximation of Boolean percolation graph by random geometric graph] \label{thm: thinning application}
For $\mcx \subset \R^d$ compact, let $\Xi$ be a point process on $\mcx_{1/q} = \tau_q^{-1}(\mcx)$ fulfilling Condition~\eqref{eq:condsigma}. Choose $\runder \geq 0$ such that $\Prob(R \geq \runder) > 0$ (to rule out a trivial case) and denote  by $P_{\runder,q} \otimes Q_{\runder,q}$ the distribution of the percolation graph $(\Xi_{\runder,q}, \Sigma_{\runder,q})$ of the Boolean model based on $\hat{R}_i$ after thinning and rescaling. \\
Let $\lambda > 0$ and $t = q^{1/d} \rover$. Furthermore, denote by $P^\lambda$ the homogeneous Poisson process distribution with intensity $\lambda$ and set $\kappa_{t}(x,y) = \1_{\{\norm{x-y} \leq 2t\}}$. \\
Then we have
\begin{align*}
    W_{\mathbb{G},i} \bigl(P_{\runder,q} \otimes Q_{\runder,q} \hbit , \, P^{\lambda} \otimes Q^{\kappa_{t}}\bigr) 
    &\leq c_V(\lambda)\, \E\bigl(\bignorm{\lambda\,\text{Leb}^{d}(\diff x) - p(\runder) \hbit \E \bigl( L(\tau_{q}^{-1}(\cdot) \mvert \Xi\tau_q^{-1}) \bigm| \Xi_{\runder,q} \bigr) }_{\text{TV}} \bigr) \\[0.5mm]
    &\hspace{6mm}  {}+c_E(\lambda) \, \lambda\, c_d\, p(\runder) \, q \, \E\Xi(\mcx_{1/q}) \,
    d \, 2^d \E\bigl[(\hat{R}-\rover) \hat{R}^{d-1} \bigm| R \geq \runder \bigr],
\end{align*}
where $c_d = \frac{\pi^{d/2}}{\Gamma(d/2 +1)}$ is the volume of the unit ball in $\R^d$ and $c_V(\lambda)$, $c_E(\lambda)$ are defined in Equations~\eqref{eq: SteinPoissonV} and~\eqref{eq: SteinPoissonE}.
% Now $\Lambda = \lambda\,\abs{\mcx}$.
\end{theorem}
\begin{proof}
We apply Theorem~\ref{thm: Poisson approx in Wasserstein}, noting that the third term in the upper bound is zero because the presence or absence of edges is completely determined by their incident edges (including the radii of their balls). Thus
\begin{align} \label{eq: application proof direct use of thm}
    &W_{\mathbb{G},i}(P_{\runder,q} \otimes Q_{\runder,q}, P^{\lambda} \otimes Q^{\kappa_t}) \notag\\
    &\hspace*{2mm} \leq c_V(\lambda)\, \E\biggl(\int_{\mcx}\bigabs{ \lambda \, \text{Leb}^{d}(\diff x) - L_{\runder,q}(\diff x \mvert \Xi_{\runder,q}) } \biggr)
    + c_E(\lambda) \int_{\mcx} 
    \E\bigl[\norm{\hat{v}(\Tilde{\Sigma}_{\Xi_{\runder,q},\Sigma_{\runder,q},x})}_1 \bigr] \, \lambda \, \text{Leb}^{d}(\diff x). 
\end{align}
By Lemma~\ref{le: kernel of thinned and rescaled process} it is immediately clear that the first term is equal to the first term in the claimed upper bound.

For the second term we introduce an auxiliary sequence $(\hat{R}'_i)$ of i.i.d.\ random variables that are independent of everything and satisfy $\hat{R}'_i \sim \mcl(q^{1/d} \hat{R}_i \mvert R_i \geq \runder)$. Denote by $\hat{R}'$, $\hat{S}'$ the corresponding generics and note that $\Prob(\hat{R}' \geq t) = \Prob(\hat{R} \geq \rover \mvert R \geq \runder) = 1$.
Setting  
\begin{align*}
\Sigma'_{\runder,q} = \sum_{\substack{i,j = 1 \\ i<j}}^{N} \1\bigl\{\norm{q^{1/d}X_i-q^{1/d}X_j} \leq \hat{R}'_i+\hat{R}'_j\bigr\} \, I_i \hbit I_j \, \delta_{\{q^{1/d}X_i,\hbit q^{1/d}X_j\}},
\end{align*}
we have $(\Xi_{\runder,q},\Sigma_{\runder,q}) \eqinlaw (\Xi_{\runder,q},\Sigma'_{\runder,q})$.
The representation in terms of the $\hat{R}'_i$ will be convenient when doing computations. 

Applying the definition of $\tilde{\Sigma}_{\xi,\sigma,x}$ around Theorem~\ref{thm: general approx result} to this setting, we obtain $\tilde{\Sigma}_{\xi,\sigma,x} = \bigl(\tilde{\Sigma}_{\xi,\sigma,x}(\{x_i,x\})\bigr)_{1 \leq i \leq n} \eqinlaw (\1_{\{\norm{x-x_i} \leq q^{1/d} \hat{R}' + q^{1/d} \hat{r}'_i\}})$ for $\xi = \sum_{i=1}^n\delta_{x_i} \in \mfn(\mcx)$ and $x \in \mcx$ in the image space of the contraction (as $\tilde{\Sigma}$ is based on $\Sigma'_{\runder,q}$, not on $\Sigma$) and $\sigma=\sigma(\xi,(\hat{r}'_j))$ for given values $(\hat{r}'_j)$. In particular, we have $\bigl( \tilde{\Sigma}_{\xi,\sigma,x}(\{x_j,x\}) \bigr)_{j \neq i} = g(\xi,x,(\hat{r}_j)_{j \neq i},\hat{R}')$ for some measurable function~$g$. Therefore,
\begin{align*}
  &\int_{\mcx} \E \bigl[\norm{\hat{v}(\Tilde{\Sigma}_{\Xi_{\runder,q},\Sigma_{\runder,q},x})}_1 \bigr] \, \text{Leb}^{d}(\diff x) \\[-1mm]
  &\hspace*{1mm} = \!\!\int_{\mcx}\! \int_{\mfn(\mcx)} \!\!\!\E \sum_{i=1}^{\abs{\xi}}  \Big\vert \!\1{\{\norm{x-x_i}\! \!\leq \!2t \}}\!-\!\Prob \bigl( \norm{x-x_i}\! \!\leq \!\hat{R}'\!+ \hat{R}'_i \bigm| g(\xi,x,(\hat{R}'_j)_{j \neq i},\hat{R}') \!\bigr)\!\Big\vert \Prob \Xi_{\runder,q}^{-1}(\diff \xi) \text{Leb}^{d}(\diff x) \\
  &\hspace*{1mm} \leq \int_{\mfn(\mcx)} \, \sum_{i=1}^{\abs{\xi}} \E \int_{\mcx} \Bigabs{\1{\{\norm{x-x_i} \leq 2t \}}- \1 \bigl\{ \norm{x-x_i} \leq \hat{R}'+\hat{R}'_i \bigr\}} \, \text{Leb}^{d}(\diff x) \; \Prob \Xi_{\runder,q}^{-1}(d\xi) \\
  &\hspace*{1mm} \leq c_d \int_{\mfn(\mcx)} \, \sum_{i=1}^{\abs{\xi}} \E \bigl[ (\hat{R}'+\hat{S}')^d - (2t)^d \bigr] \; \Prob \Xi_{\runder,q}^{-1}(d\xi) \\[0.5mm]
  &\hspace*{1mm} = c_d \, p(\runder) \, \E\Xi(\mcx_{1/q}) \, \E \bigl[ (\hat{R}'+\hat{S}')^d - (2\rover)^d \bigr].
\end{align*}
The inequality in the second last line is typically strict, due to boundary effects for $\mcx$. The last equality uses $\E \Xi_{\runder,q} (\mcx) = p(\runder) \, \E \Xi (\tau_q^{-1}(\mcx))$, which is readily checked by using the representation~\eqref{eq: percograph transformed vertices} and conditioning on $\Xi$. 

Using that $\hat{R}', \hat{S}'$ are independent and $\hat{R}', \hat{S}' \geq t$ as well as $a^d - b^d = (a-b) \sum_{l=0}^{d-1} a^{l} b^{d-1-l} \leq d (a-b) a^{d-1}$ for $a \geq b$, we obtain furthermore
\begin{align*}
  \E \bigl[ (\hat{R}'+\hat{S}')^d - (2t)^d \bigr]
  &\leq d\, \E \bigl[(\hat{R}'+\hat{S}'-2t) (\hat{R}'+\hat{S}')^{d-1} \bigr] \\
  &= 2d \, \E \bigl[(\hat{R}'-t) (\hat{R}'+\hat{S}')^{d-1} \bigr] \\
  &= 2d \sum_{j=0}^{d-1} \binom{d-1}{j} \E \big[(\hat{R}'-t\big) (\hat{R}')^j \big] \, \E \big[(\hat{S}')^{d-1-j} \bigr] \\
  &\leq d \hbit 2^{d}\, \E \big[(\hat{R}'-t\big) (\hat{R}')^{d-1}\big] \\
  &= d \hbit 2^{d}\, q \, \E \big[(\hat{R}-\rover\big) (\hat{R})^{d-1} \bigm| R \geq \runder \big],
\end{align*}
where the last inequality follows by $\hat{S} \eqinlaw \hat{R}$ and the FKG-inequality since the functions $x\mapsto (x-t)^{+} x^j$ and $x\mapsto x^{d-1-j}$ are both monotonically increasing. Putting everything together yields the claim. 
\end{proof}

In what follows we apply Theorem~\ref{thm: thinning application} to obtain a quantitative limit result for a thinned and rescaled Boolean model with centres given by a determinantal point process and a more concrete transformation rule for the radii $\hat{R}_i$.

Before we state the result in Corollary~\ref{cor: boolean_vs_grg}, we briefly recall determinantal point processes. A point process~$\Xi$ is said to have correlation functions $\rho^{(n)}:\mcx^n\to\R_+$ if 
\begin{align*}
    \E\bigg(\sum_{\{x_1,\ldots,x_n\} \subset\Xi}h(x_1,\ldots,x_n)\bigg) = \int_{\mcx}\ldots \int_{\mcx}h(x_1,\ldots,x_n) \rho^{(n)}(x_1,\ldots,x_n)\alpha(\diff x_1)\ldots \alpha(\diff x_n)
\end{align*}
for all measurable functions $h:\mcx^n\to\R_+$ and $n\in \N$.

Then a determinantal point process is a point process with correlation functions given by the determinant of a given kernel ${K:\R^d\times\R^d \to \C}$. 
\begin{definition}[Definition 2.1 \cite{georgii2005}]
   Let $K:\R^d\times\R^d \to \C$ be the kernel of a bounded positive definite Hermitian integral operator $\mathcal{K}$ on $L^2(\R^d,\text{Leb}^d)$. Then we call a point process \textit{determinantal point process (DPP) with correlation kernel $K$} if its correlation functions~$\rho^{(n)}$ are given by 
   $$\rho^{(n)}(x_1,\ldots,x_n) = \text{det}(K(x_i,x_j)_{i,j\in[n]}).$$
   for any $x_1,\ldots,x_n\in \mcx$ and $n\in\N$.
\end{definition}
By \textcite[Theorem~2.3]{lavancier2015} such a DPP exists if additionally to the above assumptions the spectrum of $\mathcal{K}$ is contained in $[0,1]$. In the following, we restrict ourselves to stationary kernel $K$, i.e.\ $K(x,y) = K(x-y)$ for $x,y\in\R^d$, and denote the intensity of a stationary DPP by $\mu = K(x,x)$ for $x\in\R^d$. Furthermore, we assume that $\mathcal{K}$ is of local trace class with spectrum in $[0,1)$ and remark that this corresponds to Hypothesis (H) in \textcite{georgii2005}. In particular, by \textcite[Theorem~3.1]{georgii2005} any DPP that fulfils the above assumptions and condition~\eqref{eq:condsigma} has a conditional intensity. We refer to \textcite{georgii2005} and \textcite{lavancier2015} for a more thorough introduction to determinantal point processes including a discussion of the assumptions made above.

\begin{corollary} \label{cor: boolean_vs_grg}
Consider a determinantal point process on $\R^d$ with kernel $K$ fulfilling the above assumptions and intensity $\mu = K(x,x)$. Let $\Xi$ be the restriction of the point process to $\mcx_{1/q}$. For a strictly increasing function $\psi$ with $\psi(0)=0$, consider the edge process where $\runder=\rover$ satisfies $\Prob(R \geq r_{*}) > 0$ and $\hat{R}_i = r_{*} + \psi(R_i - r_*)$. Denote by $P_{\rstarq} \otimes Q_{\rstarq}$ the distribution of the percolation graph $(\Xi_{\rstarq}, \Sigma_{\rstarq})$ of the Boolean model based on $\hat{R}_i$ after thinning and rescaling. \\
Let furthermore $\lambda = \lambda_{\rstarq} = \mu \tfrac{p(r_*)}{q}$, $t = q^{1/d}r_*$. Denote by $P^{\lambda}$ the homogeneous Poisson process distribution with intensity $\lambda$ and set $\kappa_t(x,y) = \1_{\{\norm{x-y} \leq 2t\}}$.\\
Then
\begin{align*}
W_{\mathbb{G},i} \bigl(&P_{\rstarq} \otimes Q_{\rstarq} \hbit , \, P^{\lambda} \otimes Q^{\kappa_{t}}\bigr) \\[1mm]
&\leq 2 c_V(\lambda_{\rstarq}) \,\lambda_{\rstarq} \abs{\mcx} \frac{q}{1-q} \\[-0.5mm]
&\hspace*{6mm} {} + c_E(\lambda_{\rstarq}) \, \lambda_{\rstarq}\, c_d\, \mu \abs{\mcx}
    d \, 2^{d} \E \bigl[ \psi(R-r_*) \bigl(r_* + \psi(R-r_*) \bigr)^{d-1} \, \1\{R \geq r_*\} \bigr] \\[1mm]
&\leq \mathrm{const} \Bigl(\min \bigl\{ \lambda_{\rstarq} \abs{\mcx}, 1 + \log^{+}(\lambda_{\rstarq} \abs{\mcx}) \bigr\} \frac{q}{1-q}  \\[-0.5mm]
&\hspace*{6mm} {} + \min \bigl\{ \lambda_{\rstarq} \abs{\mcx}, 1 \bigr\} \mu \E \bigl[ \psi(R-r_*) \bigl(r_* + \psi(R-r_*) \bigr)^{d-1} \, \1\{R \geq r_*\} \bigr] \Bigr),
\end{align*}
where $c_d = \frac{\pi^{d/2}}{\Gamma(d/2 +1)}$ is the volume of the unit ball in $\R^d$ and $c_V(\lambda)$, $c_E(\lambda)$ are defined in Equations~\eqref{eq: SteinPoissonV} and~\eqref{eq: SteinPoissonE}.
\end{corollary}

\begin{proof}
By Theorem~\ref{thm: thinning application}, we have
\begin{align} \label{eq: proof thinning with weibull upper bound}
    W_{\mathbb{G},i} \bigl(P_{\rstarq} \otimes Q_{\rstarq} \hbit , \, P^{\lambda} \otimes Q^{\kappa_{t}}\bigr)
    &\leq c_V(\lambda)\, \E\bigl(\bignorm{\lambda\,\text{Leb}^{d}(\diff x) - p(r_*) \hbit \E \bigl( L(\tau_{q}^{-1}(\cdot) \mvert \Xi\tau_q^{-1}) \bigm| \Xi_{\rstarq} \bigr) }_{\text{TV}} \bigr) \notag\\[0.5mm]
    &\hspace{6mm}  {}+c_E(\lambda) \, \lambda\, c_d\, p(r_*) \, q \, \E\Xi(\mcx_{1/q}) \,
    d \, 2^d \E\bigl[(\hat{R}-r_*) \hat{R}^{d-1} \bigm| R \geq r_* \bigr].
\end{align}

Evaluating the first term, we obtain
\begin{align*}
     \E\bigl(&\bignorm{\lambda\,\text{Leb}^{d}(\diff x)- p(r_*) \,\E(L(\tau_{q}^{-1}(\diff x)\mvert \Xi\tau_{q}^{-1})\mvert \Xi_{\rstarq})}_{\text{TV}}\bigr) \\
     &= \E\bigl(\bignorm{\lambda\,\text{Leb}^{d}(\diff x)- \tfrac{p(r_*)}{q} \,\E(\lambda(\tau_{q}^{-1}(x)\mvert \Xi\tau_{q}^{-1})\mvert \Xi_{\rstarq})\,\text{Leb}^{d}(\diff x) }_{\text{TV}}\bigr) \\
     &= \tfrac{p(r_*)}{q} \int_{\mcx} \E\bigl(\bigabs{\mu -\E(\lambda(\tau_{q}^{-1}(x)\mvert \Xi)\mvert \Xi_{\rstarq})} \bigr) \, \,\text{Leb}^{d}(\diff x)\\
     &\leq \tfrac{p(r_*)}{q} \int_{\mcx} \E\bigl(\bigabs{ \lambda(\tau_q^{-1}(x)\mvert \emptyset)-\mu } + \bigabs{\lambda(\tau_q^{-1}(x)\mvert \emptyset)-\E(\lambda(\tau_{q}^{-1}(x)\mvert \Xi)\mvert \Xi_{\rstarq})} \bigr) \, \text{Leb}^{d}(\diff x).
     %&= \tfrac{p(r*)}{q} \int_{\mcx} \E\biggl( \lambda(\tau_q^{-1}(x)\mvert \emptyset) - \mu + \lambda(\tau_q^{-1}(x)\mvert \emptyset)-\E(\lambda(\tau_{q}^{-1}(x)\mvert \Xi)\mvert \Xi_{\rstarq}) \biggr) \, \text{Leb}^{d}(\diff x)\\.
\end{align*}

By \textcite[Theorem~3.1]{georgii2005} the process $\Xi$ on $\mcx_{1/q}$ is repulsive, i.e.\ its corresponding conditional intensity $\lambda$ fulfils
$ \lambda(x\mvert \eta)\leq \lambda(x\mvert \xi) $ for any $\xi\subset \eta$ and $x\in \mcx_{1/q}$ and thus $ \E(\lambda(\tau_{q}^{-1}(x)\mvert \Xi\tau_q^{-1})\mvert \Xi_{\rstarq}) \leq \lambda(\tau_q^{-1}(x)\mvert \emptyset)$. Furthermore, the stationarity of $K$ implies 
\begin{align}\label{eq: proof DPP intensity}
 \mu = K(x,x) =%\overset{\text{stationarity}}{=} 
K(\tau_{q}^{-1}(x),\tau_{q}^{-1}(x)) = \rho(\tau_{q}^{-1}(x)) = \E(\lambda(\tau_{q}^{-1}(x)\mvert \Xi)) \leq \lambda(\tau_q^{-1}(x)\mvert \emptyset),   
\end{align}
where $\rho = \rho^{(1)}$ is the correlation function of order~$1$ and the last equality follows from the GNZ equation. Thus, we can omit the absolute values in the integrand above and, using $\mu = \E(\lambda(\tau_{q}^{-1}(x)\mvert \Xi))$ from Equation~\eqref{eq: proof DPP intensity}, obtain for the whole expression 
\begin{align*}
    %&\tfrac{p(r*)}{q} \int_{\mcx} \lambda(\tau_q^{-1}(x)\mvert \emptyset) - \mu + \lambda(\tau_q^{-1}(x)\mvert \emptyset)-\E(\lambda(\tau_{q}^{-1}(x)\mvert \Xi)) \, \text{Leb}^{d}(\diff x)\\
    2 \tfrac{p(r_*)}{q} \int_{\mcx} \lambda(\tau_q^{-1}(x)\mvert \emptyset) - \mu \,\text{Leb}^{d}(\diff x)
    &= 2 \tfrac{p(r_*)}{q} \int_{\mcx} \lambda(x\mvert \emptyset) - \mu \,\text{Leb}^{d}(\diff x).
\end{align*}
Here we also used that $\lambda(\cdot \mvert \emptyset) $ is constant as the stationarity of $K$ implies (strong) stationarity of the corresponding DPP, see \textcite[Chapter~3]{lavancier2015}.

We investigate the correlation kernel $K$ and its interaction operator $J := K(I-K)^{-1}$, where $I$ is the identity operator, in more detail. Define $K_q(x,y) := K(\tau_q^{-1}(x),\tau_q^{-1}(y))$ and note that by stationarity $\mu = K(x,x) = K_q(x,x)$ for any $x\in\mcx$. In particular, $J(x,x) = J_q(x,x)$, where $J_q:= K_q(I-K_q)^{-1}$ is the interaction operator corresponding to $K_q$. Furthermore, by \textcite[Theorem~3.1]{georgii2005}, the conditional intensity fulfils $\lambda(x\mvert \emptyset) = J(x,x)$ for all $x\in\mcx_{1/q}$. This implies
\begin{align*}
    2 \tfrac{p(r*)}{q} \int_{\mcx} \lambda(x\mvert \emptyset) - \mu \,\text{Leb}^{d}(\diff x) 
    &= 2 \tfrac{p(r*)}{q} \int_{\mcx} J(x,x) - K(x,x) \,\text{Leb}^{d}(\diff x)\\
    &= 2 \tfrac{p(r*)}{q} \int_{\mcx} J_q(x,x) - K_q(x,x) \,\text{Leb}^{d}(\diff x).
\end{align*}

To bound the right hand-side of this expression, we consider suitable representations of $J_q$ and $K_q$ that can be constructed from a representation of $K$. More precisely, by Mercer's theorem, %(2.5) in \textcite{lavancier2015} 
we find a spectral representation for $K$ (on $\mcx_{1/q}$), i.e.
$$K(x,y) = \sum_{k=1}^{\infty} \kappa_n \varphi_n(x) \overline{\varphi_n(y)},$$
where $\kappa_n$ are the eigenvalues of the operator $K$ (on $\mcx_{1/q}$) and $\varphi_n$ form an orthonormal basis (in $L^2(\mcx_{1/q})$) of corresponding eigenfunctions.  
Using this representation we can deduce a representation for $K_q$ (on $\mcx$) by 
\begin{align*}
    K_q(x,y) = K(\tau_q^{-1}(x),\tau_q^{-1}(y)) 
    %&= \sum_{k=1}^{\infty} \kappa_n \varphi_n(\tau_q^{-1}(x)) \overline{\varphi_n(\tau_q^{-1}(y))} \\
    &= \sum_{k=1}^{\infty} \kappa_n\norm{\varphi_n\circ \tau_q^{-1}}^2 \, \frac{\varphi_n\circ\tau_q^{-1}(x)}{\norm{\varphi_n\circ \tau_q^{-1}}} \,\overline{\frac{\varphi_n\circ\tau_q^{-1}(y)}{\norm{\varphi_n\circ \tau_q^{-1}}}}\\
    %&= \sum_{k=1}^{\infty} \kappa_n q  \frac{\varphi_n\circ\tau_q^{-1}(x)}{\sqrt{q}} \overline{\frac{\varphi_n\circ\tau_q^{-1}(y)}{\sqrt{q}}}\\
    &= \sum_{k=1}^{\infty} \kappa_{n,q}  \varphi_{n,q}(x) \overline{\varphi_{n,q}(y)},
\end{align*}
where $\varphi_{n,q}(x) := \frac{\varphi_n\circ\tau_q^{-1}(x)}{\sqrt{q}}$ and $\kappa_{n,q} := \kappa_n q$. The last equation follows as $(\varphi_n)_n$ is an orthonormal basis and thus 
$$\norm{\varphi_n\circ \tau_q^{-1}}^2 = \int_{\mcx} \abs{\varphi_n\circ \tau_q^{-1}(x)}^2 \diff x = q \int_{\mcx_{1/q}} \abs{\varphi_n(x)}^2 \diff x = q. $$

Then, as $\mathcal{K}$ is of local trace class with spectrum contained in $[0,1)$, we can represent the corresponding interaction operator $J_q$ by 
$$J_q(x,y) = \sum_{k=1}^{\infty} \frac{\kappa_{n,q}}{1-\kappa_{n,q}}  \varphi_{n,q}(x) \overline{\varphi_{n,q}(y)},$$
see \textcite[Chapter~3]{georgii2005}. 

An application of these representations to the above integral expression yields
\begin{align*}
    2 \tfrac{p(r_*)}{q} \int_{\mcx} J_q(x,x) - K_q(x,x) \,\text{Leb}^{d}(\diff x)
    &= 2 \tfrac{p(r_*)}{q} \int_{\mcx} \sum_{k=1}^{\infty} \frac{\kappa_{n,q}^2}{1-\kappa_{n,q}}  \abs{\varphi_{n,q}(x)}^2 \, \text{Leb}^{d}(\diff x)\\
    &\leq 2 \tfrac{p(r_*)}{1-q} \int_{\mcx} \sum_{k=1}^{\infty} \kappa_{n,q}  \abs{\varphi_{n,q}(x)}^2 \, \text{Leb}^{d}(\diff x)\\
    &= 2 \tfrac{p(r_*)}{1-q} \int_{\mcx} K_q(x,x) \, \text{Leb}^{d}(\diff x)\\
    &= 2 \tfrac{q}{1-q} \lambda_{\rstarq}\, \abs{\mcx},
\end{align*}
where the inequality follows as by assumption on the spectrum of $\mathcal{K}$ we have $\kappa_n\leq 1$ and thus $\kappa_{n,q} = q \kappa_n\leq q$. We remark that a similar result for $p=q$ is obtained in \textcite[Theorem~5.12]{decreusefond2018} and \textcite[Chapter~8]{torrisi2017}. Finally, an application of the bounds in Theorem~\ref{thm: Poisson approx in Wasserstein} with $\Lambda = \lambda_{\rstarq} $ yields the claim for the vertex terms. %and thus $\frac{\kappa_{n,q}^2}{1-\kappa_{n,q}}\leq \frac{q}{1-q}\kappa_{n,q}. $

For the second term of the upper bound in \eqref{eq: proof thinning with weibull upper bound}, we note that $q \, \E(\Xi(\mcx_{1/q})) = \mu \abs{\mcx}$. Furthermore, 
\begin{align*}
  p(r_*) \, \E\bigl[(\hat{R}-r_*) \hat{R}^{d-1} \bigm| R \geq r_{*} \bigr] 
  &= \E \bigl[ \psi(R-r_*) \bigl(r_* + \psi(R-r_*) \bigr)^{d-1} \, \1\{R \geq r_*\} \bigr].
 % &= q \, \E \bigl[ \psi(r_*) \bigl(r_* + \psi(r_*) \bigr)^{d-1} \bigr] \\
 %  &\leq q 2^{d-1} \E \bigl[ \psi(r_*) \, \max\bigl\{\psi(r_*)^{d-1}, (r_*)^{d-1} \bigr\} \bigr]
\end{align*}
Thus, the second term is bounded by
$$c_E(\lambda_{\rstarq}) \, \lambda_{\rstarq}\, c_d\, \mu \abs{\mcx}
    d \, 2^{d} \E \bigl[ \psi(R-r_*) \bigl(r_* + \psi(R-r_*) \bigr)^{d-1} \, \1\{R \geq r_*\} \bigr].$$

\end{proof}

\begin{example} \label{ex: boolean}
  We consider a special case within the setting of Corollary~\ref{cor: boolean_vs_grg}, writing now $r$ rather then $r_*$ for the cutoff (and lower bound) of the radii. By $f(r) \asymp g(r)$ as $r \to \infty$ for functions $f,g \colon (0,\infty) \to (0,\infty)$, we mean that there is an $r_0>0$ and positive constants $\kappa_1,\kappa_2$ such that $\kappa_1 \leq f(r)/g(r) \leq \kappa_2$ for all $r \geq r_0$.\\
  Suppose that the radius distribution satisfies $p(r) = \Prob(R \geq r) \asymp r^{-a}$ and the contraction rate is $q = q(r) \asymp r^{-b}$ for some $a,b\in\R_+$. Let furthermore $\psi(s) = \psi_r(s) = s^{\gamma}/r^{\delta}$ for some $\gamma \in (0,1]$ and $\delta \in [0,\infty)$. Then, using $(x+y)^{d-1} \leq 2^{d-1} (x^{d-1}+y^{d-1})$ for $x,y\in\R_+$, we have
  \begin{align*}
    2^{-d+1} \E \bigl[ &\psi(R-r) \bigl(r + \psi(R-r) \bigr)^{d-1} \, \1\{R \geq r\} \bigr] \\
    &\leq \E \bigl[ \psi(R-r) r^{d-1} \, \1\{R \geq r\} \bigr] + \E \bigl[ \psi(R-r)^{d} \, \1\{R \geq r\} \bigr].
  \end{align*}
  For $k \in \N$, we obtain 
  \begin{align*}
    \E \bigl[ \psi(R-r)^{k} \, \1\{R \geq r\} \bigr] &= \E \bigl[ (R-r)^{\gamma k} r^{-\delta k} \, \1\{R \geq r\} \bigr] \\
    &= \int_{0}^{\infty} \Prob \bigl( (R-r)^{\gamma k} r^{-\delta k} \, \1\{R \geq r\} > s \bigr) \, ds \\
     &= \int_{0}^{\infty} \Prob \bigl( R > r^{\delta/\gamma} s^{1/(\gamma k)} + r \bigr) \, ds \\
     &\asymp \int_{0}^{\infty} (r^{\delta k} s + r^{\gamma k})^{-a/(\gamma k)} \, ds \\
      &= \frac{r^{-\delta k}}{a/(\gamma k) - 1} \Bigl[ - (r^{\delta k} s + r^{\gamma k})^{-a/(\gamma k)+1} \Bigr]_{s=0}^{\infty} \\
      &= \frac{\gamma k}{a-\gamma k} \, r^{-a-\delta k + \gamma k}
  \end{align*}
  if $a > \gamma k$.\\ % $\infty$ otherwise
  In total, we thus have
  \begin{align*}
W_{\mathbb{G},i} \bigl(P_{r,q} \otimes Q_{\rstarq} \hbit , \, P^{\lambda} \otimes Q^{\kappa_{t}}\bigr) \,=\, \mathcal{O} \Bigl( &\min\bigl\{r^{-a+b} \abs{\mcx}, 1+\log^+(r^{-a+b}\abs{\mcx})\bigr\} \, r^{-b} \\ 
&{}+ \min\bigl\{r^{-a+b}\abs{\mcx},1\} \, r^{-a + \max\{\gamma-\delta+d-1, (\gamma-\delta)d \}} \Bigr)
   \end{align*}
if $\gamma < a/d$. \\
Since distances between points in the domain of the contraction are of the order of magnitude~$r^{a/d}$, our edge construction based on comparing distances with a cutoff of order of magnitude $r$ converges to a graph with empty edge set in the limit as $r \to \infty$ if $a > d$ and to a complete graph if $a < d$. \\
The most interesting case is therefore $a=d$, which requires $\gamma < 1$ for a finite upper bound. In this case, the bound simplifies to 
\begin{align*}
W_{\mathbb{G},i} \bigl(P_{r,q} \otimes Q_{\rstarq} \hbit , \, P^{\lambda} \otimes Q^{\kappa_{t}}\bigr) \,=\, \mathcal{O} \Bigl( &\min\bigl\{r^{-d+b} \abs{\mcx}, 1+\log^+(r^{-d+b}\abs{\mcx})\bigr\} \, r^{-b} \\ 
&{}+ \min\bigl\{r^{-d+b}\abs{\mcx},1\} \, r^{-(1-\gamma+\delta)} \Bigr).
\end{align*}
It is worthwhile noting that the improved Stein factors obtained in Subsection~\ref{ssec: bounding stein factors} make a considerable difference if $\abs{\mcx}$ is large or the scaling is of higher order than the thinning, i.e.\ $b > a = d$.
\end{example}
\begin{remark}
  If any of the Wasserstein distances in Corollary~\ref{cor: boolean_vs_grg} and Example~\ref{ex: boolean} go towards~$0$ as $r \to \infty$ while the approximating $P^{\lambda} \otimes Q^{\kappa_t}$ distribution is constant in $r_*=r$ (or converges itself weakly to a fixed distribution), we obtain weak convergence of $P_{\rstarq} \otimes Q_{\rstarq}$. This is due to the fact that weak convergence of point process distributions is determined locally, in the sense that $\Xi_n \inlawto \Xi$ for point processes on $\R^d$ (say) if and only if $\Xi_n\vert_{\mcx} \inlawto \Xi\vert_{\mcx}$ for every compact $\mcx \subset \R^d$, as can be seen from any of the characterisations in Theorem~4.11 of \textcite{Kallenberg2017rm}.\\
  While we cannot define a natural (i.e.\ translation invariant) GOSPA metric for infinite random graphs on the whole of $\R^d$, we can resort to the standard construction of using the metric $\overline{d}_{\G,i}$ given by
  \begin{equation*}
      \overline{d}_{\G,i}((\xi,\sigma), (\eta,\tau)) = \int_{\R_{+}} d_{\G,i}\bigl((\xi,\sigma)\vert_{\B(0,t) \times \B(0,t)^{\atwo}}, (\eta,\tau)\vert_{\B(0,t) \times \B(0,t)^{\atwo}}\bigr) \, e^{-t} \, dt
  \end{equation*}
  for graphs $(\xi,\sigma)$, $(\eta,\tau)$ on $\R^d$ with only finite degrees. In the Wasserstein metric $\overline{W}_{\G,i}$ based on $\overline{d}_{\G,i}$, we obtain the same rates as above without the $\abs{\mcx}$ terms, which go into constants of the form $\pi \int_0^{\infty} t^d e^{-t} \, dt$ or $\pi d \int_0^{\infty} \log(t) e^{-t} \, dt$.\\
  Thus, for example, if $a=b=d$ and $\delta=0$ in Example~\ref{ex: boolean}, we obtain
  \begin{equation*}
     P_{r,q} \otimes Q_{r,q} \longrightarrow P^{\lambda} \otimes Q^{\kappa_{t}} \quad \text{ weakly as $r \to \infty$}
  \end{equation*}
  at rate
  \begin{equation*}
    \overline{W}_{\G,i}\bigl(P_{r,q} \otimes Q_{r,q} \hbit , \, P^{\lambda} \otimes Q^{\kappa_{t}}\bigr) = \mathcal{O}\bigl(r^{-(1-\gamma)}\bigr).
  \end{equation*}
  Note that Wasserstein convergence \emph{characterises} weak convergence since $d_{\G,i}$ and hence $\overline{d}_{\G,i}$ is bounded.
\end{remark}

\subsection{Discrete Graphs}

In the following, we compare a generalised random geometric graph on the compact set $\mcx \subset \mathbb{R}^d$ to a discretised graph whose vertices lie in a finite subset $\Lambda = \Lambda_n = \{y_i\}_{i\in [n]}\subset \mcx$.
Consider a partition $\X = \X_{\Lambda} = (\mcx_i)_{i \in [n]}$ of the space $\mcx$ that satisfies $y_i\in \mcx_i$ for every $i\in[n]$ and let $r_V = \max_{i\in [n]} \sup_{x\in \mcx_i} d_V(x,y_i)$ be the maximal radius of all cells in $\X$. Furthermore, denote by $\mfn_{\Lambda}$ the space of finite simple counting measures on $\Lambda$ (equipped with the natural $\sigma$-algebra given by its power set).

Given a random geometric graph $\mathrm{RGG}(\lambda,\kappa)$, we aim to construct an approximating discrete random graph. Starting with the vertices, we define a Gibbs process on $\Lambda$ as a point process that has a hereditary density $u_{\Lambda}:\mfn_{\Lambda}\to \R_+$ with respect to the distribution of a Poisson point process on $\Lambda$ with intensity measure $\alpha_{\Lambda}$ given by $\alpha_{\Lambda}(y_i) = \alpha(\mcx_i)$ for $i\in [n]$. By setting $\lambda_{\Lambda}(\cdot\mvert \xi) := u_{\Lambda}(\xi + \cdot)/u_{\Lambda}(\xi)$ we obtain the corresponding conditional intensity. For a Gibbs process on $\mcx$ with density $u$ and conditional intensity~$\lambda$, we use $u_{\Lambda} = u\vert_{\Lambda} \big/ \sum_{i \in [n]} u(y_i)$ for its discretised version, which is thus a Gibbs process on $\Lambda$ having conditional intensity~$\lambda_{\Lambda}$ given by $\lambda_{\Lambda} = \lambda\vert_{\mfn_{\Lambda}\times \Lambda}$. This corresponds to what \textcite[Chapter~4.3]{stucki2014} refer to as ``discrete analogon'' of a Gibbs process. To avoid problems with possible multi-points, we assume that the original (and thus the discretised) Gibbs process has zero density on point patterns with multi-points. Finally, we let the edges of the corresponding discretised graph be distributed according to $Q^{\kappa_{\Lambda}}$ for a given matrix $\kappa_{\Lambda}^n = (\kappa_{\Lambda}(y_i,y_j))_{i,j,\in [n]}$ of connection probabilities based on the centres of the cells $\mcx_i$.
Given a random geometric graph $(\Xi,\Sigma)\sim \mathrm{RGG}(\lambda,\kappa)$ on $\mcx$, this construction yields a corresponding discretised graph $(\Xi_{\Lambda},\Sigma_{\Lambda})\sim P^{\lambda_{\Lambda}}\otimes Q^{\kappa_{\Lambda}}$ on $\Lambda$.

Before we can state the result for the approximation of a random geometric graph by its discretised graph, we need to introduce an intermediate graph that is constructed from the discretised graph but has vertices that can lie anywhere in $\mcx$. 
%Additional to the discretised graph~$(\Xi_{\Lambda},\Sigma_{\Lambda})$ we introduce an intermediate graph $(\Xi_I,\Sigma_I)$ on~$\mcx$. 
%This intermediate graph is necessary as the upper bound in this Theorem~\ref{thm: Poisson approx in Wasserstein} compares the conditional intensities and the edge connection functions in a total variation norm style. 
More precisely, given the vertex process~$\Xi_{\Lambda}\sim P^{\lambda_{\Lambda}}$ on $\Lambda$, we construct the intermediate vertex process $\Xi_I$ by resampling each point of $\Xi_{\Lambda}$ uniformly in the corresponding cell~$\mcx_i$, i.e.\ 
$\Xi_I = \sum_{i=1}^n N_i \delta_{U_i}$ for $N_i = \Xi_{\Lambda}(\{y_i\}) \in \{0,1\}$ a.s.\ and independent $U_i\sim\alpha\vert_{\mcx_i}(\cdot)/\alpha(\mcx_i)$.  
By \textcite[Lemma~28]{stucki2014}, $\Xi_I$ is again a Gibbs process with conditional intensity $\lambda_I(x\mvert\xi) = \lambda_{\Lambda}(t(x)\mvert t(\xi)) = \lambda(t(x)\mvert t(\xi))$, where $t:\mcx\to\Lambda$ is given by $t(x) = y_i$ for the $i\in[n]$ with $x\in \mcx_i$ and $t(\xi):= \sum_{y\in\xi} \delta_{t(y)}$.
Finally, we let the distribution of the intermediate edge process be given by $Q^{\kappa_I}$ with $\kappa_I(x,x') := \kappa_{\Lambda}(t(x),t(x'))$ and thus obtain an intermediate graph $(\Xi_I,\Sigma_I)\sim P^{\lambda_I}\otimes Q^{\kappa_I}$ on $\mcx$.

We can now state the approximation result of a random geometric graph by its discretised graphs.  

\begin{theorem}[Approximation of RGG by discretised graphs] \label{thm: discretization application}
Let $(\Xi,\Sigma)$ be an $\mathrm{RGG}(\lambda,\kappa)$ such that $\Xi\sim P^{\lambda}$ fulfills \eqref{eq: stability cond}. Construct the corresponding discretised graph $(\Xi_{\Lambda},\Sigma_{\Lambda})\sim P^{\lambda_{\Lambda}}\otimes Q^{\kappa_{\Lambda}}$ and intermediate graph $(\Xi_I,\Sigma_I)\sim P^{\lambda_I}\otimes Q^{\kappa_I}$ as described above. Then
    \begin{align*}
        &W_{\G,i}(P^{\lambda}\otimes Q^{\kappa},P^{\lambda_{\Lambda}}\otimes Q^{\kappa_{\Lambda}}) \\
        &\leq r_V + C_i B^*\, \E\biggl(\int_{\mcx}\bigl\vert\lambda(x\mvert \Xi_I) -\lambda(t(x) \mvert t(\Xi_I))\bigr\vert \, \alpha(\diff x)\biggr)\\
    &\quad + \frac14 C_E\,\E\biggl(\int_{\mcx^2} \bigabs{\kappa(x,y) - \kappa_{\Lambda}(t(x),t(y))}\,  \lambda(x \mvert \Xi_I + \delta_y) \hbit \lambda(t(y)\vert t(\Xi_I)) \, \alpha^2(\diff (x,y))\biggr)
    \end{align*}
\end{theorem}
\begin{proof}
By the triangle inequality we have 
\begin{align}\label{eq: proof of discretised approx triangle ineq}
    W_{\G,i}(P^{\lambda}\otimes Q^{\kappa},P^{\lambda_{\Lambda}}\otimes Q^{\kappa_{\Lambda}}) \leq W_{\G,i}(P^{\lambda}\otimes Q^{\kappa},P^{\lambda_I}\otimes Q^{\kappa_I}) + W_{\G,i}(P^{\lambda_I}\otimes Q^{\kappa_I},P^{\lambda_{\Lambda}}\otimes Q^{\kappa_{\Lambda}}).
\end{align}
For the first summand, we apply Theorem~\ref{thm: Poisson approx in Wasserstein} in the general case of Remark~\ref{re: Stein factor for general process with GOSPA}. Since both graphs are generalised random geometric graphs, the third term vanishes, which yields 
\begin{align*}
    &W_{\G,i}(P^{\lambda}\otimes Q^{\kappa},P^{\lambda_I}\otimes Q^{\kappa_I})\\
    %&\leq C_i B^*\,\E\biggl(\int_{\mcx}\big\vert\lambda(x\mvert \Xi_I) -\lambda_I(x\vert\Xi_I)\big\vert \alpha(\diff x)\biggr) + \frac14 C_E\E\biggl(\int_{\mcx} 
    %\E\bigl[\norm{\hat{v}({\Sigma}_{I,\Xi_I,x})}_1 \bigr] \lambda(x\vert\Xi_I)\,\alpha(\diff x)\biggr)\\
    &\leq C_i B^*\,\E\biggl(\int_{\mcx}\,\big\vert\lambda(x\mvert \Xi_I) -\lambda(t(x)\vert t(\Xi_I))\big\vert\, \alpha(\diff x)\biggr)\\
    &\quad + \frac14 C_E\E\biggl(\int_{\mcx}  \sum_{y\in \Xi_I} \abs{\kappa_{\Lambda}(t(x),t(y)) - \kappa(x,y)} \lambda(x\vert \Xi_I)\,\alpha(\diff x)
    \biggr)\\
    &= C_i B^*\,\E\biggl(\int_{\mcx}\,\big\vert\lambda(x\mvert \Xi_I) -\lambda(t(x)\vert t(\Xi_I))\big\vert\, \alpha(\diff x)\biggr)\\
    &\quad + \frac14 C_E\,\E\biggl(\int_{\mcx}\int_{\mcx} \abs{\kappa_{\Lambda}(t(x),t(y)) - \kappa(x,y)}\,  \lambda(x\vert \Xi_I + \delta_y)\,\lambda(t(y)\vert t(\Xi_I))\,\alpha(\diff x)\,\alpha(\diff y)  \biggr),
\end{align*}
by the GNZ equation~\eqref{eq:gnz}.

It remains to consider the second summand in \eqref{eq: proof of discretised approx triangle ineq}. 
By construction of the intermediate graph, writing $M = \abs{\Xi_{\Lambda}} = \abs{\Xi_I}$, the intermediate vertex process $\Xi_{I} = \sum_{i=1}^M \delta_{\tilde{U}_i}$ is obtained by resampling every vertex $\tilde{Y}_i$ of the discretised process $\Xi_{\Lambda} = \sum_{i=1}^M \delta_{\tilde{Y}_i}$ from the cell $\mcx_j$ whose centre $y_j$ is equal to $\tilde{Y}_i$. This implies $d_V(\tilde{Y}_i, \tilde{U}_i) \leq r_V$ and $\kappa_{\Lambda}(\tilde{Y}_i,\tilde{Y}_{i'}) = \kappa_{\Lambda}(t(\tilde{U}_i),t(\tilde{U}_{i'})) = \kappa_I(\tilde{U}_i,\tilde{U}_{i'})$. By the latter, we can couple the edge processes in such a way that $\Sigma_I(\{\tilde{U}_i,\tilde{U}_{i'}\}) = \Sigma_{\Lambda}(\{Y_i,Y_{i'}\})$. 

Thus we obtain
\begin{align*}
    W_{\G,i}(P^{\lambda_I}\otimes Q^{\kappa_I},P^{\lambda_{\Lambda}}\otimes Q^{\kappa_{\Lambda}}) %= \sup_{f\in \mcf_d} \vert \E(f(\Xi_I,\Sigma_I)-f(\Xi_{\Lambda},\Sigma_{\Lambda}))\vert
    \leq \E[d_{\mathbb{G},i}((\Xi_I,\Sigma_I),(\Xi_{\Lambda},\Sigma_{\Lambda}))]%\\
    %&\leq \E\bigg(\frac{1}{M} \sum_{i=1}^M \E(d_V(Y_i,U_i)) + \frac12 \frac{1}{M\,(M-1)} \sum_{(i,j)\in[M]^2} \E(\abs{\kappa_{\Lambda}(Y_i,Y_j) - \kappa_{\Lambda}(Y_i,Y_j)})\bigg)\\
    %\leq \E\bigg(\frac{1}{M} \sum_{i=1}^M \E(d_V(Y_i,U_i)) \bigg)
    \leq r_V,
\end{align*}
where the last inequality is immediately seen by \eqref{eq: dgi same cardinality}.
\end{proof}

We apply this approximation result to a generalised random geometric graph with vertices given by an inhibitory pairwise interaction processes as introduced in Section~\ref{ch: gibbs}.

\begin{corollary}
Let  $(\Xi,\Sigma)$ be an $\mathrm{RGG}(\lambda,\kappa)$ such that $\Xi$ is an inhibitory pairwise interaction process with constant activity function~$\beta$, i.e.\ $\lambda(x\mvert \xi) = \beta \prod_{y\in\xi} \varphi(x,y)$ for $\beta\in\R_+$ and $\varphi:\mcx^2\to \R_+$ with $\varphi\leq 1$. Then
\begin{align} \label{eq: discrete approx}
    W_{\G,i}(&P^\lambda\otimes Q^{\kappa},P^{\lambda_{\Lambda}}\otimes Q^{\kappa_{\Lambda}}) \leq r_V + C_iB^*\,\beta^2\, \norm{\varphi - \varphi \circ t}_1 
      + \frac14 C_E\,\beta^2\,\norm{\kappa - \kappa_{\Lambda}\circ t}_1,
\end{align}
where we define $(g \circ t) (x,y):=g(t(x),t(y))$ for any function $g \colon \mcx^2 \to \R_+$. \\
If furthermore $\varphi$ and $\kappa$ are Lipschitz continuous in both components with constants $L_V$ and $L_E$, respectively, and we have $\kappa_{\Lambda}=\kappa$, then
\begin{align} \label{eq: discrete approx if Lipschitz}
    W_{\G,i}(P^u\otimes Q^{\kappa},P^{u_{\Lambda}}\otimes Q^{\kappa}) 
   % &\leq r_V + C_iB^*\,\beta\, \E\abs{\Xi_I}\,\sup_{y\in \mcx }\int_{\mcx}\big\vert\varphi(x,y) -\varphi(t(x),t(y))\big\vert \alpha(\diff x)\\
    %&\quad + \frac14 C_E\,\beta^2\, \int_{\mcx^2} \abs{\kappa(t(x),t(y)) - \kappa(x,y)}\,  \alpha^2(\diff (x,y)) \\
    \leq \Bigl( 1 + \bigl(2 C_iB^*\,L_V + \tfrac12 C_E\,L_E \bigr) \beta^2 \alpha(\mcx)^2 \Bigr) \, r_V.
\end{align}
\end{corollary}
\begin{proof}
Since $\Xi$ is inhibitory with a constant and thus integrable activity function, it fulfils the stability condition~\eqref{eq: stability cond}. Hence, we can apply Theorem~\ref{thm: discretization application} and obtain by the specified form of the conditional intensity
    \begin{align}\label{eq: discretisation proof pip}
    \notag&W_{\G,i}(P^u\otimes Q^{\kappa},P^{u_{\Lambda}}\otimes Q^{\kappa_{\Lambda}}) \\
    %&\leq r_V + C_iB^*\, \E\biggl(\int_{\mcx}\big\vert\lambda(x\mvert \Xi_I) -\lambda(t(x)\vert t(\Xi_I))\big\vert \alpha(\diff x)\biggr)\\
    %&\quad + \frac14 C_E\,\E\biggl(\int_{\mcx}\int_{\mcx} \abs{\kappa(t(x),t(y)) - \kappa(x,y)}\, \Xi_I(\diff y)\,  \lambda(x\vert \Xi_I)\alpha(\diff x)\alpha(\diff y) \biggr)\\
    \notag&\leq r_V + C_iB^*\, \E\bigg(\int_{\mcx}\beta\,\Bigl\vert \prod_{y\in\Xi_I}\varphi(x,y) -\prod_{y\in\Xi_I}\varphi(t(x),t(y))\Bigr\vert \, \alpha(\diff x)\biggr)\\
    &\quad + \frac14 C_E\,\E\biggl(\int_{\mcx^2} \bigabs{\kappa(x,y) - \kappa_{\Lambda}(t(x),t(y))}\,  \beta^2 \! \prod_{z\in \Xi_I+\delta_{y}}\varphi(x,z)\,\prod_{z'\in \Xi_I}\varphi(t(y),t(z'))\,\alpha^2(\diff (x,y))\biggr). 
    \end{align}
Consider the first expectation in \eqref{eq: discretisation proof pip}. For any $\xi_I = \sum_{i=1}^n \delta_{x_i}\in\mfn$ we can expand the absolute value into a telescopic sum and obtain
\begin{align*}
&\bigg\vert \prod_{y\in\xi_I}\varphi(x,y) -\prod_{y\in\xi_I}\varphi(t(x),t(y))\bigg\vert \\
%&= \bigg\vert\prod_{i=1}^n\varphi(x,x_i) -\prod_{i=1}^n\varphi(t(x),t(x_i))\bigg\vert    \\
&= \bigg\vert\sum_{j=1}^n\bigg(\prod_{i=1}^j\varphi(x,x_i)\prod_{i=j+1}^n\varphi(t(x),t(x_i)) -\prod_{i=1}^{j-1}\varphi(x,x_i)\prod_{i=j}^n\varphi(t(x),t(x_i))\bigg)\bigg\vert \\
&= \bigg\vert\sum_{j=1}^n \bigl( \varphi(x,x_j) - \varphi(t(x),t(x_j)) \bigr) \prod_{i=1}^{j-1}\varphi(x,x_i)\prod_{i=j+1}^n\varphi(t(x),t(x_i))\bigg\vert\\
&\leq \sum_{j=1}^n \big\vert\varphi(x,x_j) - \varphi(t(x),t(x_j))\big\vert,
%&\leq \sum_{j=1}^n L_V\big( d_V(x,t(x)) + d_V(x_j,t(x_j))\big) \leq 2\abs{\xi_I}L_V r_V, 
\end{align*}
using $\varphi\leq 1$ for the last inequality. Applying this bound to the first expectation in \eqref{eq: discretisation proof pip} yields 
\begin{align*}
     \E\bigg(\int_{\mcx}\beta&\,\Bigl\vert \prod_{y\in\Xi_I}\varphi(x,y) -\prod_{y\in\Xi_I}\varphi(t(x),t(y))\Bigr\vert \, \alpha(\diff x)\bigg) \\
     &\leq  \E\bigg(\beta\sum_{y\in\Xi_I}\int_{\mcx} \bigl\vert\varphi(x,y) - \varphi(t(x),t(y))\bigr\vert \, \alpha(\diff x)\biggr)\\
     % &= \beta \int_{\mcx} \biggl( \int_{\mcx} \bigl( \varphi(x,y) - \varphi(t(x),t(y)) \bigr) \, \alpha(dx) \: \lambda(t(y)) \biggr) \, \alpha(dy) \\
     &\leq \beta^2 \, \int_{\mcx} \int_{\mcx} \bigl\vert\varphi(x,y) - \varphi(t(x),t(y))\bigr\vert \, \alpha(\diff x) \, \alpha(\diff y),
\end{align*}
where the last inequality follows from the GNZ equation~\eqref{eq:gnz} and $\lambda_I(x \mvert \xi) = \lambda(t(x) \mvert t(\xi)) \leq \beta$.

For the second expectation in \eqref{eq: discretisation proof pip}, we apply $\varphi\leq 1$ to obtain
\begin{align*}
    \E\biggl(\int_{\mcx^2} \bigabs{\kappa(x,y) - \kappa_{\Lambda}(t(x),t(y))}\,  \beta^2\,\prod_{z\in \Xi_I+\delta_{y}}&\varphi(x,z)\,\prod_{z'\in \Xi_I}\varphi(t(y),t(z'))\,\alpha^2(\diff (x,y))\biggr)\\
    %&\leq \int_{\mcx^2} \abs{\kappa_{\Lambda}(t(x),t(y))- \kappa(t(x),t(y))}\beta^2\,\alpha^2(\diff (x,y)) \\
    %&\quad+ \E\biggl(\int_{\mcx}\int_{\mcx}\abs{\kappa(t(x),t(y)) - \kappa(x,y)}\,  \beta^2\,\prod_{z'\in \Xi_I}\varphi(t(y),t(z'))\alpha(\diff x)\alpha(\diff y)\biggr)\\
    &\leq \int_{\mcx^2} \bigabs{\kappa(x,y) - \kappa_{\Lambda}(t(x),t(y))} \, \beta^2\,\alpha^2(\diff (x,y)).
    % &\quad + \int_{\mcx^2} \bigabs{\kappa(t(x),t(y)) - \kappa(x,y)}\,  \beta \lambda(t(y)\mvert  t(\Xi_I))\alpha(\diff x)\alpha(\diff y)\biggr)\\[0.5mm]
    % &\leq \beta^2\,\norm{\kappa_{\Lambda}\circ t - \kappa\circ t}_1 + \E\bigg(\int_{\mcx}\sum_{y\in\Xi_I}\abs{\kappa(t(x),t(y)) - \kappa(x,y)}\,  \beta\,\alpha(\diff x)\biggr)\\
    % &\leq \beta^2\,\norm{\kappa_{\Lambda}\circ t - \kappa\circ t}_1  + \beta\,\E\abs{\Xi_I}\sup_{y\in\mcx}\int_{\mcx}\abs{\kappa(t(x),t(y)) - \kappa(x,y)} \alpha(\diff x),
    %&\leq \beta^2\,\norm{\kappa_{\Lambda}\circ t - \kappa\circ t}_1+ 2 \beta^2\, L_E\, r_V\,\alpha(\mcx)^2.
\end{align*}
This implies \eqref{eq: discrete approx} for general $\varphi \leq 1$, $\kappa$ and $\kappa_{\Lambda}$.

If in addition we have that $\varphi$ and $\kappa_{\Lambda} = \kappa$ are Lipschitz in both components with constants $L_V$ and $L_E$, respectively, we obtain
\begin{align*}
\bigabs{\varphi(x,y) - \varphi(t(x),t(y))}
&\leq \bigabs{\varphi(x,y) - \varphi(t(x),y)} + \bigabs{\varphi(t(x),y) - \varphi(t(x),t(y))}\\
&\leq L_V\big( d_V(x,t(x)) + d_V(y,t(y))\big)\leq 2 \hbit L_V \hbit r_V     
\end{align*}
and similarly
\begin{align*}
    \abs{\kappa(x,y)-\kappa(t(x),t(y))} \leq L_E \big(d_V(x,t(x)) + d_V(y,t(y)) \big) \leq 2 \hbit L_E \hbit r_V
\end{align*}
for all $x,y\in \mcx$. Applying these bounds to the general bound yields the claim.
\end{proof}

%\section{References}
%\renewcommand{\refname}{}
%\vspace{-1cm}
\printbibliography
\end{document}